\documentclass[a4paper,reqno,11pt]{amsart}
\usepackage{amsmath, amsfonts, amssymb, amsthm, amscd}
\usepackage{graphicx}
\usepackage{perpage}
\usepackage{url}
\usepackage{color}
\usepackage{xcolor}
\usepackage{mathrsfs}
\usepackage{stmaryrd}
\SetSymbolFont{stmry}{bold}{U}{stmry}{m}{n}
\usepackage{mathabx}
\usepackage{scalerel} 
\usepackage{tikz,tikz-cd}\usepackage{caption,subcaption}
\usetikzlibrary{arrows,decorations.pathmorphing,backgrounds,positioning,fit,petri}

\usetikzlibrary{decorations.pathreplacing}
\usepackage{subdepth}

\usepackage{dsfont} 

\usepackage[utf8]{inputenc}
\usepackage[T1]{fontenc}
\usepackage{microtype}
\usepackage{mathtools}
\usepackage[a4paper,scale={0.85,0.80},marginratio={1:1},footskip=7mm,headsep=10mm]{geometry}

\usepackage[linktocpage]{hyperref}
\definecolor{dukeblue}{rgb}{0.0, 0.0, 0.61}

\hypersetup{
	colorlinks=true, linkcolor=dukeblue,
	citecolor=dukeblue
}
\usepackage{fancyhdr}
\usepackage{nameref}

\makeatletter
\def\@secnumfont{\bfseries}

\def\section{\@startsection{section}{1}%
  \z@{.7\linespacing\@plus\linespacing}{.5\linespacing}%
  {\normalfont\large\bfseries\centering}}

\def\subsection{\@startsection{subsection}{2}%
  \z@{.5\linespacing\@plus.7\linespacing}{-.5em}%
  {\normalfont\bfseries}}

\def\subsubsection{\@startsection{subsubsection}{3}%
  \z@{.5\linespacing\@plus.7\linespacing}{-.5em}%
  {\normalfont}}

\def\specialsection{\@startsection{section}{1}%
  \z@{\linespacing\@plus\linespacing}{.5\linespacing}%
  {\normalfont\centering\large\bfseries}}
\makeatother

\makeatletter

\renewenvironment{proof}[1][\proofname]{\par
\pushQED{\qed}%
\normalfont \topsep4\p@\@plus4\p@\relax
\trivlist
\item[\hskip\labelsep
\bfseries
#1\@addpunct{.}]\ignorespaces
}{%
\popQED\endtrivlist\@endpefalse
}
\makeatother

\makeatletter
\newcommand \Dotfill {\leavevmode \leaders \hb@xt@ 6pt{\hss .\hss }\hfill \kern \z@}
\makeatother

\makeatletter
\def\@tocline#1#2#3#4#5#6#7{\relax
  \ifnum #1>\c@tocdepth 
  \else
    \par \addpenalty\@secpenalty\addvspace{#2}%
    \begingroup \hyphenpenalty\@M
    \@ifempty{#4}{%
      \@tempdima\csname r@tocindent\number#1\endcsname\relax
    }{%
      \@tempdima#4\relax
    }%
    \parindent\z@ \leftskip#3\relax \advance\leftskip\@tempdima\relax
    \rightskip\@pnumwidth plus4em \parfillskip-\@pnumwidth
    #5\leavevmode\hskip-\@tempdima
      \ifcase #1
       \or\or \hskip 1.65em \or \hskip 3.3em \else \hskip 4.95em \fi%
      #6\nobreak\relax
    \Dotfill
    \hbox to\@pnumwidth{\@tocpagenum{#7}}\par
    \nobreak
    \endgroup
  \fi}
\makeatother

\makeatletter
\def\l@section{\@tocline{1}{0pt}{1pc}{}{}}
\renewcommand{\tocsection}[3]{%
\indentlabel{\@ifnotempty{#2}{\ignorespaces#1 #2.\hskip 0.7em}}#3}
\def\l@subsection{\@tocline{2}{0pt}{1pc}{5pc}{}}

\def\l@subsubsection{\@tocline{3}{0pt}{1pc}{7pc}{}}

\makeatother

\numberwithin{equation}{section}

\tikzset{
    kernel/.style={
        thick
    },
    timeslice/.style={
        thick
    },
    vertex/.style={
        circle,
        fill=black,
        inner sep=1.8pt
    },
    particle/.style={
        font=\small,
        inner sep=1pt
    },
    pairlabel/.style={
        font=\scriptsize,
        inner sep=1pt
    }
}

\newtheoremstyle{mytheorem}{.7\linespacing\@plus.3\linespacing}{.7\linespacing\@plus.3\linespacing}%
     {\itshape}
     {}
     {\bfseries}
     {. }
     {0.3ex}
     {\thmname{{\bfseries #1}}\thmnumber{ {\bfseries #2}}\thmnote{ (#3)}}  

\theoremstyle{mytheorem}

\newtheorem{theorem}{Theorem}[section]
\newtheorem*{theoremA}{Theorem A}
\newtheorem*{theoremB}{Theorem B}
\newtheorem*{theoremC}{Theorem C}
\newtheorem{lemma}[theorem]{Lemma}
\newtheorem{proposition}[theorem]{Proposition}

\newtheorem{remark}[theorem]{Remark}

\newcommand{\bbE}{{\ensuremath{\mathbb E}} }

\newcommand{\bbN}{{\ensuremath{\mathbb N}} }

\newcommand{\bbP}{{\ensuremath{\mathbb P}} }
\newcommand{\bbQ}{{\ensuremath{\mathbb Q}} }
\newcommand{\bbR}{{\ensuremath{\mathbb R}} }

\newcommand{\bbZ}{{\ensuremath{\mathbb Z}} }

\newcommand\bsS{\boldsymbol{S}}

\newcommand{\cA}{{\ensuremath{\mathcal A}} }
\newcommand{\cB}{{\ensuremath{\mathcal B}} }

\newcommand{\cE}{{\ensuremath{\mathcal E}} }

\newcommand{\cG}{{\ensuremath{\mathcal G}} }

\newcommand{\cL}{{\ensuremath{\mathcal L}} }
\newcommand{\cM}{{\ensuremath{\mathcal M}} }
\newcommand{\cN}{{\ensuremath{\mathcal N}} }

\newcommand{\cR}{{\ensuremath{\mathcal R}} }
\newcommand{\cS}{{\ensuremath{\mathcal S}} }

\newcommand{\cU}{{\ensuremath{\mathcal U}} }
\newcommand{\cV}{{\ensuremath{\mathcal V}} }
\newcommand{\cW}{{\ensuremath{\mathcal W}} }

\newcommand{\ga}{\alpha}
\newcommand{\gb}{\beta}

\newcommand{\gd}{\delta}
\newcommand{\gD}{\Delta}

\newcommand{\gl}{\lambda}

\newcommand{\gs}{\sigma}

\newcommand{\go}{\omega}

\DeclareMathSymbol{\leqslant}{\mathalpha}{AMSa}{"36} 
\DeclareMathSymbol{\geqslant}{\mathalpha}{AMSa}{"3E} 
\DeclareMathSymbol{\eset}{\mathalpha}{AMSb}{"3F}     

\DeclareMathOperator*{\union}{\bigcup}       
\DeclareMathOperator*{\inter}{\bigcap}       

\newcommand{\PEfont}{\mathrm}

\DeclareMathOperator{\cov}{\ensuremath{\PEfont Cov}}
\newcommand{\p}{\ensuremath{\PEfont P}}

\newcommand{\E}{\ensuremath{\PEfont  E}}

\DeclareMathOperator{\dist}{dist}

\newcommand{\ind}{\mathds{1}}

\newcommand{\eps}{\varepsilon}
\renewcommand{\epsilon}{\varepsilon}
\renewcommand{\rho}{\varrho}

\newenvironment{myenumerate}{%
\renewcommand{\theenumi}{\arabic{enumi}}%
\renewcommand{\labelenumi}{{\rm(\theenumi)}}%
\begin{list}{\labelenumi}
	{%
	\setlength{\itemsep}{0.4em}%
	\setlength{\topsep}{0.5em}%
	\setlength\leftmargin{2.45em}%
	\setlength\labelwidth{2.05em}%
	\setlength{\labelsep}{0.4em}%
	\usecounter{enumi}%
	}%
	}%
{\end{list}
}

{\end{list}
}

{\end{list}
}

\renewenvironment{enumerate}{
\begin{myenumerate}}%
{\end{myenumerate}}

\newenvironment{myitemize}{%
\begin{list}{$\bullet$}%
 	{%
	\setlength{\itemsep}{0.4em}%
	\setlength{\topsep}{0.5em}%
	\setlength\leftmargin{2.65em}%
	\setlength\labelwidth{2.65em}%
	\setlength{\labelsep}{0.4em}%
	}%
	}%
{\end{list}}

{\end{myitemize}}

\MakePerPage[2]{footnote} 

\date{\today}

\newcommand\dd{\mathrm{d}}

\newcommand\sfE{\mathsf E}

\newcommand\sfL{\mathsf L}

\newcommand\sfP{\mathsf P}

\newcommand\sfR{\mathsf R}

\newcommand\sfT{\mathsf T}
\newcommand\sfU{\mathsf U}

\newcommand\sff{\mathsf f}
\newcommand\sfg{\mathsf g}
\newcommand\sfi{\mathsf i}

\newcommand\sfj{\mathsf j}
\newcommand\sfk{\mathsf k}

\newcommand\sfp{\mathsf p}

\newcommand\bx{\boldsymbol{x}}
\newcommand\by{\boldsymbol{y}}
\newcommand\bz{\boldsymbol{z}}

\newcommand\bw{\boldsymbol{w}}

\newcommand\bS{\boldsymbol{S}}

\newcommand{\ty}{\tilde{y}}

\newcommand{\tB}{\tilde{B}}

\newcommand{\tE}{\tilde{E}}

\newcommand{\tS}{\tilde{S}}

\newcommand{\tW}{\tilde{W}}

\newcommand\bI{\boldsymbol{I}}

\newcommand\hy{\widehat{y}}

\definecolor{cadmiumgreen}{rgb}{0.0, 0.42, 0.24}
\definecolor{red(munsell)}{rgb}{0.95, 0.0, 0.24}

\usepackage{mathtools}

\newcommand{\f}{\frac}

\newcommand{\hb}{\hat{\beta}}
\newcommand{\bhb}{\boldsymbol{\hb}}
\newcommand{\hbm}{\hat{\beta}_{\text{max}}}
\newcommand{\hbn}{\hat{\beta}_{\text{min}}}

\newcommand{\sfe}{\mathsf{e}}

\usepackage[nameinlink,capitalize,noabbrev]{cleveref}
\usepackage{enumitem}
\usepackage{booktabs}
\usepackage{tabularx}

\crefname{theorem}{theorem}{theorems}
\crefname{proposition}{proposition}{propositions}
\crefname{lemma}{lemma}{lemmas}
\crefname{corollary}{corollary}{corollaries}
\crefname{assumption}{assumption}{assumptions}
\crefname{remark}{remark}{remarks}

\providecommand{\1}{\mathbf{1}}

\title{Independence threshold for collision times of many planar random walks}
\author{Ziyang Liu}
\date{July 2026}
\begin{document}
\begin{abstract}
We study collision times of many independent simple random walks on $\bbZ^2$ through the joint moment generating function of their pairwise collision local times. For a fixed number of walks, these collision times are known to be asymptotically independent after a suitable logarithmic normalisation. We investigate the extent to which this independence persists when the number of walks grows. For $N$ being the walk length, our results identify a threshold $\asymp (\log N)^{\f{1}{3}}$, on which the transition from independence to dependence happens. The proofs combine chaos expansion techniques and a correlation inequality, which is the result of a local limit theorem. 
\end{abstract}
\maketitle

\section{Introduction}

The main goal of this paper is to understand when the collision times of many
independent two-dimensional random walks behave asymptotically as if the
pairwise collisions were independent, and when this picture breaks down.
In the 1960s, Erd\H{o}s and Taylor identified the correct normalisation for the collision local time of two planar random walks~\cite{ErdosTaylor1960}. Only recently, Lygkonis and Zygouras established the asymptotic independence of the collision local times of finitely many planar random walks by computing the limiting moment generating function
of the total pairwise collision time~\cite{LygkonisZygouras2024Multivariate}. The present paper studies the same object while allowing the number $h=h_N$ of random walks to depend on the walk length. We focus on the
transition from a regime of asymptotically independent pairwise collisions to one in which correlations between pairwise collisions emerge.

Let $\bsS_n = (S^1_n,...,S^h_n)$ be a vector of $h$ independent simple symmetric random walks on $\bbZ^2$ all started from the origin. We define 
$$
\Pi_h
:=
\Big\{ (\sfi,\sfj): 1\le \sfi<\sfj\le h \Big\}.
$$ 
For each pair $1\le \sfi<\sfj\le h$, we let
\begin{align*}
    \sfL^{\sfi,\sfj}_N = \sum_{1\le n\le N} \ind_{S^{\sfi}_n = S^{\sfj}_n}
\end{align*}
be the pairwise local time of the pair $\{\sfi,\sfj\}$. We also let $\sfE^{\otimes h}$ be the expectation with respect to $(\bsS_n)_{n\in \bbN}$. 

\begin{theoremA}[Erd\H{o}s--Taylor theorem \cite{ErdosTaylor1960}]
    For two independent simple symmetric random walks on $\bbZ^2$,
    \begin{align*}
        \f{\pi}{\log N}\sum_{n=1}^N \ind_{S^1_n=S^2_n}
        \xrightarrow[N\to\infty]{(d)}
        \mathrm{Exp}(1).
    \end{align*}
    In particular, for every $0<\hb<1$,
    \begin{align*}
        \sfE^{\otimes 2}\Big[\exp\Big( \f{\pi\hb}{\log N} 
        \sfL^{1,2}_N \Big)\Big]
        \xrightarrow[N\to\infty]{}
        \f{1}{1-\hb}.
    \end{align*}
\end{theoremA}

Due to Theorem A, we normalise each pairwise collision time by $\pi\hb_{\sfi,\sfj}/\log N$ and consider the joint moment generating function
\begin{align}\label{def: M}
    \cM^h_{N,\bhb}
    :=
    \sfE^{\otimes h} \Big[ e^{\sum_{1\le \sfi<\sfj\le h} \f{\pi \hb_{\sfi,\sfj} }{\log N} \sfL^{\sfi,\sfj}_N } \Big].
\end{align}
\begin{theoremB}[Lygkonis--Zygouras \cite{LygkonisZygouras2024Multivariate}]
    Fix $h\in\bbN$ and assume $0<\hb_{\sfi,\sfj}<1$ for all $1\le \sfi<\sfj\le h$. Then:
    \begin{align*}
        \cM^h_{N,\bhb}
        \xrightarrow[N\to\infty]{}
        \prod_{1\le \sfi<\sfj\le h} \f{1}{1-\hb_{\sfi,\sfj}}.
    \end{align*}
Hence the normalised pairwise collision times are asymptotically independent.
\end{theoremB}

There is also a related continuum literature on logarithmic scaling laws for planar Brownian motion and Brownian windings, including works of Pitman and Yor \cite{PitmanYor1986}, Yor \cite{Yor1991}, and Knight \cite{Knight1993,Knight1994}. These results provide Brownian analogues in which logarithmic renormalisation leads to limiting independent random variables, although the methods and the objects are different from the discrete collision local times considered here.

Cosco and Zeitouni \cite{CoscoZeitouni2023,CoscoZeitouni2024} studied the same type of moment problem from the perspective of partition functions of Gaussian directed polymers in the subcritical regime; we will make connection to those later in this section. Their result, translated into the language of exponential moments of collision times, is the following:

\begin{theoremC}[Cosco--Zeitouni \cite{CoscoZeitouni2023,CoscoZeitouni2024}]
    In the homogeneous case $\hb_{\sfi,\sfj}\equiv \hb\in(0,\bar{\gb})$ for some $\bar{\gb}<\f{1}{72}<1$ and for $h=h_N\le C\sqrt{\log N}$ with some $C>0$:
    \begin{align*}
        \cM^h_{N,\hb}
        =
        \bigg(\f{1}{1-\hb}\bigg)^{\binom{h}{2}\big(1+o(1)\big)}.
    \end{align*}
\end{theoremC}
Thus, when $h=h_N$ grows at a controlled rate, the factorised form remains valid at asymptotic level. However, their exponent error $o(1) = O\big(\f{1}{\sqrt{h}}\big)$. Consequently the error factor 
$(1-\hb)^{-\binom{h}{2} \cdot o(1)}$ need not go to $1$. 
 This prevents one from deducing asymptotic independence of the full family of pairwise collision times from Theorem~C. A main contribution of the present paper is to identify the precise threshold scale 
 \begin{align*}
  h = h_N \asymp (\log N)^{\f{1}{3}}
 \end{align*}
 at which the behaviour transitions from asymptotic independence to dependence. 

\begin{theorem}[Independence regime]\label{thm:ind}
  Fix constants 
  $$
  0<\hbn\le \hbm<1,
  $$
  independent of $N$, where $\hbm<\f{\bar{\gb}}{4}$ for the same $\bar{\gb}$ appearing in Theorem~C. For every $N$, let
  \begin{align*}
    h = h_N = o\big( (\log N)^{\f{1}{3}} \big),
  \end{align*}
  and consider $\bhb = (\hb_{\sfi,\sfj})_{1\le \sfi<\sfj\le h}$, such that $\hbn\le \hb_{\sfi,\sfj}\le \hbm$ for all $1\le \sfi<\sfj\le h$. We have that
  \begin{align}
    \lim_{N\rightarrow\infty} 
    \Big( \cM^h_{N,\bhb} \prod_{1\le \sfi<\sfj\le h} (1-\hb_{\sfi,\sfj}) \Big)
    = 1.
  \end{align}
\end{theorem}

\begin{theorem}[Dependence regime]\label{thm:dep}
  Fix constants 
  $$
  0<\hbn\le \hbm\le \f{1}{4}\bar \gb<1,
  $$
  independent of $N$ and suppose $\bhb = (\hb_{\sfi,\sfj})_{1\le \sfi<\sfj\le h}$, such that $\hbn\le \hb_{\sfi,\sfj}\le \hbm$ for all $1\le \sfi<\sfj\le h$. Then for any constant $C>0$, such that if 
  \begin{align*}
    C(\log N)^{\f{1}{3}}
    \le h 
    =o \Big( \sqrt{\f{\log N}{\log\log N}}\Big),
  \end{align*}
  then we have
  \begin{align*}
    \liminf_{N\rightarrow\infty} 
    \Big( \cM^h_{N,\bhb} \prod_{1\le \sfi<\sfj\le h} (1-\hb_{\sfi,\sfj}) \Big)
    >1.
  \end{align*}
\end{theorem}

Thus, for pairwise collision times, asymptotic independence breaks down at the scale $(\log N)^{\f{1}{3}}$. 
For the rest of the paper, we will denote a pair of indices by $\sfe = \{\sfi,\sfj\}$, and we define $\Pi_h$ to be the set of all pairs:
\begin{align*}
  \Pi_h:= \Big\{ \sfe = \{\sfi,\sfj\}: 1\le \sfi<\sfj\le h \Big\}.
\end{align*}
We also denote that
$
  L = \log N
$. Theorem \ref{thm:ind} and \ref{thm:dep} are direct consequences of the following two results:

\begin{theorem}[Upper bound]\label{ub:thm-upper}
Fix constants
\[
 0<\hbn\le \hbm<\f{1}{4}\bar\gb<1,
\]
independent of $N$, where $\bar\gb$ is the constant appearing in
Theorem~C.  For every $N$, let $h=h_N$ and let
$\bhb=(\hb_{\sfe})_{\sfe\in\Pi_h}$ satisfy
\[
 \hbn\le \hb_{\sfe}\le \hbm,
 \qquad \sfe\in\Pi_h.
\]
Suppose further that $h\rightarrow\infty$ and $h=o(\sqrt L)$ as $N\rightarrow \infty$.  Then there exists
$c,C>0$, depending only on $\hbn$ and $\hbm$, such that the
following assertion holds.
\begin{equation}\label{ub:eq-upper-two-scale}
 \cM^h_{N,\bhb}
 \le \Big(\prod_{\sfe\in \Pi_h} \f{1}{1-\hb_{\sfe}} \Big)
 e^{ C \big(\frac{h^3}{L}+\frac{h^6}{L^2}\big) }
 + e^{-ch^2}.
\end{equation}

Consequently, together with Theorem B, for $h = o\big( L^{1/3} \big)$:
\begin{align*}
  \limsup_{N\rightarrow\infty}
  \Big( \cM^h_{N,\bhb} \times \prod_{\sfe\in \Pi_h} (1-\hb_{\sfe}) \Big)
  \le 1.
\end{align*}
\end{theorem}

\begin{theorem}[Lower bounds]\label{lb:thm-main}
Retain the parameter assumptions of Theorem~\ref{ub:thm-upper}.  There exist
constants $c_0,C>0$, depending only on $\hbn$ and $\hbm$, such that,
whenever
\begin{equation}\label{lb:eq:intro-working-range}
  h \le  \sqrt{\f{c_0 L}{\log L}},
\end{equation}
the following assertions hold uniformly over
$\bhb=(\hb_{\sfe})_{\sfe\in\Pi_h}$.
\begin{enumerate}[label=\textup{(\roman*)}]
\item There exists $r_{N,h}\ge0$ such that
\begin{equation}\label{lb:eq:quant-lower}
 \cM^h_{N,\bhb}
 \ge
 \Big(\prod_{\sfe\in \Pi_h} \f{1}{1-\hb_\sfe} \Big)
 e^{-r_{N,h}},
 \qquad
 r_{N,h}\le C \frac{h^2\log L}{L}.
\end{equation}
Consequently, if $h^2\log L/L\to0$, then 
\begin{equation*}
 \liminf_{N\to\infty}
 \Big(
   \cM^h_{N,\bhb}\prod_{\sfe\in\Pi_h}(1-\hb_{\sfe})
 \Big)
 \ge1.
\end{equation*}

\item For any $C'>0$, there exists $\gd_{C'}>0$
such that, if in addition $h = o\big(\sqrt{L/\log L}\big)$ and
\begin{equation*}
 h\ge C' L^{1/3},
\end{equation*}
then, for all sufficiently large $N$,
\begin{equation}\label{lb:eq:strict-log}
 \cM^h_{N,\bhb}
 \ge
 e^{\gd_{C'}} \prod_{\sfe\in \Pi_h} \f{1}{1-\hb_{\sfe}}.
\end{equation}
\end{enumerate}
\end{theorem}

The threshold $h\asymp L^{1/3}$ is governed by the
correction term $h^3/L$. Intuitively, this correction term is due to two
pairwise collisions occurring at different times and involving a common
particle, such as $\{\sfi,\sfj\}$ and $\{\sfj,\sfk\}$. For more details, see remark \ref{ub:intuition} and Figure \ref{fig:three-particle-collision-diagrams}.

\subsection{Proof strategy}
The proof of Theorem \ref{ub:thm-upper} is based on analysis on the expansion of the moment generating function. For the first few steps, we follows a scheme similar to \cite{CoscoZeitouni2023}, by eliminating the contributions from collisions with more than 3 walks, and processing basic reductions to the expansion. This is section \ref{ub:sec-summary} and the beginning of section \ref{ub:sec-reduced}. At this stage, to deal with a recursive summations, we employ a probablistic argument by bounding the recursive kernel by a cumulative distribution function, which is the rest of section \ref{ub:sec-reduced}.  

For the proof of Theorem \ref{lb:thm-main}, for collision events of a single pair, we consider them in a renewal measure. Then Theorem \ref{lb:thm-main} is investigating the correlation structures of the renewal measures corresponding to collision events of different pairs, which is section \ref{lb:sec:change-measure}. In section \ref{lb:sec:renewal}, by restricting to a good set of typical probability under the renewal measures, we are able to apply a local limit theorem to transfer transition probabilities to Gaussian kernels. On this level, the Hadamard-Fischer inequality directly yields \eqref{lb:eq:quant-lower}. For \eqref{lb:eq:strict-log}, we first find an event, on which the correlation effect is nontrivial, and then show that such an event start to have non-trivial probability once we have $h>C(\log N)^{\f{1}{3}}$, for some $C>0$.

\subsection{Connection to partition function of directed polymers}
Exponential moments of collision times of planar random walks are connected to moments of 2D polymer partition function. Let $(S_n)_{n\in\bbN}$ be a simple symmetric random walk on $\bbZ^2$ starting from the origin, and let $(\go_{n,x})_{(n,x)\in\bbN\times\bbZ^2}$ be an i.i.d.\! set of standard Gaussian variables, independent of the walk. Denote expectation over the walk by $\sfE$ and over the environment by $\bbE$. For $\gl(\gb):=\log\bbE[e^{\gb\go_{0,0}}]$, the partition function is
\begin{align*}
    Z_{N,\gb}
    :=
    \sfE\Big[
    e^{\sum_{n=1}^N \big( \gb \go_{n,S_n} - \gl(\gb)\big) }
    \Big].
\end{align*}
In the intermediate-disorder scaling
$
    \gb_{N;\sfi}=\hb_\sfi\,\,\sqrt{\f{\pi}{\log N}}
$, with 
$
1\le \sfi\le h,
$
the mixed moments of these partition functions are exactly of the form \eqref{def: M}. If $\hb_{\sfi,\sfj}=\hb_\sfi\hb_\sfj$, then
\begin{align}\label{eq: Z to M}
    \bbE \Big[ \prod_{\sfi=1}^h Z_{N,\gb_{N;\sfi}} \Big]
    =
    \cM_{N,\bhb}^h.
\end{align}
Indeed, by Gaussianity of $(\go_{n,x})_{n,x}$ and recalling $\sfE^{\otimes h}$ as the joint expectation of $h$ independent walks $S^1,\ldots,S^h$, we have:
\begin{align*}
    \bbE \Big[ \prod_{\sfi=1}^h Z_{N,\gb_{N;\sfi}} \Big]
    & =
    \bbE \bigg[\sfE^{\otimes h} \Big[
    e^{ \sum_{1\le \sfi\le h}
    \sum_{n=1}^N
    \big( \gb_{N;\sfi} \go_{n,S^\sfi_n}
    -\gl(\gb_{N;\sfi})\big)}
    \Big]\bigg] \\
    & =
    \sfE^{\otimes h}\bigg[
    e^{ \sum_{1\le \sfi<\sfj\le h} \sum_{n=1}^N
    \gb_{N;\sfi}\gb_{N;\sfj}
    \ind_{S^\sfi_n = S^\sfj_n} }
    \bigg],
\end{align*}
which gives \eqref{eq: Z to M}. For more information on directed polymers in random environments, see the reviews of Comets \cite{Comets2017} and Zygouras \cite{Zygouras2024Review}.

\subsection{Declaration on the use of AI}
The observation that part of the recursive kernel can be bounded by a cumulative distribution function \eqref{eq:ub-K-H} is done with the help of ChatGPT 5.5 Pro, after prompting ChatGPT to construct a simplex structure inside the recursive integral.  All remaining proofs were developed and written solely by the author.

\section{Genuine multibody collisions}\label{ub:sec-summary}

In this section we   separate the terms in the collision expansion that select a
single pair at every selected time from those that select at least two distinct
pairs at one selected time.  For
 $\bhb=(\hb_{\sfe})_{\sfe\in\Pi_h}$,   set
\[
 \gs_N^{\sfe}:=e^{\pi\hb^{\sfe}/L}-1,
 \qquad L:=\log N.
\] 
  Then
\begin{align}\label{def: collision moment}
\cM^h_{N,\bhb}
&=
\sfE^{\otimes h}\bigg[
\prod_{n=1}^N
\prod_{\sfe=\{\sfi,\sfj\}\in\Pi_h}
\Big(1+\gs_N^{\sfe}
\ind_{\{S_n^\sfi=S_n^\sfj\}}\Big)
\bigg]\notag \\
& =
  1
  +
  \sum_{m\ge 1}
  \sum_{\substack{ \emptyset \neq \pi_1,\cdots \pi_m \subseteq \Pi_h \\ 1\le  n_1< ...< n_m\le N   }}
  \hspace{-.5cm}
  \sfE^{\otimes h}\Big[\prod_{r=1}^m \prod_{\sfe = \{\sfi,\sfj\}\in \pi_r} \gs^{\sfe}_N \ind_{\{S^{\sfi}_{n_r} = S^{\sfj}_{n_r}\}}
  \Big],
\end{align} 
Notice that here at each $n_r$, more than two walks are allowed to collide ($|\pi_r|\ge 2$). We separate the terms that contains contributions that due to collisions of a single pairs at each time:
\begin{align}\label{def: double collision contribution}
\cM^{\mathrm{pair},h}_{N,\bhb}
&:=1+
\sum_{m\ge 1}
\sum_{\substack{\sfe_1,\cdots \sfe_m\in \Pi_h \\ 1\le n_1<\cdots <n_m\le N }}
\sum_{\substack{\bw_1,\ldots,\bw_m\in\bbZ^{2h}\\
\bw_r\sim\sfe_r,\ 1\le r\le m}} 
\prod_{r=1}^m \gs_N^{\sfe_r} \,\,
 p_{n_r-n_{r-1}}(\bw_r-\bw_{r-1}).
\end{align} 
Here we take $\bw_0 = 0$ and $p_n(\bx),\bx\in \bbZ^{2\times h}$ is the transition probability of $h$ copies of independent planar simple symmetric random walk, which comes from taking expectation over the indicator functions on the collision events. For $\bx = (x^1,...,x^h)\in \bbZ^{2\times h}$ and $\sfe = \{\sfi,\sfj\}\in \Pi_h$, we write $\bx\sim \sfe$ if and only if 
$
  x^{\sfi} = x^{\sfj}
$. 
For an arbitrary initial configuration $\bz\in\bbZ^{2\times h}$, we use
$\cM^{\mathrm{pair},h}_{N,\bhb}(\bz)$ for the analogous expression with
$\bw_0=\bz$.

We also define
\begin{align}
  \cM^{\mathrm{mlti},h}_{N,\bhb}
  :=
  \cM^{h}_{N,\bhb}
  -
  \cM^{\mathrm{pair},h}_{N,\bhb}.
\end{align}
 
The following reduction is adapted from \cite{CoscoZeitouni2023}. It shows that $\cM^{\mathrm{mlti},h}_{N,\bhb}$, the contribution from genuine multibody collisions, is negligible when comparing to $\cM^{\mathrm{pair},h}_{N,\bhb}$.
Its proof is given in Appendix~\ref{ub:app-auxiliary}.

\begin{lemma}
\label{ub:lem-multibody-domination}
Assume $h=o(\sqrt L)$  and   $\hbm\le\f{1}{4}\bar\gb$.  There is a constant
$C>0$ and 
\begin{align*}
 \delta_{N,h}\le C \frac{h^4}{L^2}=o(1)
\end{align*}
such that for all sufficiently large $N$,
\begin{equation}\label{ub:eq-multibody-domination}
 \cM^h_{N,\bhb}
 \le
 \frac{1}{1-\delta_{N,h}}
 \sup_{\bz\in\bbZ^{2h}}
 \cM^{\mathrm{pair},h}_{N,\bhb}(\bz).
\end{equation} 
 \end{lemma}

All estimates for the pair contribution in the next section are uniform in
the initial configuration.  We therefore write the formulas for walks
started from the origin, as in \eqref{def: double collision contribution} .

\section{An upper bound for independence}\label{ub:sec-reduced}

In this section we  prove Theorem~\ref{ub:thm-upper}.  In view of
Lemma~\ref{ub:lem-multibody-domination}, it is enough to obtain the stated
bound uniformly for the pair contribution. 
\begin{lemma}\label{ub:upper-bound}
  Let $\hbn,\hbm$ and $\bhb$ be as in Theorem \ref{ub:thm-upper}. Suppose that $h=o(\sqrt L)$.  Then there exists
$c,C>0$, depending only on $\hbn$ and $\hbm$, such that the following assertion holds.
\begin{equation}\label{ub:eq-pair-upper-two-scale}
 \cM^{\mathrm{pair},h}_{N,\bhb}
 \le 
 \Big(\prod_{\sfe\in \Pi_h} \f{1}{1-\hb_{\sfe}} \Big)
 e^{ C \big(\frac{h^3}{L}+\frac{h^6}{L^2}\big) }
 + e^{-ch^2}.
\end{equation}
\end{lemma}

The rest of this section is devoted to the proof of Lemma \ref{ub:upper-bound}. We give a roadmap of the proof here. First, in subsection \ref{Preliminary reduction}, we do a preliminary reduction, similar to the scheme in \cite{CoscoZeitouni2023}. In subsection \ref{Upper cells}, we upper bound the discrete summation by an integral. There we make a slight modification to the recursive kernel, to avoid boundary terms when comparing sums to integrals. In subsection \ref{ub:new-sec-log-ladder}, we bound the recursive modified kernel by a cumulative distribution function, hence transforming the recursive integral into a probability over a shifted simplex, which is the essential step of the proof. In subsections \ref{ub:new-sec-shifted-simplex}, \ref{ub:new-sec-logistic-walk} and \ref{ub:new-sec-exact-reference}, we are estimating the error from eliminating the shifts from the simplex. In subsection \ref{ub:new-sec-sum-m}, we close the proof of Lemma \ref{ub:upper-bound}. 

\subsection{Preliminary reduction}
\label{Preliminary reduction}

We first group consecutive  collisions of the same pair.  For
  $\sfe\in\Pi_h$, set
\begin{align}
 U_N^{\sfe}(0,x)&:=\ind_{\{x=0\}},
 \quad \quad U_N^{\sfe}(0) :=1,\label{ub:eq-U-zero}\\
 U_N^{\sfe}(n,x)
 &:={}
 \sum_{k\ge0}(\gs_N^{\sfe})^{k+1}
 \sum_{\substack{0=n_0<n_1<\cdots<n_k<n=n_{k+1}\\
 x_0=0,\ x_1,\ldots,x_k\in\bbZ^2,\ x_{k+1}=x}}
 \prod_{r=1}^{k+1}
 p_{n_r-n_{r-1}}(x_r-x_{r-1})^2,
 \qquad n\ge1,\label{ub:eq-U-def}\\
 U_N^{\sfe}(n)&:=\sum_{x\in\bbZ^2}U_N^{\sfe}(n,x).
\end{align} 
    Let
$
 \Pi_h^{\otimes m}
 :=\{(\sfe_1,\ldots,\sfe_m)\in\Pi_h^m:
       \sfe_r\neq\sfe_{r+1},\ 1\le r<m\}.
$ 
  For $\sfe_r=\{\sfi_r,\sfj_r\}$ and a particle involved in the $r$-th collision $\sfi_r$, define the
predecessor block index
\begin{equation*}
 \sfp(\sfi_r):=
 \max\bigl(\{0\}\cup\{1\le s<r:\sfi_r\in\sfe_s\}\bigr),
\end{equation*} 
and set $b_0:=0$ and $y_0:=0$. By contracting consecutive pair collisions of the same pair, and applying the Chapman--Kolmogorov property, we obtain
\begin{align}\label{eq: ub-after ck}
  \cM^{\mathrm{pair},h}_{N,\bhb}
  &=1+
  \sum_{m\ge1}
  \sum_{\substack{\bx,\by\in\bbZ^{2m},\
                   \vec\sfe=(\sfe_1,\ldots,\sfe_m)\in\Pi_h^{\otimes m}\\
       1\le a_1\le b_1<a_2\le b_2<\cdots<a_m\le b_m\le N}}
  p_{a_1}(x_1)^2
  \prod_{r=1}^m
  \gs_N^{\sfe_r} \, U_N^{\sfe_r}(b_r-a_r,y_r-x_r)\notag\\
  &\quad\times
  \prod_{r=1}^{m-1}
  p_{a_{r+1}-b_{\sfp(\sfi_{r+1})}}
   \bigl(x_{r+1}-y_{\sfp(\sfi_{r+1})}\bigr)
  p_{a_{r+1}-b_{\sfp(\sfj_{r+1})}}
   \bigl(x_{r+1}-y_{\sfp(\sfj_{r+1})}\bigr).
\end{align}
For a nonzero initial configuration, only the two occurrences of $y_0$ in
the first use of a particle are replaced by its initial position.  The bounds
below use a supremum of the corresponding heat kernel and are therefore
uniform in that initial configuration. For a graphical representation, see Figure \ref{fig:U-contraction}.

\begin{figure}
    \centering
    \begin{tikzpicture}[scale=0.4]
        \newcommand{\drawverticals}{%
            \foreach \x in {0,2,6,8,12,14,16,20}{
                \draw[thick] (\x,-6) -- (\x,6);
            }
        }
        \newcommand{\drawtimelabels}{%
            \node at (0.5,-6) {$0$};
            \node at (2.5,-6) {$a_1$};
            \node at (6.5,-6) {$b_1$};
            \node at (8.5,-6) {$a_2$};
            \node at (12.5,-6) {$b_2$};
            \node at (14.5,-6) {$a_3$};
            \node at (16.5,-6) {$b_3$};
            \node at (20.5,-6) {$N$};
        }
        \newcommand{\drawpointlabels}{%
            \node at (2.7,-0.7) {$x_1$};
            \node at (6.7,-0.8) {$y_1$};
            \node at (8.7,1.7) {$x_2$};
            \node at (12.7,2) {$y_2$};
            \node at (14.7,-2) {$x_3$};
            \node at (16.7,-1) {$y_3$};
        }
        \newcommand{\drawsmoothpaths}{%
            \draw[thick] (0,0)
                to[out=82,in=195] (2,3.2)
                to[out=8,in=172] (6,3.1)
                to[out=-8,in=150] (8,2.5);
            \draw[thick] (12,2.5)
                to[out=25,in=190] (14,3.05)
                to[out=3,in=177] (15,3.0)
                to[out=-3,in=177] (16,2.95)
                to[out=-5,in=180] (20,2.6);

            \draw[thick] (0,0) to[out=36,in=144] (2,0);
            \draw[thick] (6,0) to[out=66,in=210] (8,2.5);
            \draw[thick] (12,2.5) to[out=-76,in=100] (14,-3);
            \draw[thick] (16,-2) to[out=28,in=180] (20,-1);

            \draw[thick] (0,0) to[out=-36,in=-144] (2,0);
            \draw[thick] (6,0)
                to[out=-28,in=170] (8,-0.55)
                to[out=-7,in=173] (10,-0.8)
                to[out=-7,in=173] (12,-1.05)
                to[out=18,in=195] (13.02,-0.76);
            \draw[thick] (13.52,-0.56)
                to[out=15,in=195] (14,-0.35)
                to[out=20,in=200] (15,0)
                to[out=20,in=200] (16,0.35)
                to[out=12,in=188] (20,1);

            \draw[thick] (0,0)
                to[out=-76,in=176] (2,-3)
                to[out=-8,in=172] (6,-3.4)
                to[out=8,in=188] (8,-3.2)
                to[out=12,in=190] (10,-2.75)
                to[out=12,in=192] (12,-2.25)
                to[out=-14,in=150] (14,-3);
            \draw[thick] (16,-2) to[out=-28,in=180] (20,-3);
        }
        \newcommand{\drawbaseblackdots}{%
            \foreach \p in {(0,0),
                (2,3.2),(2,-3),
                (6,3.1),(6,-3.4),
                (8,-0.55),(8,-3.2),
                (12,-1.05),(12,-2.25),
                (14,3.05),(14,-0.35),
                (16,2.95),(16,0.35),
                (20,2.6),(20,1),(20,-1),(20,-3)}{
                \fill[black] \p circle[radius=0.15];
            }
        }
        \newcommand{\drawupperextrablackdots}{%
            \foreach \p in {(10,-0.8),(10,-2.75),(15,3.0),(15,0)}{
                \fill[black] \p circle[radius=0.15];
            }
        }
        \newcommand{\drawlowerreddots}{%
            \foreach \p in {(2,0),(6,0),(8,2.5),(12,2.5),(14,-3),(16,-2)}{
                \fill[red] \p circle[radius=0.18];
            }
        }
        \newcommand{\drawupperreddots}{%
            \foreach \p in {(2,0),(6,0),(8,2.5),(10,2.5),(12,2.5),
                              (14,-3),(15,-2.5),(16,-2)}{
                \fill[red] \p circle[radius=0.18];
            }
        }

        \begin{scope}[shift={(0,17)}]
            \drawverticals
            \foreach \x in {10,15}{
                \draw[thick] (\x,-6) -- (\x,6);
            }
            \drawsmoothpaths

            \draw[thick] (2,0) to[out=38,in=142] (6,0);
            \draw[thick] (2,0) to[out=-38,in=-142] (6,0);

            \draw[thick] (8,2.5) to[out=38,in=142] (10,2.5)
                                   to[out=38,in=142] (12,2.5);
            \draw[thick] (8,2.5) to[out=-38,in=-142] (10,2.5)
                                   to[out=-38,in=-142] (12,2.5);

            \draw[thick] (14,-3) to[out=38,in=142] (15,-2.5)
                                 to[out=38,in=142] (16,-2);
            \draw[thick] (14,-3) to[out=-38,in=-142] (15,-2.5)
                                 to[out=-38,in=-142] (16,-2);

            \drawbaseblackdots
            \drawupperextrablackdots
            \drawupperreddots
            \drawtimelabels
            \drawpointlabels
            \node at (4,1.45) {\tiny $\{2,3\}$};
            \node at (10,3.95) {\tiny $\{1,2\}$};
            \node at (15,-4.15) {\tiny $\{1,4\}$};
            \node[anchor=east] at (-1.15,5.3) {\textbf{(a)}};
        \end{scope}

        \draw[-{Latex[length=3mm,width=2.2mm]},very thick] (9.3,10.4) -- (9.3,6.6);
        \node[anchor=west,align=left] at (10.2,8.5)
              {contract each consecutive run\\[-1mm]
               with the corresponding $U_{N}^{\sfe}$};

        \begin{scope}
            \drawverticals
            \drawsmoothpaths
            \draw[thick,decorate,
                  decoration={snake,amplitude=.4mm,segment length=2mm}] (2,0) -- (6,0);
            \draw[thick,decorate,
                  decoration={snake,amplitude=.4mm,segment length=2mm}] (8,2.5) -- (12,2.5);
            \draw[thick,decorate,
                  decoration={snake,amplitude=.4mm,segment length=2mm}] (14,-3) -- (16,-2);

            \drawbaseblackdots
            \drawlowerreddots
            \drawtimelabels
            \drawpointlabels
            \node at (4,1) {\tiny $\{2,3\}$};
            \node at (10,3.3) {\tiny $\{1,2\}$};
            \node at (15,-4) {\tiny $\{1,4\}$};
            \node[anchor=east] at (-1.15,5.3) {\textbf{(b)}};
        \end{scope}
    \end{tikzpicture}
    \caption{Red dots denote
    the collision weights $\f{\hb_{\sfe}}{L}$ of the corresponding pair $\sfe$, and the smooth curves denote transition probabilities of $2d$ random walk.  Contracting each run with the corresponding
    function $U_{N}^{\sfe}$ gives the lower diagram. We also apply the Chapman-Kolmogorov property on each black dot to extend the smooth curves.}
    \label{fig:U-contraction}
\end{figure}
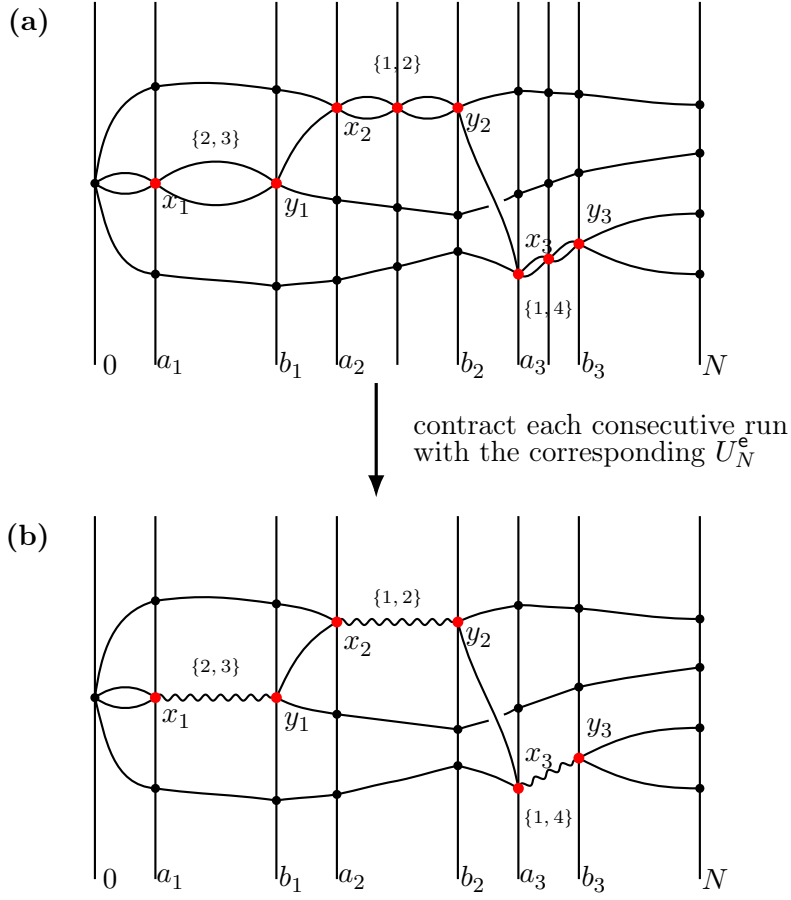

\begin{remark}\label{ub:intuition}
  We give an intuitive explanation to the threshold at $h\asymp (\log N)^{\f{1}{3}}$. 
  The major correction term $\exp\big(\f{h^3}{L}+\f{h^6}{L^2}\big)$ in \eqref{ub:eq-upper-two-scale} arises from two pairwise collisions occurring at different times and involving a common
particle, such as $\{\sfi,\sfj\}$ and $\{\sfj,\sfk\}$. We illustrate this with an overly simplified example. For some fixed $\hb\in (0,1)$, consider the following sums:
\begin{align*}
  A_{\mathrm{cor}}:= \big(\f{\pi\hb}{\log N}\big)^2 \hspace{-.3cm}
  \sum_{\substack{1\le n+m\le N, n,m\ge 1\\ x,y\in \bbZ^2}} \hspace{-.5cm} p_{n}(x)^2 p_m(y-x) p_{n+m}(y),
  \quad  \quad
  A_{\mathrm{ind}}:= \big(\f{\pi\hb}{\log N}\big)^2 \hspace{-.3cm}
  \sum_{\substack{1\le n+m\le N, n,m\ge 1\\ x,y\in \bbZ^2}} \hspace{-.5cm} p_{n}(x)^2 p_{n+m}(y)^2.
\end{align*}
Here $A_{\mathrm{cor}}$ is the simplest term involving $3$ walks in \eqref{eq: ub-after ck} assuming all $\hb_\sfe\equiv \hb$ for some $\hb\in (0,1)$, whereas $A_{\mathrm{ind}}$ is the corresponding term obtained under the hypothetical assumption that the local times are independent. They have graphical representation Figure \ref{fig:three-particle-collision-diagrams}. We wish to show that $A_{\mathrm{cor}}$ can be replaced by $A_{\mathrm{ind}}$. By completing the sum, one could compare the two expression in the following way:
\begin{align*}
  \f{A_{\mathrm{cor}}}{A_{\mathrm{ind}}}
  =
  1+O\big(\f{1}{\log N}\big).
\end{align*}
Consequently, the first-order correction, associated with a fixed pair of
collision events $\{\sfi,\sfj\}$ and $\{\sfj,\sfk\}$, is of order
$1/\log N$.  Since the number of configurations consisting of two pairwise
collisions involving a common particle is $\binom{h}{3}$, the total
first-order correction is expected to be
$O\bigl(h^3/\log N\bigr)$.

\begin{figure}[ht]
    \centering

    \begin{minipage}[t]{0.48\textwidth}
        \centering

        \begin{tikzpicture}[
            x=0.72cm,
            y=0.72cm,
            line cap=round,
            line join=round
        ]

            \draw[timeslice] (0,-2.7) -- (0,2.7);
            \draw[timeslice] (4,-2.7) -- (4,2.7);
            \draw[timeslice] (8,-2.7) -- (8,2.7);

            \node[below=2pt] at (0,-2.7) {$0$};
            \node[below=2pt] at (4,-2.7) {$n$};
            \node[below=2pt] at (8,-2.7) {$n+m$};

            \coordinate (O)   at (0,0);
            \coordinate (C12) at (4,1.15);
            \coordinate (C13) at (8,-1.05);

            \draw[kernel]
                (O)
                to[out=65,in=165]
                node[particle,pos=0.40,above] {$1$}
                (C12);

            \draw[kernel]
                (C12)
                to[out=-20,in=160]
                node[particle,pos=0.55,above] {$1$}
                (C13);

            \draw[kernel]
                (O)
                to[out=15,in=-155]
                node[particle,pos=0.45,below] {$2$}
                (C12);

            \draw[kernel]
                (O)
                to[out=-65,in=-160]
                node[particle,pos=0.42,below] {$3$}
                (C13);

            \node[vertex] at (O)   {};
            \node[vertex][red] at (C12) {};
            \node[vertex][red] at (C13) {};

            \node[left=5pt] at (O) {$0^{\otimes 3}$};

            \node[pairlabel,above right=2pt] at (C12)
                {$\{1,2\}$};

            \node[pairlabel,below right=2pt] at (C13)
                {$\{1,3\}$};

            \node[font=\small] at (4,-3.55)
                {(a) Joint collision diagram};
            
            \node[pairlabel,below right=4pt] at (C12)
                {$x$};

            \node[pairlabel,below left=2pt] at (C13)
                {$z$};

        \end{tikzpicture}
    \end{minipage}
    \hfill
    \begin{minipage}[t]{0.48\textwidth}
        \centering

        \begin{tikzpicture}[
            x=0.72cm,
            y=0.72cm,
            line cap=round,
            line join=round
        ]

            \draw[timeslice] (0,-2.7) -- (0,2.7);
            \draw[timeslice] (4,-2.7) -- (4,2.7);
            \draw[timeslice] (8,-2.7) -- (8,2.7);

            \node[below=2pt] at (0,-2.7) {$0$};
            \node[below=2pt] at (4,-2.7) {$n$};
            \node[below=2pt] at (8,-2.7) {$n+m$};

            %
            \coordinate (O)   at (0,0);
            \coordinate (C12) at (4,1.25);
            \coordinate (C13) at (8,-1.25);

            \draw[kernel]
                (O)
                to[out=70,in=160]
                node[particle,pos=0.42,above] {$1$}
                (C12);

            \draw[kernel]
                (O)
                to[out=20,in=-155]
                node[particle,pos=0.48,below] {$2$}
                (C12);

            \draw[kernel]
                (O)
                to[out=-18,in=160]
                node[particle,pos=0.34,below] {$1$}
                (C13);

            \draw[kernel]
                (O)
                to[out=-70,in=-160]
                node[particle,pos=0.42,below] {$3$}
                (C13);

            \node[vertex] at (O)   {};
            \node[vertex][red] at (C12) {};
            \node[vertex][red] at (C13) {};

            \node[left=5pt] at (O) {$0^{\otimes 3}$};

            \node[pairlabel,above right=2pt] at (C12)
                {$\{1,2\}$};

            \node[pairlabel,below right=2pt] at (C13)
                {$\{1,3\}$};

            \node[font=\small] at (4,-3.55)
                {(b) Factorised collision diagram};
            
            \node[pairlabel,below right=3pt] at (C12)
                {$x$};

            \node[pairlabel,below left=2pt] at (C13)
                {$z$};

        \end{tikzpicture}
    \end{minipage}

    \caption{
        The joint collision diagram is the simplest version of Figure \ref{fig:U-contraction} involving only $3$ walks. 
        The factorised collision diagram is when we assume independence among collision times of the $3$ walks.
    }
    \label{fig:three-particle-collision-diagrams}
\end{figure}
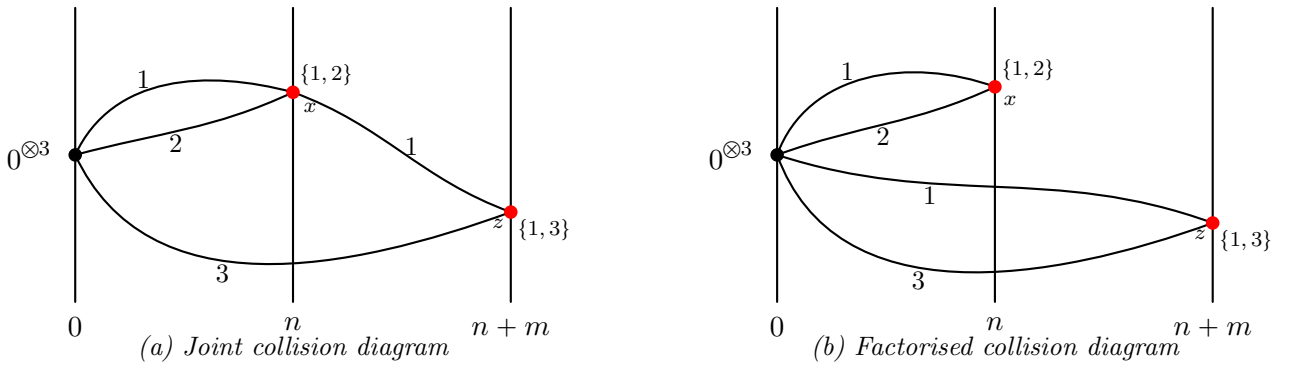
\end{remark}

\newpage  

\begin{lemma}\label{ub:lem-spatial-summation}
One has
\begin{align}\label{eq: ub-frac}
 \cM^{\mathrm{pair},h}_{N,\bhb}
 &\le1+
 \sum_{m\ge1}
 \sum_{\substack{\vec\sfe\in\Pi_h^{\otimes m}\\
 1\le a_1\le b_1<a_2\le b_2<\cdots<a_m\le b_m\le N}}
 \frac{1}{\pi a_1}
 \prod_{r=1}^m\gs_N^{\sfe_r} \, U_N^{\sfe_r}(b_r-a_r)\notag\\
 &\quad\times
 \prod_{r=1}^{m-1}
 \frac{1}{\pi\bigl(a_{r+1}-b_r+\frac{1}2(b_r-a_r)
       +c_{r+1}(a_r-b_{r-1})\bigr)},
\end{align} 
where for $\sfe,\sff\in\Pi_h$, we define $c_{r+1} = c(\sfe_r,\sfe_{r+1})$, where
\begin{align*}
 c(\sfe,\sff):=
 \begin{cases}
  \frac{1}{2},&\sfe\cap\sff\neq\varnothing,\\
  1,&\sfe\cap\sff=\varnothing,
 \end{cases}
\end{align*}
 The same estimate holds uniformly over all initial configurations.
 \end{lemma}

 \begin{proof}
  Sum first  over $y_m$  .  By \eqref{ub:eq-U-def}--\eqref{ub:eq-U-zero}, this
replaces the last spatial renewal kernel by $U_N^{\sfe_m}(b_m-a_m)$.
Next, summing over $x_m$ gives
\begin{align*}
 &\sum_{x_m\in\bbZ^2}
 p_{a_m-b_{\sfp(\sfi_m)}}
   (x_m-y_{\sfp(\sfi_m)})
 p_{a_m-b_{\sfp(\sfj_m)}}
   (x_m-y_{\sfp(\sfj_m)})\notag\\
 &\qquad=
 p_{2a_m-b_{\sfp(\sfi_m)}-b_{\sfp(\sfj_m)}}
   (y_{\sfp(\sfi_m)}-y_{\sfp(\sfj_m)}).
\end{align*} 
The   local limit bound
\begin{equation*}
 \sup_{x\in\bbZ^2}p_n(x)\le\frac{2}{\pi n},
 \qquad n\ge1,
\end{equation*} 
  therefore applies.  Since $\sfe_m\neq\sfe_{m-1}$, if the two pairs share a
vertex, then
\[
 2a_m-b_{\sfp(\sfi_m)}-b_{\sfp(\sfj_m)}
 \ge2(a_m-b_{m-1})+(b_{m-1}-b_{m-2}).
\] 
  If they are disjoint, both predecessor indices are at most $m-2$, and hence
\[
 2a_m-b_{\sfp(\sfi_m)}-b_{\sfp(\sfj_m)}
 \ge2(a_m-b_{m-1})+2(b_{m-1}-b_{m-2}).
\]
These two estimates yield the last denominator in \eqref{eq: ub-frac}.
Iterating the same summation backwards over
$y_{m-1},x_{m-1},\ldots,y_1,x_1$, and finally using
\(
 \sum_xp_{a_1}(x)^2=p_{2a_1}(0)\le1/(\pi a_1)
\), proves the claim.  For arbitrary starting points, the last expression is
$p_{2a_1}(z^{\sfi_1}-z^{\sfj_1})$, which satisfies the same bound .
\end{proof}

  For $\sfe\in\Pi_h$, define
\begin{equation}\label{ub:def-K}
 K_{N,\sfe}(t):=
 \frac{\hb_\sfe}{L}\frac{1}{t\big( 1-\hb_\sfe\frac{\log t}{L} \big)}.
\end{equation} 
Since we assume $\hb_{\sfe}<\hbm<1$, \eqref{ub:def-K} can be extended outside $[1,N]$. Also, for all sufficiently large $N$, $K_{N,\sfe}(t)$ is decreasing on
$[1,3N]$; this can be checked by differentiation. For $\vec \sfe = (\sfe_1,...,\sfe_m)$, we define:
\begin{align}\label{ub:def-S}
  S(m,\vec \sfe)
  :=
  \sum_{\substack{ 
    1\le \sum_{i=1}^m u_i \le N \\ u_r\ge 1\text{ for all $r$}}} \hspace{-0cm}
    \f{1}{1-\hb_{\sfe_m}} \f{\hb_{\sfe_m}}{L}
    \f{1}{u_1} 
    \prod_{2\le r\le m} K_{N,\sfe_{r-1}}\big(u_r+c_{r} u_{r-1}\big).
\end{align}

\begin{lemma}\label{ub:lem-renewal-reduction}
Uniformly in the initial configuration,
 \begin{align}\label{ub:first-ineq}
  \cM^{\mathrm{pair},h}_{N,\bhb}
  & \le
  1
  +
  \sum_{m\ge 1}
  \sum_{\vec \sfe \in \Pi_h^{\otimes m}}
  \big(1+\f{C}{L}\big)^m
  S(m,\vec \sfe).
\end{align}
\end{lemma}

 \begin{proof}
  In \eqref{eq: ub-frac}, set
\begin{equation*}
 v_i:=b_i-a_i,
 \qquad
 u_i:=a_i-b_{i-1}.
\end{equation*} 
  Then
\begin{align*}
 \cM^{\mathrm{pair},h}_{N,\bhb}
 &\le1+
 \sum_{m\ge1}
 \sum_{\substack{\vec\sfe\in\Pi_h^{\otimes m}\\
        \sum_{i=1}^m(u_i+v_i)\le N,\ u_i\ge1,\ v_i\ge0}}
        \hspace{-1cm}
 \frac{1}{\pi u_1}
 \prod_{r=1}^m \gs_N^{\sfe_r} \, U_N^{\sfe_r}(v_r)
 \prod_{r=1}^{m-1}
 \frac{1}{\pi\bigl(u_{r+1}+v_r/2
     +c_{r+1}u_r\bigr)}.
\end{align*} 
  By Proposition~\ref{ub:prop-renewal-convolution}, which is proven in
Appendix~\ref{ub:app-auxiliary}, we have that,
\begin{align}
 \sum_{v=0}^N U_N^{\sfe}(v)
 &\le\frac{1+C/L}{1-\hb_\sfe},\notag\\
 \sum_{v=0}^N
 \frac{U_N^{\sfe}(v)}{\pi(u+v/2)}
 &\le
 \left(1+\frac{C}L\right)
 \frac{1}{\pi u}\frac{1}{1-\f{\hb_{\sfe}}{L} \log u },
 \qquad 1\le u\le N.\label{ub:U(v)-sum}
\end{align}
We enlarge the summation range for each $v_i$ to $[0,N]$.  Then we apply \eqref{ub:U(v)-sum} and using
$
 \frac{\gs_N^{\sfe}}{\pi}
 =\frac{\hb_\sfe}{L}\bigl(1+O(L^{-1})\bigr)
$,  
we have \eqref{ub:first-ineq}.
\end{proof}

\subsection{Upper cells}
\label{Upper cells}
In this subsection we introduce modified  upper-cell kernels  $\widehat K_{N,\sfe}$ as a bound of $K_{N,\sfe}$. The purpose of this substitution is to avoid the boundary terms later when we compare the sums over $u_i$'s with integrals. 

To start, we set the cutoff
\begin{equation*}
 m_\star:=\left\lceil D_0h^2 \right\rceil = o(L),
\end{equation*}
where $D_0$ will be defined later in Lemma \ref{ub:tail-cont}. 
Fix $C_0\ge 6$. Further define the upper cell:
  \begin{equation*}
 \widehat K_{N,\sfe}(t):=e^{\f{C_0}{t}} K_{N,\sfe}(t),\qquad t\ge1.
\end{equation*}

\begin{lemma}\label{ub:lem-upper-cell}
For $m\le m_\star$ and $\vec \sfe=(e_1,\ldots,e_m)\in \Pi_h^{\otimes m}$, we have that:

\begin{align}
 S(m,\vec \sfe)
 &\le \f{1}{1-\hb_{\sfe_m}}
 \int_{[1,N+m_\star]^m}
  \f{\hb_{\sfe_m} \, e^{\f{C_0}{x_1}}}{L x_1}
 \prod_{r=2}^m
\widehat K_{N,\sfe_{r-1}}
 \bigl(x_r+c_r x_{r-1}\bigr)
 \,\mathrm d\mathbf x .
 \label{ub:eq-upper-cell}
\end{align}
\end{lemma}
\begin{proof}
Fix an integer vector $(u_1,\ldots,u_m)$ occurring in the summation in \eqref{ub:def-S}. Write $x_r=u_r+\delta_r$, where $0\le\delta_r<1$. We have that,
\begin{align*}
 \frac{x_1}{u_1}\le1+\frac{1}{u_1}
 \le e^{1/u_1}
 \le e^{2/x_1}.
\end{align*}
Since $C_0\ge6$, this gives
\begin{align}
 \frac{\hb_{\sfe_m}}{L u_1}
 \le 
 \f{\hb_{\sfe_m}e^{C_0/x_1}}{L x_1}.
\end{align} 

Next, we first notice that for $t\ge \f{3}{2}$ and $d\in [0,2]$, we have
\begin{align*}
 \frac{K_{N,\sfe}(t)}{K_{N,\sfe}(t+d)}
 &=\frac{t+d}{t}
   \frac{1-\f{\hb_{\sfe}}{L}\log(t+d) }{ 1-\f{\hb_{\sfe}}{L} \log t}
 \le1+\frac{d}t
 \le e^{d/t}
 \le e^{\f{C_0}{d+t}}.
\end{align*}
Hence for such $t,d$, we have 
\begin{align}\label{ub:K-ratio}
  K_{N,\sfe}(t) \le \widehat K_{N,\sfe}(t+d).
\end{align}
For $r\ge2$, notice that:
\begin{align*}
 x_r+c_r x_{r-1}
 =
 \big( u_r+c_r u_{r-1} \big)
 +
\big( \gd_r+c_r \gd_{r-1}\big).
\end{align*}
It follows that $u_r+c_r u_{r-1}\ge \f{3}{2}$ and $\gd_r+c_r \gd_{r-1}\in [0,2]$. Hence by substituting $t = u_r+c_r u_{r-1}$, $d = \gd_r+c_r \gd_{r-1}$ and applying \eqref{ub:K-ratio} for each $r\ge 2$, we show
\begin{align*}
  \prod_{r=2}^m K_{N,\sfe_{r-1}}(u_r+c_r u_{r-1}) 
  \le 
  \prod_{r=2}^m \widehat K_{N,\sfe_{r-1}}(x_r+c_r x_{r-1})
\end{align*} 
for
$(x_1,...,x_m)\in \prod_{r=1}^m [u_r,u_r+1)$. 
Finally, if $u_1+\cdots+u_m\le N$ and $m\le m_\star$, then the vector $(x_1,\ldots,x_m)$
belongs to $[1,N+m]^m\subset[1,N+m_\star]^m$.   Moreover,
$N+m_\star\le 3N$ for all sufficiently large $N$, so every $\widehat K_{N,\sfe}(\cdot )$ in
the enlarged box remains decreasing.   Summing the cells and enlarging their union to the full box
gives \eqref{ub:eq-upper-cell}.
\end{proof}

\subsection{A probabilistic interpretation}
\label{ub:new-sec-log-ladder}

In this subsection, we use a probabilistic argument to transform the integral in \eqref{ub:eq-upper-cell} to an integral over a shifted simplex. The end result of this subsection is \eqref{ub:new-eq-ladder-shifted}.

\begin{lemma}
Fix \(m\le m_\star\) and
\(\vec\sfe=(\sfe_1,\ldots,\sfe_m)\in\Pi_h^{\otimes m}\). We have that:
\begin{align}\label{ub:new-eq-S-to-L}
 S(m,\vec\sfe)
 & \le
 \frac{1}{1-\hb_{\sfe_m}}
 \int_{[0,M_N]^m}
 \frac{\hb_{\sfe_m}}L e^{C_0e^{-y_1}}
 \prod_{r=2}^m
 \left[
  e^{C_0e^{-y_r}}
  \frac{\hb_{\sfe_{r-1}}/L}{1-\hb_{\sfe_{r-1}}y_r/L}
  \frac{\widehat K_{N,\sfe_{r-1}}(e^{y_r}+c_re^{y_{r-1}})}{\widehat K_{N,\sfe_{r-1}}(e^{y_r})}
 \right]
 \,\dd\mathbf y,
\end{align} 
where 
$
M_N:= \log (N+m_\star)
$. 
For \(m=1\), the product over \(r\) is empty.
\end{lemma}
\begin{proof}
  Apply change of variables \(x_r=e^{y_r}\) in
\eqref{ub:eq-upper-cell}, the Jacobian factors give the result.
\end{proof}

We next dominate each ratio by a distribution function.  Put
\begin{equation}\label{ub:new-eq-Fc-dc}
 F_c(z):=\frac{1}{1+ce^{-z}}.
\end{equation}

\begin{lemma}\label{ub:new-lem-kernel-cdf}
There exists $\gamma_\star>0$, dependent only on $\hbn$ and $\hbm$, such that for all   $x,y\in [0,M_N]$ :
  \begin{align}\label{eq:ub-K-H}
    \frac{\widehat K_{N,\sfe}(e^y+ce^x)}{\widehat K_{N,\sfe}(e^y)}
    \le 
    F_c(y-x)^{1-\f{\gamma_\star}{L}}.
  \end{align}
  Moreover, $F_{c}(z)^{1-\f{\gamma_\star}{L}}$ is a cumulative distribution function.
\end{lemma}
\begin{proof}
For the unmodified kernel, direct substitution gives
\begin{align}
 \frac{K_{N,\sfe}(e^y+ce^x)}{K_{N,\sfe}(e^y)}
 &= \f{ e^y }{ (e^y+ce^x) } \f{ 1-\f{\hb_{\sfe}}{L} y }{ 1-\f{\hb_{\sfe}}{L}\log (e^y+ce^x) } \notag \\
 & =F_c(y-x)
 \frac{1-\f{\hb_{\sfe}}{L} y}{1-\f{\hb_{\sfe}}{L}\bigl(y+\log(1+ce^{x-y}) \bigr)}.
 \label{ub:new-eq-exact-ratio}
\end{align}

Notice that, there exists $\gamma_\star>0$, such that for $d\ge0$,
\begin{align*}
 \log\frac{1-\f{\hb_{\sfe}}{L}y}{1-\f{\hb_{\sfe}}{L}(y+d)}
 &=\int_y^{y+d}
 \frac{\f{\hb_{\sfe}}{L}}{1-\f{\hb_{\sfe}}{L}s}\,\dd s
 \le \frac{\gamma_\star}{L}d,
\end{align*}
since $1-\f{\hb_{\sfe}s}{L}$ is bounded away from $0$ for all $s\in [0,M_N+\log 2]$, as $\hb_{\sfe}<\hbm<1$. Plug this into \eqref{ub:new-eq-exact-ratio} and observe that 
\begin{align*}
  0\le y+\log (1+ce^{x-y}) \le M_N+\log 2
\end{align*}
for all $x,y\in [0,M_N]$, we have that
\begin{equation}
 \frac{K_{N,\sfe}(e^y+ce^x)}{K_{N,\sfe}(e^y)}
 \le F_c(y-x) \, (1+ce^{x-y})^{\f{\gamma_\star}{L}} 
 \le F_c(y-x)^{1-\f{\gamma_\star}{L}}.
\label{ub:K-bdd}\end{equation}

Now if we look at the   upper-cell  modifier:
\begin{align*}
\frac{\widehat K_{N,\sfe}(e^y+ce^x)}{\widehat K_{N,\sfe}(e^y)}
=
 \frac{K_{N,\sfe}(e^y+ce^x)}{K_{N,\sfe}(e^y)}
 \exp\left\{C_0\left(
 \frac{1}{e^y+ce^x}-\f{1}{e^y}
 \right)\right\}
 \le
 \frac{K_{N,\sfe}(e^y+ce^x)}{K_{N,\sfe}(e^y)}.
\end{align*}
  Combining \eqref{ub:K-bdd} we have \eqref{eq:ub-K-H}.

To see that $F_c(z)^{1-\f{\gamma_\star}{L}}$ is a cumulative distribution function, it suffices to show that $F_c(z)$ is a distribution function.  Indeed,
\(F_c\) is smooth, strictly increasing, and has limits zero and one at
\(-\infty\) and \(+\infty\), respectively. 
\end{proof}

Let $Z_2,\ldots,Z_m$ be a family of independent random variables with 
\begin{equation}\label{ub:new-eq-Z-cdf}
 \bbP(Z_r\le z)= F_{c_r}(z)^{1-\f{\gamma_\star}{L}},
 \qquad 2\le r\le m.
\end{equation}
Define
\begin{equation}
 \mathcal S_1:=0,
 \qquad
 \mathcal S_r:=\sum_{a=2}^r Z_a,
 \qquad 2\le r\le m.
\label{ub:def-cumulative-S}\end{equation}

\begin{lemma}
Let $\mathcal S = (\mathcal S_1,...,\mathcal S_m)$ be defined as in \eqref{ub:def-cumulative-S}. Then:
\begin{align}\label{ub:new-eq-ladder-shifted}
 S(m,\vec \sfe)
 &\le 
 \frac{\hb_{\sfe_m}/L}{1-\f{\hb_{\sfe_m}}{L}M_N}\,\,
 \bbE_{\cS}\Bigg[
 \int_{\substack{0\le y_r\le M_N\\
 y_1-\cS_1\le\cdots\le y_m-\cS_m}}
 \prod_{r=1}^m e^{C_0e^{-y_r}}
 \prod_{r=2}^m
 \frac{\hb_{\sfe_{r-1}}/L}{1-\f{\hb_{\sfe_{r-1}}}{L}y_r}
 \,\dd\mathbf y \Bigg].
\end{align}
\end{lemma}

\begin{proof}
We work on \eqref{ub:new-eq-S-to-L}. First we notice that, as $M_N = \log (N+m_\star)\ge L$,
$$
 \frac{1}{1-\hb_{\sfe_m}}
 \le
 \frac{1}{1-\f{\hb_{\sfe_m}}{L}M_N}.
$$ 

Next, by \eqref{eq:ub-K-H} and \eqref{ub:new-eq-Z-cdf}, we have:
\begin{align}\label{eq:ub-cdf bound}
    S(m,\vec \sfe)
    &\le
    \frac{\hb_{\sfe_m}/L}{1-\f{\hb_{\sfe_m}}{L}M_N}
    \int_{[0,M_N]^m}
    e^{C_0e^{-y_1}}
 \prod_{r=2}^m
 \left[
  e^{C_0e^{-y_r}}
  \frac{\hb_{\sfe_{r-1}}/L}{1-\f{\hb_{\sfe_{r-1}}}{L}y_r} \,\,
  \bbP_{\cS}\big( Z_r \le y_r-y_{r-1} \big)
 \right]
    \, \dd \mathbf y.
\end{align}
Now, observe that by independence:
\begin{align*}
  \prod_{r=2}^m \bbP_{\cS}\big( Z_r\le y_r-y_{r-1} \big)
  & =
  \bbP_{\cS}\big( Z_2\le y_2-y_1,...,Z_m\le y_m-y_{m-1} \big)\\
  & =
  \bbP_{\cS}\big( y_1\le y_2-Z_2, \, y_2\le y_3-Z_3,\, ...\, , \, y_{m-1} \le y_m-Z_m \big)\\
  & =
  \bbE_{\cS} \ind_{ \{y_1-\cS_1 \le y_2-\cS_2 \le ...\le y_m-\cS_m\}  }.
\end{align*}
Hence by combining this with \eqref{eq:ub-cdf bound} and exchanging order of expectation and integral, we obtain \eqref{ub:new-eq-ladder-shifted}.
\end{proof}

\subsection{Removing the simplex shift and the upper-cell factors}
\label{ub:new-sec-shifted-simplex}

The goal of this subsection is to show that, one can transform the integral in \eqref{ub:new-eq-ladder-shifted} over a randomly shifted simplex to an unshifted simplex with a controlled cost.  The final result of this subsection can be seen in Lemma \ref{ub:new-lem-after-shift-lemma}.

To start, we first prove an upper bound for a fixed shift $\mathbf s = (s_1,...,s_m)\in \bbR^m$. 

\begin{lemma}\label{ub:new-lem-shifted-simplex}
  Fix deterministic $\mathbf s = (s_1,...,s_m)\in \bbR^m$. We have that
  \begin{align}\label{ub:new-eq-shifted-simplex-bound}
    & \int_{\substack{0\le y_r\le M_N\\
 y_1-s_1\le\cdots\le y_m-s_m}}
 \prod_{r=1}^m e^{C_0e^{-y_r}}
 \prod_{r=2}^m
 \frac{\hb_{\sfe_{r-1}}/L}{1-\f{\hb_{\sfe_{r-1}}}{L}y_r}
 \,\dd\mathbf y \notag \\
 & \le 
 \exp\Big\{
 C \frac{m\bigl(R(\mathbf s)+1\bigr)}{L}
 \Big\}
 \int_{0<y_1<\cdots<y_m<M_N}
 \prod_{r=2}^m
 \frac{\hb_{\sfe_{r-1}}/L}{1-\f{\hb_{\sfe_{r-1}}}{L}y_r}
 \,\dd\mathbf y,
  \end{align}
  where 
  $R(\mathbf s) := \max_r s_r - \min_r s_r$.
\end{lemma}

\begin{proof}
Adding the same constant to all the shifts does not change the shifted
ordering.  Hence we may assume
$$
 0=\min_{1\le r\le m}s_r\le s_r\le R:= R(\mathbf s).
$$
For $1\le r\le m$, define the change of variables:
\begin{align*}
  w_r = w_r(y_r):= \f{M_N}{M_N+R} \big( y_r-s_r+R \big).
\end{align*}
Then the map $y_r\mapsto w_r: [0,M_N] \to [0,M_N]$ is injective, but not necessarily surjective. Hence if we define the reverse map of $w_r\mapsto y_r$ on the full range of $[0,M_N]$, we might map outside $[0,M_N]$. Therefore, we introduce the following clipping function.  For $x\in\bbR$, define its clipping to $[0,M_N]$ by
\[
 [x]_{[0,M_N]}
 :=\min\{M_N,\max\{0,x\}\}
 =
 \begin{cases}
  0, & x<0,\\
  x, & 0\le x\le M_N,\\
  M_N, & x>M_N.
 \end{cases}
\]
The clipping function $x\mapsto[x]_{[0,M_N]}$ is increasing and one-Lipschitz.

The original domain 
\begin{align}\label{ub:original-domain}
\big\{y_1-s_1 \le \cdots \le y_m- s_m, y_r\in [0,M_N]\big\}, 
\end{align}
after applying the change of variable, is a subset of the simplex
\begin{align}\label{ub:new-domain}
\big\{ 0\le w_1 \le \cdots \le w_m \le M_N \big\},
\end{align}
and the Jacobian is $\dd \mathbf y = \Big( \f{M_N+R}{M_N} \Big)^m \dd \mathbf w$.

For $0\le w\le M_N$, set the formal and genuine reverse of $y_r\mapsto w_r$ to be, respectively,
\begin{equation*}
 \ty_r(w):=\frac{M_N+R}{M_N}w-R+s_r,
 \qquad
 \hy_r(w):=[\ty_r(w)]_{[0,M_N]}.
\end{equation*}
On the image of the original domain \eqref{ub:original-domain}, $\ty_r(w_r)=y_r\in[0,M_N]$;
hence $\hy_r(w_r)=y_r$.

Since $[w]_{[0,M_N]}=w$ and clipping is one-Lipschitz,
\begin{align}
 |\hy_r(w)-w|
 &=\big|[\ty_r(w)]_{[0,M_N]}-[w]_{[0,M_N]}\big|
 \le |\ty_r(w)-w|
 =\left|\frac{R}{M_N}w-R+s_r\right|
 \le R.
 \label{ub:new-eq-theta-distance1}
\end{align}
Moreover, by monotonicity of clipping and $\ty_r(w)-(w-R) = \f{Rw}{M_N}+s_r\ge 0$, 
\begin{align}
 \hy_r(w)
 &=[\ty_r(w)]_{[0,M_N]}
 \ge[w-R]_{[0,M_N]}
 =(w-R)_+.
 \label{ub:new-eq-theta-distance2}
\end{align}

Now, by changing variables $y_r\mapsto w_r$ and extending integration domain, we have that
\begin{align}\label{ub:y-to-w}
  & \int_{\substack{0\le y_r\le M_N\\
 y_1-s_1\le\cdots\le y_m-s_m}}
 \prod_{r=1}^m e^{C_0e^{-y_r}}
 \prod_{r=2}^m
 \frac{\hb_{\sfe_{r-1}}/L}{1-\f{\hb_{\sfe_{r-1}}}{L}y_r}
 \,\dd\mathbf y \notag \\
 & \le
 \Big(1+ \f{R}{M_N}\Big)^m
 \int_{0\le w_1\le \cdots \le w_m\le M_N}
 \prod_{r=1}^m e^{C_0e^{-\hy_r(w_r)}}
 \prod_{r=2}^m
 \frac{\hb_{\sfe_{r-1}}/L}{1-\f{\hb_{\sfe_{r-1}}}{L}\hy_r(w_r)}
 \,\dd\mathbf w.
\end{align}

For an arbitrary $\sfe\in \Pi_h$ and $1\le r\le m$, by the mean value theorem,
\begin{align*}
  \log \frac{\hb_{\sfe}/L}{1-\f{\hb_{\sfe}}{L}\hy_r(w)}
  -
  \log \frac{\hb_{\sfe}/L}{1-\f{\hb_{\sfe}}{L}w}
  & \le 
  \bigg|\f{\dd}{\dd y} \log \frac{\hb_{\sfe}/L}{1-\f{\hb_{\sfe}}{L}y} \bigg|\cdot
  \big|\hy_r(w) - w\big|\notag \\
  & =
  \bigg| \frac{\hb_{\sfe}/L}{1-\f{\hb_{\sfe}}{L}y} \bigg| 
  \cdot
  \big|\hy_r(w) - w\big|
  \le \f{CR}{L},
\end{align*}
for some $C>0$, where in the last inequality we have used \eqref{ub:new-eq-theta-distance1}. Thus
\begin{align}\label{ub:no-shift1}
  \frac{\hb_{\sfe}/L}{1-\f{\hb_{\sfe}}{L}\hy_r(w)}
  \le 
  e^{\f{CR}{L}}
  \frac{\hb_{\sfe}/L}{1-\f{\hb_{\sfe}}{L}w}.
\end{align}
Also we have
$
 \left(1+\frac{R}{M_N}\right)^m
 \le\exp\left\{\frac{mR}{M_N}\right\}
 \le\exp\left\{\frac{mCR}{L}\right\}
$. Together with \eqref{ub:no-shift1}, we upper bound \eqref{ub:y-to-w} by
\begin{align}\label{ub:drop-modif1}
  \exp\Big\{\f{mCR}{L}\Big\}
 \int_{0\le w_1\le \cdots \le w_m\le M_N}
 \prod_{r=1}^m e^{C_0e^{-\hy_r(w_r)}}
 \prod_{r=2}^m
 \frac{\hb_{\sfe_{r-1}}/L}{1-\f{\hb_{\sfe_{r-1}}}{L}w_r}
 \,\dd\mathbf w.
\end{align}

By \eqref{ub:new-eq-theta-distance2}, this can be further bounded by
\begin{align}\label{ub:drop-modif2}
 \exp\Big\{\f{mCR}{L}\Big\}
 \int_{0\le w_1\le \cdots \le w_m\le M_N}
 \prod_{r=1}^m e^{C_0e^{-(w_r-R)_+}}
 \prod_{r=2}^m
 \frac{\hb_{\sfe_{r-1}}/L}{1-\f{\hb_{\sfe_{r-1}}}{L}w_r}
 \,\dd\mathbf w.
\end{align}
Suppose that $\mathbf W = \{W_1<...<W_m\}$ is a family of ordered uniform random variables on $[0,M_N]$. Then the above can be written in the following expectation form
\begin{align*}
 \f{M_N^m}{m!}
  \exp\Big\{\f{mCR}{L}\Big\}
 \bbE_{\mathbf W}\bigg[
 \prod_{r=1}^m e^{C_0e^{-(W_r-R)_+}}
 \prod_{r=2}^m
 \frac{\hb_{\sfe_{r-1}}/L}{1-\f{\hb_{\sfe_{r-1}}}{L}W_r}
 \bigg].
\end{align*}   
The first product in the above expectation is decreasing with respect to each $W_r$, and the second product is increasing with respect to each $W_r$. Hence the two products are negatively correlated, and the above is bounded above by
\begin{align}
 & \f{M_N^m}{m!}
  \exp\Big\{\f{mCR}{L}\Big\}
 \bbE_{\mathbf W}\bigg[
 \prod_{r=1}^m e^{C_0e^{-(W_r-R)_+}}\bigg] \,
 \bbE_{\mathbf W}\bigg[
 \prod_{r=2}^m
 \frac{\hb_{\sfe_{r-1}}/L}{1-\f{\hb_{\sfe_{r-1}}}{L}W_r}
 \bigg] \notag \\ 
 & = 
 \exp\Big\{\f{mCR}{L}\Big\}
 \f{m!}{M_N^m} \Bigg( \int_{0\le w_1\le\cdots \le w_m\le M_N}
 \prod_{r=1}^m
 e^{C_0e^{-(w_r-R)_+}}
 \, \dd \mathbf w \Bigg)
 \Bigg( \int_{0\le w_1\le \cdots \le w_m\le M_N}
 \prod_{r=2}^m
 \frac{\hb_{\sfe_{r-1}}/L}{1-\f{\hb_{\sfe_{r-1}}}{L}w_r}
 \,\dd\mathbf w\Bigg)\notag \\ 
 & =
 \exp\Big\{\f{mCR}{L}\Big\}
 \Bigg( \f{1}{M_N} \int_{0}^{ M_N}
 e^{C_0e^{-(w-R)_+}}
 \, \dd  w \Bigg)^m
 \Bigg( \int_{0\le w_1\le \cdots \le w_m\le M_N}
 \prod_{r=2}^m
 \frac{\hb_{\sfe_{r-1}}/L}{1-\f{\hb_{\sfe_{r-1}}}{L}w_r}
 \,\dd\mathbf w\Bigg).\label{ub:drop-modif3}
\end{align}
The last task is to estimate the integral under the $m$'th power. If $R\ge M_N$, then $(w-R)_+=0$ for $0\le w\le M_N$, and
\begin{align}\label{ub:int11}
\frac1{M_N}\int_0^{M_N}
 e^{C_0 e^{-(w-R)_+}} \,\dd w
=e^{C_0}
 \le\exp\left\{C_0\frac{R}{M_N}\right\}.
\end{align}
If $0\le R<M_N$, then
\begin{align}\label{ub:int22}
\frac1{M_N}\int_0^{M_N}
 e^{C_0 e^{-(w-R)_+}} \,\dd w
 &=\frac{R}{M_N}e^{C_0}
   +\frac1{M_N}\int_R^{M_N}e^{C_0e^{R-w}}\,\dd w\notag \\
 &=1+\frac{R}{M_N}(e^{C_0}-1)
   +\frac1{M_N}\int_R^{M_N}
      \bigl(e^{C_0e^{R-w}}-1\bigr)\,\dd w\notag \\
 &\le1+\frac{R}{M_N}(e^{C_0}-1)
   +\frac{C_0e^{C_0}}{M_N}
      \int_R^{M_N}e^{R-w}\,\dd w\notag \\
 &\le1+C \frac{R+1}{M_N}
 \le\exp\Big\{C \frac{R+1}{M_N}\Big\}.
\end{align}
Combining \eqref{ub:int11} and \eqref{ub:int22} and plugging into \eqref{ub:drop-modif3}, we finish the proof of \eqref{ub:new-eq-shifted-simplex-bound}.
\end{proof}

We now apply \eqref{ub:new-eq-shifted-simplex-bound} on randomised shifts. Define
\begin{equation}\label{ub:def-R}
 \mathcal R_m
 :=\max_{1\le r\le m}\mathcal S_r
   -\min_{1\le r\le m}\mathcal S_r.
\end{equation}

Combining \eqref{ub:new-eq-ladder-shifted} with
\cref{ub:new-lem-shifted-simplex} gives
\begin{lemma}\label{ub:new-lem-after-shift-lemma}
\begin{equation}\label{ub:new-eq-after-shift-lemma}
 S(m,\vec \sfe)
 \le
 \frac{\hb_{\sfe_m}/L}{1-\f{\hb_{\sfe_m}}{L}M_N}\,\,
 \exp\Big\{ \f{Cm}{L} \Big\}
 \bbE_{\cS}\bigg[
 \exp\Big\{
 C \frac{m \cR_m}{L}
 \Big\}\bigg]
 \int_{0<y_1<\cdots<y_m<M_N}
 \prod_{r=2}^m
 \frac{\hb_{\sfe_{r-1}}/L}{1-\f{\hb_{\sfe_{r-1}}}{L}y_r}
 \,\dd\mathbf y.
\end{equation}
\end{lemma}

\subsection{Estimating the shifted simplex weight}
\label{ub:new-sec-logistic-walk}

The goal of this subsection is to obtain an upper bound of the exponential moment of $\cR_m$ and to reach Lemma \ref{ub:new-prop-fixed-word}.  We first perform a decomposition on $Z_r$. Recall \eqref{ub:new-eq-Fc-dc}. Since \(F_c(z)=F_1(z-\log c)\), we can decompose
\begin{equation*}
 Z_r=\log c_r+W_{r,N},
\end{equation*}
where \(W_{2,N},\ldots,W_{m,N}\) are independent and identically distributed
with
\begin{equation*}
 \bbP(W_{r,N}\le z)=\left(\frac{1}{1+e^{-z}}\right)^{1-\f{\gamma_\star}{L}}.
\end{equation*}
Denote $\bbE_{\cS}[W_{r,N}] = \mu_N$ for all $2\le r\le m$, and write the centralised variables $\tW_{r,N} := W_{r,N} - \mu_N$. Then we decompose $\cS_r$:
\begin{align*}
  \cS_r 
  = 
  \sum_{i=2}^r (\log c_i + \mu_N)
  +
  \sum_{i=2}^r \tW_{i,N}.
\end{align*}
Recall definition of $\cR_m$ from \eqref{ub:def-R}. We then obtain the bound
\begin{align}\label{ub:R-bound}
  \cR_m
  \le 
  2 \max_{1\le r\le m} \Big|\sum^r_{i=2} \tW_{i,N} \Big|
  +s(\vec \sfe) \log 2
  + m|\mu_N|,
\end{align}
where
\begin{align*}
  s(\vec\sfe)
 :=\#\left\{2\le r\le m:
 \sfe_{r-1}\cap\sfe_r\neq\varnothing\right\}.
\end{align*}
is the number of indices on which there are two consecutive pairs and they share one common walk. The next lemma gives an estimate on $\mu_N$ and more information on $W_{r,N}$ and $\tW_{r,N}$.

\begin{lemma}
\label{ub:new-lem-logistic-mgf}
There exists $C>0$ such that 
\begin{align}\label{ub:new-eq-mu-small}
|\mu_N|\le \f{C}{L}.
\end{align} 
Moreover, for the centralised variables, there are constants
\(t_0,C_\star>0\) such that
\begin{equation}\label{ub:new-eq-centered-local-mgf}
 \bbE e^{t\tW_{r,N}}\le e^{C_\star  t^2},
 \qquad \text{ for all }|t|\le t_0.
\end{equation}
\end{lemma}

\begin{proof}
By differentiating the distribution function of $W_{r,N}$ and a change of variable $u:= (1+e^{-w})^{-1}$,
\begin{align}
 \bbE e^{tW_{r,N}}
 & = \big(1-\f{\gamma_\star}{L}\big)\int_\bbR e^{(t-1)w} \Big( \f{1}{1+e^{-w}}\Big)^{2 -\f{\gamma_\star}{L}}  \, \dd w\notag \\
 &=\big(1-\f{\gamma_\star}{L}\big)\int_0^1u^{t-\f{\gamma_\star}{L}}(1-u)^{-t}\,\dd u\notag \\
 &=\big(1-\f{\gamma_\star}{L}\big)
    B\big(1-\f{\gamma_\star}{L}+t,1-t\big)
 =\frac{\Gamma(1-\f{\gamma_\star}{L}+t)\Gamma(1-t)}{\Gamma\big(1-\f{\gamma_\star}{L}\big)}.\label{ub:mgf-W}
\end{align}
where $B$ is the beta function with $B(p,q) = \Gamma(p)\Gamma(q)/\Gamma(p+q)$. 
For \eqref{ub:new-eq-mu-small}, we have:
\begin{align}\label{ub:psi}
\bbE W_{r,N} = \f{\dd}{\dd t} \log \bbE e^{tW_{r,N}}\Big|_{t=0}
& =
\psi\big(1-\f{\gamma_\star}{L}\big)
-\psi(1),
\end{align}
  where \(\psi=\Gamma'/\Gamma\) is the digamma function.
By \cite{DLMF}, $\psi'$ has a series expansion:
\begin{align}\label{ub:digamma1expan}
 \psi'(x)=\sum_{k=0}^\infty\frac{1}{(k+x)^2},
 \qquad x>0.
\end{align}
For large $N$, $1-\f{\gamma_\star}{L}\in [1/2,1]$. Since \eqref{ub:digamma1expan} is uniformly bounded on \([1/2,1]\), using the mean value theorem on \eqref{ub:psi} proves
\eqref{ub:new-eq-mu-small}. 

For \eqref{ub:new-eq-centered-local-mgf}, we Taylor-expand 
$
\log \bbE_{\cS} \big[
    e^{t\tW_{r,N}}\big]
$. 
Naturally the constant order term vanishes. By \eqref{ub:mgf-W}, as $\tW_{r,N}$ is centralised, the first derivative term also vanishes. For the second derivative term,
\begin{align*}
  \f{\dd^2}{\dd t^2}\log \bbE_{\cS} \big[
    e^{t\tW_{r,N}}\big]
  & = \psi'\big( 1-\f{\gamma_\star}{L}+t \big)
  +\psi'(1-t).
\end{align*}
Again by \eqref{ub:digamma1expan}, the above is bounded away from $0$ and bounded above, for all $|t|\le \f{1}{4}$ and sufficiently large $N$. Hence Taylor's theorem gives \eqref{ub:new-eq-centered-local-mgf}.
\end{proof}

Write
\begin{align*}
  \cW_m^*:=\max_{1\le r\le m}\Big|\sum_{i=2}^r \tW_{i,N} \Big|.
\end{align*}

\begin{lemma}
\label{ub:new-lem-maximal-walk}
There exists $C,\theta_0>0$ such that
 \begin{align}
 \bbE e^{\theta\mathcal W_m^*}
 &\le e^{ C \big(\theta\sqrt m+\theta^2m\big)},
 \qquad \text{ for all }\, 0\le\theta\le\theta_0.
 \label{ub:new-eq-max-mgf}
\end{align}
\end{lemma}

\begin{proof}
We start with a tail estimate: there exists $c>0$ such that
\begin{align}
 \bbP(\mathcal W_m^*\ge x)
 &\le
 2\exp\Big\{-c\min\Big(\frac{x^2}{m},x\Big)\Big\}.
 \label{ub:new-eq-max-tail}
 \end{align}
To see this, we apply a martingale maximal inequality. Fix $t\in(0,t_0]$ and let
$
\mathcal F_r=\sigma\big(\tW_{2,N},\ldots,\tW_{r,N}\big)
$.  
By \eqref{ub:new-eq-centered-local-mgf},
\[
 \exp\Big\{t \sum_{i=2}^r \tW_{i,N}-C_\star(r-1)t^2\Big\},
 \qquad 1\le r\le m,
\]
is a nonnegative supermartingale.  The maximal inequality for a nonnegative
supermartingale gives
\[
 \bbP\Big(\max_{r\le m}\sum_{i=2}^r \tW_{i,N} \ge x\Big)
 \le\exp\{-tx+C_\star  mt^2\}.
\]
The unconstrained minimiser of the above is \(t=x/(2C_\star m)\). However \eqref{ub:new-eq-centered-local-mgf} is only proved for $|t|\le t_0$, so if $x/(2C_\star m)\le t_0$, we take $t$ to be the unconstrained minimiser and we get a Gaussian tail:
\begin{align}\label{ub:oneside1}
  \bbP\Big(\max_{r\le m} \sum_{i=2}^r \tW_{i,N} \ge x\Big)
 \le\exp\{-\f{x^2}{4C_\star m}\}.
\end{align}
If $x/(2C_\star m)> t_0$, we simply take $t = t_0$. This gives:
\begin{align}\label{ub:oneside2}
  \bbP\left(\max_{r\le m} \sum_{i=2}^r \tW_{i,N} 
  \ge x\right)
 \le \exp\{-tx+C_\star mt^2\}
 \le \exp \{ -\f{t_0 x}{2}   \}.
\end{align}
Similarly, 
$$
 \exp\left\{-t \, \sum_{i=2}^r \tW_{i,N} -C_\star(r-1)t^2\right\},
 \qquad 1\le r\le m,
$$
is also a supermartingale. Hence again by the maximal inequality for a nonnegative supermartingale,
\begin{align}\label{ub:otherside}
  \bbP\Big(\min_{r\le m}\sum_{i=2}^r \tW_{i,N}\le  -x\Big)
  =
  \bbP\Big(\max_{r\le m} \big( -\sum_{i=2}^r \tW_{i,N}\big) \ge x\Big)
  \le
  \exp\big\{ - \min\{ \f{x^2}{4C_\star m}, \f{t_0 x}{2} \} \big\}.
\end{align}
Gathering \eqref{ub:oneside1}, \eqref{ub:oneside2}, \eqref{ub:otherside} and by a union bound, we finish the proof of \eqref{ub:new-eq-max-tail}.

Now we prove \eqref{ub:new-eq-max-mgf}. As $e^{\theta \cW^*_m}$ is a nonnegative random variable, by \eqref{ub:new-eq-max-tail}, we have:
\begin{align}
 \bbE e^{\theta \cW^*_m}
 & =
 1+\theta\int_0^\infty e^{\theta x}\bbP(\cW^*_m\ge x)\,\dd x\notag \\
 & \le
 1+2 \theta \int_0^\infty e^{\theta x} \exp\big\{ -c \min\big( \f{x^2}{m},x \big) \big\}
 \, \dd x\notag \\
&  =
1+2 \theta 
\bigg( \int_0^m e^{\theta x} \exp\big\{ -c \f{x^2}{m} \big\}\, \dd x
+ \int_m^\infty e^{\theta x} \exp\big\{ -c x  \big\} \, \dd x \bigg).\label{ub:int0}
\end{align}
For the first integral above, we complete the square and extend the integral to $(-\infty,+\infty)$ to obtain:
\begin{align}\label{ub:int1}
  \int_0^m e^{\theta x} \exp\Big\{ -c  \f{x^2}{m} \Big\}\, \dd x
  & \le 
  \sqrt{\f{\pi m}{c }}e^{\f{m\theta^2}{4c }}.
\end{align}
For the second integral, for $\theta\in [0,\f{c}{2}]$, 
  \begin{align}\label{ub:int2}
  \int_m^\infty e^{ (\theta  -c ) x }  \, \dd x 
  = \f{e^{-(c -\theta)m}}{c -\theta}
  \le 
  \f{2e^{-c m/2}}{ c  }.
\end{align} 
  Plugging  \eqref{ub:int1} and \eqref{ub:int2} into \eqref{ub:int0} finishes the proof. Indeed,
\begin{align*}
 \bbE e^{\theta\cW_m^*}
 &\le
 1+C \theta\sqrt m\,
       e^{C m\theta^2}
   +C \theta e^{-c m/2}
   \le
 1+C \theta\sqrt m\,
       e^{C  m\theta^2}
  \le 
  e^{C\big(\theta\sqrt{m} + \theta^2 m\big)}.
\end{align*}
\end{proof}

Combining \eqref{ub:new-eq-after-shift-lemma}, \eqref{ub:R-bound}
\eqref{ub:new-eq-mu-small}, and
\eqref{ub:new-eq-max-mgf} with
\(\theta=C_\star m/L\) proves the end result of this subsection.

\begin{lemma}
\label{ub:new-prop-fixed-word}
There is \(c>0\) such that, uniformly for
\(m\le c L\),
\begin{align}
 S(m,\vec \sfe)
 &\le
 \exp\left\{C \left(
 \frac{m^{3/2}}L
 +\frac{m^3}{L^2}
 +\frac{m\,s(\vec\sfe)}L
 +\frac{m^2}{L^2}
 +\frac{m}L
 \right)\right\}\notag \\
 & \quad \times 
 \frac{\hb_{\sfe_m}/L}{1-\f{\hb_{\sfe_m}}{L}M_N}\,\,
 \int_{0<y_1<\cdots<y_m<M_N}
 \prod_{r=2}^m
 \frac{\hb_{\sfe_{r-1}}/L}{1-\f{\hb_{\sfe_{r-1}}}{L}y_r}
 \,\dd\mathbf y.
\label{ub:bound-s}\end{align}
\end{lemma}

\subsection{Counting consecutive pairs with one common index}
\label{ub:new-sec-exact-reference}
In this subsection, we sum over all possible diagrams up to length $m_\star$. We first show that, the major term in \eqref{ub:bound-s} sums up to exactly the expression we hope. We first introduce an abbreviation. For any $T\in [0,M_N]$ and $\sfe\in \Pi_h$, we define
\begin{align*}
  \gl_{\sfe}(T)
  :=
  \f{ \hb_{\sfe}/L }{ 1-\f{\hb_{\sfe}}{L}T }.
\end{align*}
With this notation, \eqref{ub:bound-s} becomes
\begin{align}
 S(m,\vec \sfe)
 &\le
 \exp\left\{C \left(
 \frac{m^{3/2}}L
 +\frac{m^3}{L^2}
 +\frac{m\,s(\vec\sfe)}L
 +\frac{m^2}{L^2}
 +\frac{m}L
 \right)\right\}\notag \\
 & \quad \times 
 \gl_{\sfe_m}(M_N)\,\,
 \int_{0<y_1<\cdots<y_m<M_N}
 \prod_{r=2}^m
 \gl_{\sfe_{r-1}}(y_r)
 \,\dd\mathbf y.
\label{ub:bound-s2}\end{align}

\begin{lemma}
\label{ub:new-prop-exact-reference}
For every \(0\le T\le M_N\),
\begin{equation}\label{ub:new-eq-exact-reference-sum}
 G(T)
 := 
 1+\sum_{m\ge1}\sum_{\vec\sfe\in\Pi_h^{\otimes m}}
 \gl_{\sfe_m}(T) 
 \int_{0<y_1<\cdots < y_m<T}
 \prod_{r=2}^m \gl_{\sfe_{r-1}}(y_r)
 \, \dd \mathbf y
 =
 \prod_{\sfe\in \Pi_h}
 \Big( \f{1}{1-\hb_{\sfe}\f{T}{L}} \Big) .
\end{equation} 
\end{lemma}
\begin{proof}
For $\sfe\in \Pi_h$, define
$$
 A_{\sfe}(T)
 :=\sum_{m\ge1 }
 \sum_{\substack{\vec\sff\in\Pi_h^{\otimes {m-1}} \\ \sff_{m-1} \neq \sfe}}
 \gl_{\sfe}(T) 
 \int_{0<y_1<\cdots < y_m<T}
 \prod_{r=2}^m \gl_{\sff_{r-1}}(y_r)
 \, \dd \mathbf y.
$$
where when $m=1$, the product in the integral is understood to be $1$. We have an extra constraint $\sff_{m-1}\neq \sfe$ in the summation as we do not allow consecutive identical pairs. Naturally, one has
\begin{equation}\label{ub:G-A}
 G(T)=1+\sum_{\sfe\in\Pi_h}A_{\sfe}(T).
\end{equation}
For $A_{\sfe}(T)$, by conditioning on $y_m = t$ and changing variable $l=m-1$, we have the recursive relation
\begin{align}
  A_{\sfe}(T)
  & = \gl_{\sfe}(T)\bigg( 
    \int^T_0 \Big[ 1+ 
    \sum_{l\ge 1} \sum_{\substack{\vec \sff\in \Pi_h^{\otimes {l}} \\ \sff_l\neq \sfe }} 
    \gl_{\sff_l}(t)
    \int_{0<y_1<\cdots<y_{l}<t} 
  \prod_{r=2}^{l} \gl_{\sff_{r-1}}(y_r)
  \, \dd \mathbf y\bigg) 
  \Big]\, \dd t \notag \\ 
  & =
  \lambda_{\sfe}(T)
   \int_0^T
   \Big[1+\sum_{\sfg\ne\sfe}A_{\sfg}(t)\Big]\dd t
   =\lambda_{\sfe}(T)
   \int_0^T\bigl(G(t)-A_{\sfe}(t)\bigr)\,\dd t,
 \label{ub:eq-A-recursion}
\end{align}
where in the last equality we have used \eqref{ub:G-A}. Since
$
 \lambda_{\sfe}'(T)=\lambda_{\sfe}(T)^2
$, 
differentiation of \eqref{ub:eq-A-recursion} yields
\begin{align}\label{ub:int-A}
 A_{\sfe}'(T)
 &=\lambda_{\sfe}(T)^2
   \int_0^T\bigl(G(t)-A_{\sfe}(t)\bigr)\,\dd t
   +\lambda_{\sfe}(T)\bigl(G(T)-A_{\sfe}(T)\bigr)\notag \\
 &=\lambda_{\sfe}(T)A_{\sfe}(T)
   +\lambda_{\sfe}(T)\bigl(G(T)-A_{\sfe}(T)\bigr)
  =\lambda_{\sfe}(T)G(T),
\end{align}
where in the second equality we have used \eqref{ub:eq-A-recursion}. Summing over $\sfe\in \Pi_h$, we obtain
\begin{align*}
 G'(T)=\sum_{\sfe\in\Pi_h}A_{\sfe}'(T)
 =G(T)\sum_{\sfe\in\Pi_h}\lambda_{\sfe}(T),
 \qquad G(0)=1.
\end{align*}
Hence solving the ODE we obtain \eqref{ub:new-eq-exact-reference-sum}.
\end{proof}

Next we wish to show that, the weights from the error term in \eqref{ub:bound-s2} does not affect too much when we sum over all the sequence $\vec \sfe$. In this spirit, we again consider them in a probability measure and show that they have negligible exponential mean. Define the space of all sequences of pairs:
\begin{align*}
  \Pi_h^{\otimes \bbN_0}
  :=
  \union_{m\ge 0} \Pi_h^{\otimes m}
\end{align*}
where $\Pi_h^{\otimes 0}:= \{\emptyset\}$. For $T\in [0,M_N]$, define a probability measure $\bbQ_T$ on $\Pi_h^{\otimes \bbN_0}$ by
\begin{align}\label{ub:def-Q}
  \bbQ_T(\vec E = \vec \sfe)
  :=
  \f{\gl_{\sfe_m}(T)}{G(T)} 
 \int_{0<y_1<\cdots < y_m<T}
 \prod_{r=2}^m \gl_{\sfe_{r-1}}(y_r)
 \, \dd \mathbf y,
 \quad \quad 
 \bbQ_T(\vec E = \emptyset)
 := \f{1}{G(T)}.
\end{align}
By \eqref{ub:new-eq-exact-reference-sum} this is indeed a probability measure. The only problematic error in \eqref{ub:bound-s2} is the term with $s(\vec \sfe)$. All the other terms can be bounded by raising $m$ to $m_\star$. Therefore, for a random pair-sequence $\vec E$, we wish to understand the exponential moment of $s(\vec E)$ under $\bbQ_{M_N}$. 
\begin{lemma}
\label{ub:lem-W-bound}
There exist constants \(C,\theta_0>0\), depending only on
\(\hbn\) and \(\hbm\) , such that for every \(h\ge3\),
\(m_\star\ge1\), \(0<T\le M_N\), and \(0\le\theta\le\theta_0\),
\begin{equation}\label{ub:W-bound}
 \bbE_{\bbQ_T}\left[
   e^{\theta \, s(\vec E)}\ind_{\{|\vec E|\le m_\star\}}
 \right]
 \le
 \exp\left\{\frac{C \, m_\star\theta}{h}\right\}.
\end{equation} 
\end{lemma}
\begin{proof}
Let $\vec E = (E_1,...,E_{|\vec E|})$ be a random variable in $\Pi_h^{\otimes \bbN_0}$ with law $\bbQ_T$. For $1\le r\le |\vec E|$, we define the reversed sequence by
$
  \tE_r
  :=
  E_{|\vec E|-r+1}
$. 
Also define a filtration $\cG_r:= \gs\big( \tE_1,...,\tE_r \big)$.  We further define:
\begin{align*}
  I_r 
  := 
  \ind_{ \{ 1\le r\le |\vec E|-1, \, 
              \tE_{r}\cap \tE_{r+1}\neq \emptyset  \}}.
\end{align*}
For $|\vec E|\le m_\star$, we have
$
  s(\vec E)
  =
  \sum_{r=1}^{m_\star -1} I_r
$. 
We claim that there exists $C>0$, for all $r\in \bbN$, 
\begin{align}\label{ub:1step-R}
  \bbE_{\bbQ_T} \big[
      e^{\theta I_r}\big| \cG_r \big]
      \le 
      1+\f{C(e^\theta-1)}{h}.
\end{align}
Let $M\in \bbN$. Assuming \eqref{ub:1step-R}, as $e^{\theta \sum_{1\le r\le M-1} I_r}$ is $\cG_M$-measurable, by Tower law,
\begin{align*}
  \bbE_{\bbQ_T} \Big[
    e^{\theta \sum_{1\le r\le M} I_r}
    \Big]
  & = 
  \bbE_{\bbQ_T} \Big[
    e^{\theta \sum_{1\le r\le M-1} I_r}
    \,\,
    \bbE_{\bbQ_T} \big[
      e^{\theta I_M}\big| \cG_M \big]
    \Big]\notag \\ 
  & \le 
  \big(1+\f{C(e^\theta-1)}{h}\big)
  \bbE_{\bbQ_T} \Big[
    e^{\theta \sum_{1\le r\le M-1} I_r}
    \Big]
    \le 
  \big(1+\f{C(e^\theta-1)}{h}\big)^M.
\end{align*}
Therefore, for $\theta\in [0,\theta_0]$, we have
\begin{align*}
  \bbE_{\bbQ_T} \Big[
    e^{\theta s(\vec E) } \ind_{|\vec E|\le m_\star}
    \Big]
    \le 
    \big(1+\f{C(e^\theta-1)}{h}\big)^{m_\star-1}
    \le 
    e^{\f{Cm_\star \theta}{h}}.
\end{align*}
Now the remaining task is to show \eqref{ub:1step-R}. First, observe that there exist $c,C>0$ such that for all $\sfe\in \Pi_h$ and $u\in [0,M_N]$, we have
$
  \f{c}{L} \le \gl_{\sfe}(u) \le \f{C}{L}
$. 
By integrating \eqref{ub:int-A}, we have
\begin{equation*}
 A_{\sfe}(t)=\int_0^t\lambda_{\sfe}(u)G(u)\,\dd u.
\end{equation*}
Hence there exists $C>0$ such that for all $T\in [0,M_N]$, we have
\begin{equation}\label{ub:eq-A-comparison-short}
 \max_{\sfe\in\Pi_h}A_{\sfe}(T)
 \le C \min_{\sfe\in\Pi_h}A_{\sfe}(T),
 \qquad 0\le T\le M_N.
\end{equation}
For $r\ge |\vec E|$, we always have $\bbE_{\bbQ_T}\big[ I_r |\cG_r \big] = 0$, so $\bbE_{\bbQ_T}\big[ e^{\theta I_r}|\cG_r \big] = 1$ for any $\theta\in \bbR$.
For $1\le r \le |\vec E|-1$, to study probability conditioned on $\cG_r$, it suffices to study conditional probability on 
$\big\{ \tE_1 = \sfe_1,\cdots, \tE_r = \sfe_r \big\}$ for all $(\sfe_1,...,\sfe_r)\in \Pi_h^{\otimes r}$, as they form the atoms in $\cG_r$. So for any $\sfe\neq \sfe_r$, we consider
\begin{align}
  & \bbQ_T\Big[
    1\le r\le |\vec E|-1,\, 
    \tE_{r+1} = \sfe \,
    \Big| \, \tE_1 = \sfe_1,\cdots ,\tE_r = \sfe_r
    \Big]\notag \\
  & =
  \f{
    \int^T_0 A_{\sfe}(t)
    \Big( \gl_{\sfe_1}(T) \int_{t<y_2<\cdots <y_r<T} 
    \prod_{i=2}^r \gl_{\sfe_{r-i+2}}(y_i) \, \dd y_2\cdots \dd y_r \Big)
    \, \dd t
  }{
    \int^T_0 
    \big( 1+\sum_{\sff\neq \sfe_r} A_{\sff}(t)\big)
    \Big( \gl_{\sfe_1}(T) \int_{t<y_2<\cdots <y_r<T} 
    \prod_{i=2}^r \gl_{\sfe_{r-i+2}}(y_i) \, \dd y_2\cdots \dd y_r \Big)
    \, \dd t
  }.
\end{align}
To get the above equation, refer to \eqref{ub:def-Q} and recall that we define $(\tE_r)$ to be the reverse sequence. Therefore,
\begin{align}
  & \bbQ_T\Big[
    1\le r\le |\vec E|-1,\, 
    \tE_{r+1} \cap \sfe_r \neq \emptyset \,
    \Big| \, \tE_1 = \sfe_1,\cdots ,\tE_r = \sfe_r
    \Big]\notag \\
  & =
  \f{
    \sum_{\sfe:|\sfe \cap \sfe_r| =1}
    \int^T_0 A_{\sfe}(t)
    \Big( \gl_{\sfe_1}(T) \int_{t<y_2<\cdots <y_r<T} 
    \prod_{i=2}^r \gl_{\sfe_{r-i+2}}(y_i) \, \dd y_2\cdots \dd y_r \Big)
    \, \dd t
  }{
    \int^T_0 
    \big( 1+\sum_{\sff\neq \sfe_r} A_{\sff}(t)\big)
    \Big( \gl_{\sfe_1}(T) \int_{t<y_2<\cdots <y_r<T} 
    \prod_{i=2}^r \gl_{\sfe_{r-i+2}}(y_i) \, \dd y_2\cdots \dd y_r \Big)
    \, \dd t
  }.
\label{ub:cond-prob}\end{align}
Notice that in the numerator in \eqref{ub:cond-prob}, the summation contains $2(h-2)$ terms, while in the denominator, the summation contains $\binom{h}{2}-1$ terms. By \eqref{ub:eq-A-comparison-short}, we see that the summands are all comparable, so \eqref{ub:cond-prob} is upper bounded by $C \f{ 2 (h-2)}{\binom{h}{2}-1} \le \f{C'}{h}$, for some $C,C'>0$. Since this upper bound is uniform for all $(\sfe_1,\cdots \sfe_r)\in \Pi_h^{\otimes r}$, we achieve: 
\begin{align*}
  \bbQ_T\big[ 
    1\le r\le |\vec E|-1,\, 
    \tE_{r+1} \cap \tE_r \neq \emptyset \,
    \Big| \, \cG_r
     \big]
  \le \f{C}{h}. 
\end{align*}
To finish, just notice that
\begin{align*}
  \bbE_{\bbQ_T} \big[
      e^{\theta I_r}\big| \cG_r \big]
      =
      1+
      (e^{\theta}-1) \, \bbQ_T\big[ 
                    1\le r\le |\vec E|-1,\, 
                    \tE_{r+1} \cap \tE_r \neq \emptyset \,
                    \Big| \, \cG_r
                    \big]
      \le 1+\f{C(e^{\theta}-1)}{h}
\end{align*}
and we get \eqref{ub:1step-R}.
\end{proof}

We have reached the final result of this subsection.

\begin{lemma}
\begin{align}
 &1+\sum_{1\le m\le m_\star}
   \sum_{\vec\sfe\in\Pi_h^{\otimes m}}
   \big(1+\f{C}{L}\big)^m
   S(m,\vec \sfe) \notag\\
 &\quad\le
 \exp\left\{C \left(
   \frac{m_\star^{3/2}}L
  +\frac{m_\star^3}{L^2}
  +\frac{m_\star^2}{L^2}
  +\frac{m_\star}{L}
  +\frac{m_\star^2}{hL}
 \right)\right\}
 \prod_{\sfe\in \Pi_h}
 \Big( \f{1}{1-\hb_{\sfe}\f{M_N}{L}} \Big)\notag \\ 
 & \quad  \le 
 \exp\left\{C \left(
   \frac{m_\star^{3/2}}L
  +\frac{m_\star^3}{L^2}
  +\frac{m_\star^2}{L^2}
  +\frac{m_\star}{L}
  +\frac{m_\star^2}{hL}
  +\f{h^2}{L}
 \right)\right\}
 \prod_{\sfe\in \Pi_h}
 \Big( \f{1}{1-\hb_{\sfe}} \Big).
\label{ub:eq-truncated-reference-bound}
\end{align} 
\end{lemma}
 \begin{proof}
Apply \eqref{ub:W-bound} with
\(\theta=C m_\star/L\), which is smaller than \(\theta_0\) for all
sufficiently large \(N\), since $m_\star = o(L)$. Combining this estimate with
\eqref{ub:bound-s2} and \eqref{ub:new-eq-exact-reference-sum}, and using
\((1+C/L)^m\le \exp\{Cm/L\}\) we obtain the first inequality above. To see the second inequality, notice that since $m_\star = o(L)$, we have $\f{M_N}{L} = \f{\log (N + m_\star)}{\log N} \le 1+\f{\log 2}{L} $. This contributes at most $e^{h^2/L}$ to the error.
 \end{proof}

\subsection{Summation over the diagram length}
\label{ub:new-sec-sum-m}
In this subsection we close the proof of theorem \ref{ub:thm-upper}. Recall \eqref{eq: ub-after ck}. In the last few subsections we have derived bounds for the summation when $m\le m_\star$. We now give an upper bound for the $m$-tail.

\begin{lemma}
 \label{ub:tail-cont}
 Define 
 \begin{align*}
  \mathsf{Tail}_{N,h}(\bz)
  & :=
  \sum_{m> m_\star}
  \sum_{\substack{\bx,\by\in\bbZ^{2m},\
                   \vec\sfe=(\sfe_1,\ldots,\sfe_m)\in\Pi_h^{\otimes m}\\
       1\le a_1\le b_1<a_2\le b_2<\cdots<a_m\le b_m\le N}}
  p_{a_1}(x_1-z^{\sfi_1}) p_{a_1}(x_1-z^{\sfj_1})
  \prod_{r=1}^m
  \gs_N^{\sfe_r} \, U_N^{\sfe_r}(b_r-a_r,y_r-x_r)\notag\\
  &\quad\times
  \prod_{r=1}^{m-1}
  p_{a_{r+1}-b_{\sfp(\sfi_{r+1})}}
   \bigl(x_{r+1}-y_{\sfp(\sfi_{r+1})}\bigr)
  p_{a_{r+1}-b_{\sfp(\sfj_{r+1})}}
   \bigl(x_{r+1}-y_{\sfp(\sfj_{r+1})}\bigr).
 \end{align*}
 where $\by_0 = \bz$. Assume \(h=o(\sqrt L)\)  and   \(h\to\infty\).  Choose
\begin{equation}\label{ub:eq-D0-choice}
 D_0>
 \frac{1}{2\log2}\log\frac{1}{1-\bar\gb}.
\end{equation}
and let \(m_\star=\lceil D_0h^2\rceil\).  Then there exists $c>0$ such that
\begin{equation}\label{ub:smallness}
 \sup_{\bz\in\bbZ^{2h}}\mathsf{Tail}_{N,h}(\bz)
 \le 
 e^{-ch^2}.
\end{equation} 
\end{lemma}

\begin{proof}
Define
$
 \widebar \gs_N:=e^{\pi \bar\gb/\log N}-1
$ 
and let $\widebar U_N$ be defined as in
\eqref{ub:eq-U-zero}--\eqref{ub:eq-U-def} by replacing
$\gs_N^{\sfe}$ with $\widebar \gs_N$.  Since $\hb_{\sfe}\le \f{1}{4}\bar\gb$, for all sufficiently
large $N$,
\begin{equation}\label{ub:eq-rho-half-short}
 \gs_N^{\sfe}\le\frac12\widebar \gs_N,
 \qquad
 U_N^{\sfe}(n,x)\le \widebar U_N(n,x).
\end{equation}

For an initial
configuration $\bz$, applying \eqref{ub:eq-rho-half-short} and denoting $\sfe_1 = \{\sfi_1,\sfj_1\}$ gives
\begin{align*}
 \mathsf{Tail}_{N,h}(\bz)
 &=
 \sum_{m> m_\star}
  \sum_{\substack{\bx,\by\in\bbZ^{2m},\
                   \vec\sfe=(\sfe_1,\ldots,\sfe_m)\in\Pi_h^{\otimes m}\\
       1\le a_1\le b_1<a_2\le b_2<\cdots<a_m\le b_m\le N}}
  p_{a_1}(x_1-z^{\sfi_1}) p_{a_1}(x_1-z^{\sfj_1})
  \prod_{r=1}^m
  \gs_N^{\sfe_r} \, U_N^{\sfe_r}(b_r-a_r,y_r-x_r)\notag\\
  &\quad\quad \times
  \prod_{r=1}^{m-1}
  p_{a_{r+1}-b_{\sfp(\sfi_{r+1})}}
   \bigl(x_{r+1}-y_{\sfp(\sfi_{r+1})}\bigr)
  p_{a_{r+1}-b_{\sfp(\sfj_{r+1})}}
   \bigl(x_{r+1}-y_{\sfp(\sfj_{r+1})}\bigr)
   \notag\\[-2mm]
 &\le\sum_{m>m_\star}2^{-m}
 \sum_{\substack{\bx,\by\in\bbZ^{2m},\,
                  \vec\sfe\in\Pi_h^{\otimes m}\\
 1\le a_1\le b_1<a_2\le b_2<\cdots<a_m\le b_m\le N}}
p_{a_1}(x_1-z^{\sfi_1}) p_{a_1}(x_1-z^{\sfj_1})
 \prod_{r=1}^m
   \widebar \gs_N  \,
   \widebar U_N (b_r-a_r,y_r-x_r)\notag\\[-2mm]
  &\quad\quad \times
  \prod_{r=1}^{m-1}
  p_{a_{r+1}-b_{\sfp(\sfi_{r+1})}}
   \bigl(x_{r+1}-y_{\sfp(\sfi_{r+1})}\bigr)
  p_{a_{r+1}-b_{\sfp(\sfj_{r+1})}}
   \bigl(x_{r+1}-y_{\sfp(\sfj_{r+1})}\bigr)
   \notag\\
 &\le2^{-m_\star}
 \cM^h_{N,\bar \gb}(\bz).
\end{align*}
By Lemma~\ref{ub:lem-coarse-pair-CZ}, we have 
$$
 \sup_{\bz} \cM^h_{N,\bar \gb} (\bz)
 \le
 \Big(\f{1}{1-\bar \gb}\Big)^{\binom h2\,\big(1+o(1)\big)}.
$$
Consequently,
\begin{align*}
  \sup_{\bz}\mathsf{Tail}_{N,h}(\bz)
 &\le
  \exp\Big\{ -m_\star\log2
 +\binom h2\log\frac1{1-\bar\gb}\,\big(1+o(1)\big)\Big\}\\
 &\le
 \exp\Big\{ -h^2\big(
 D_0\log2-\frac12\log\frac1{1-\bar\gb}+o(1)
 \big)\Big\}.
\end{align*}
The choice \eqref{ub:eq-D0-choice} proves \eqref{ub:smallness}.
\end{proof}

  \begin{proof}[Proof of Lemma \ref{ub:upper-bound}]
First suppose that \(h=h_N\to\infty\).  With
\(m_\star=\lceil D_0h^2\rceil\), the assumption \(h=o(\sqrt L)\) gives
\(m_\star=o(L)\), so all preceding estimates in the previous sections apply.  By
Lemma~\ref{ub:tail-cont} and \eqref{ub:eq-truncated-reference-bound}, we have
\begin{align}
\sup_{\bz}\cM^{\mathrm{pair},h}_{N,\bhb}(\bz)
 &\le
 \prod_{\sfe\in \Pi_h} \Big(\f{1}{1-\hb_{\sfe}}\Big) 
 \times \exp\Big\{ C \Big(
   \frac{m_\star^{3/2}}L
  +\frac{m_\star^3}{L^2}
  +\frac{m_\star^2}{L^2}
  +\frac{m_\star}{L}
  +\frac{m_\star^2}{hL}
  +\f{h^2}{L}
 \Big)\Big\}
 +e^{-ch^2}\notag \\ 
 & \le 
 \prod_{\sfe\in \Pi_h} \Big(\f{1}{1-\hb_{\sfe}}\Big) 
 \times \exp\Big\{ C \Big(
   \frac{h^3}L
  +\frac{h^6}{L^2}
 \Big)\Big\}
 + e^{-ch^2},
\label{ub:eq-before-final-substitution}
\end{align}
where we have used $m_\star = O(h^2)$.

If \(h\) remains bounded, we simply cite {\cite[Theorem~1.1]{LygkonisZygouras2024Multivariate}}.
Combining the bounded and unbounded cases proves the theorem for every
sequence \(h=o(\sqrt L)\).  
\end{proof}

\section{Chaos expansion and a renewal measure}
\label{lb:sec:change-measure}

We now turn to the proof of the lower bounds in
Theorem~\ref{lb:thm-main}. We first provide a roadmap of the proof. In the first part of section \ref{lb:sec:change-measure}, we start from expanding the moment generating function, and write the expansion in a product renewal measure. Then we wish to study the correlation structure of the collision times under this renewal measure. In the second part of section \ref{lb:sec:change-measure}, we collect some useful estimates about the renewal measure. In subsection \ref{lb:sec:good-set}, we introduce a good event with high probability in the renewal measure. In subsection \ref{lb:sec:comparison}, we show that events constrained on the good set can be well-approximated by heat kernels. In subsection \ref{lb:sec:wedge-gain}, we display some linear algebra facts. By the Hadamard-Fischer inequality we show asymptotic non-negative correlation of the collision events over different pairs. We also identify an event that produces a non-negligible correlation under the renewal measure and show in the final subsection that this event has positive probability whenever $h\ge CL^{1/3}$ for some $C>0$.

Throughout Sections~\ref{lb:sec:change-measure}--
\ref{lb:sec:renewal}, the symbols $c,C\in(0,\infty)$ denote constants
which depend only on $\hbn$ and $\hbm$ and may change from line to line.
{A final constant of the form $C_{\cdot}$ is indexed by the number of the lemma/proposition in which it appears; for example, $C_{\ref*{lb:prop:G}}$ denotes a final constant in Lemma~\ref{lb:prop:G}.}

We will start with expanding the moment generating function. Let $I=\{n_1<\cdots<n_k\}\subseteq\{1,\cdots, N\}$ and fix $\sfe = \{\sfi,\sfj\}\in \Pi_h$. We define the collision event of the pair $\sfe$ according to $I$ to be:
\begin{equation*}
 \cE_I^{\sfe}
 :=\{S_{n_r}^{\sfi}=S_{n_r}^{\sfj},\ 1\le r\le k\}.
\end{equation*}
For $\sfe\in \Pi_h$, take 
\begin{align*}
& I_\sfe = \{n_1^\sfe<\cdots <n_{m_\sfe}^\sfe\} \subseteq \{1,\cdots ,N\},
\quad \quad 
m_\sfe:= |I_\sfe|,\\
& \bI := \union_{\sfe\in \Pi_h} I_\sfe 
      = \big\{ n^{\sfe_1}_1 \le \cdots \le n^{\sfe_m}_m \big\},
\quad \quad 
m:= \sum_{\sfe\in \Pi_h}m_\sfe = |\bI|.
\end{align*}
Here the elements in $\bI$ are reordered so that they are aligned in increasing order; this ordering might not be unique. The superscript of $n^{\sfe_k}_k$ implies $n^{\sfe_k}_k\in I_{\sfe_k}$, that it comes from the time set of the corresponding pair. By \eqref{def: collision moment}, 
\begin{equation}\label{lb:eq:exact-M}
 \cM^h_{N,\bhb}
 =\sum_{\boldsymbol I=(I_{\sfe})_{\sfe\in\Pi_h}}
 \Big(\prod_{\sfe\in\Pi_h}(\gs_N^{\sfe})^{|I_{\sfe}|}\Big)\,\,
 \sfP^{\otimes h}\Big(
   \bigcap_{\sfe\in\Pi_h}\mathcal E_{I_{\sfe}}^{\sfe}
 \Big).
\end{equation}
Observe that for each pair $\sfe\in \Pi_h$ and $I_{\sfe} = \{n_1<\cdots <n_k\}$, we have
$
  \sfP ( \cE^\sfe_{I_\sfe}  )
  =
  \prod_{r=1}^k p_{2(n_r-n_{r-1})}(0)
$. 
We then have
\begin{align}\label{lb:uncor}
  \sum_{\boldsymbol I=(I_{\sfe})_{\sfe\in\Pi_h}}
 \Big(\prod_{\sfe\in\Pi_h}(\gs_N^{\sfe})^{|I_{\sfe}|}\Big)\,\,
 \prod_{\sfe\in\Pi_h}
 \sfP^{\otimes h}(
   \mathcal E_{I_{\sfe}}^{\sfe})
  =
  \prod_{\sfe\in \Pi_h} \sum_{0\le n\le N} U^\sfe_N(n)
  \ge 
  e^{\f{-Ch^2}{L}}\prod_{\sfe\in \Pi_h} \Big(\f{1}{1-\hb_{\sfe}}\Big),
\end{align}
where the equality above follows from definition of $U^\sfe_N$, and the inequality follows from \eqref{lb:asymp-U}. By comparing \eqref{lb:eq:exact-M} and \eqref{lb:uncor} and recalling 
\begin{align*}
  h \le \sqrt{\f{c_0 L}{\log L}},
\end{align*}
one observes the key is to compare $\sfP^{\otimes h}\Big( \inter_{\sfe\in \Pi_h} \cE^\sfe_{I_\sfe}  \Big)$ with $\prod_{\sfe\in\Pi_h}\sfP^{\otimes h}(\cE^\sfe_{I_\sfe})$ in a weighted sum. For each $\sfe\in \Pi_h$, we define the abbreviation:
\begin{align}
  \sfU^\sfe_{N,M} 
  :=
  \sum_{n=0}^M U_N^\sfe(n),
\end{align}
and set the renewal measure associated to this pair to be
\begin{align*}
  \nu_{\sfe,N}(I)
  :=
  \f{\big(\gs_N^\sfe\big)^{|I|} \sfP^{\otimes h}(\cE^\sfe_I) }{\sfU^\sfe_{N,N}}
\end{align*}
for $I\subseteq \{1,\cdots, N\}$.  Moreover, define the product renewal measure
\begin{equation*}
 \boldsymbol\nu_{\bhb,N}
 :=\bigotimes_{\sfe\in\Pi_h}\nu_{\sfe,N}.
\end{equation*}
Therefore, comparing \eqref{lb:eq:exact-M} and \eqref{lb:uncor}, we have
\begin{align}\label{lb:start}
  \cM^h_{N,\bhb}
 \ge 
 e^{-\f{Ch^2}{L}}\,
 \E_{\boldsymbol\nu_{\bhb,N}}\bigg[
\frac{
 \sfP^{\otimes h}(\bigcap_{\sfe}\mathcal E_{I_{\sfe}}^{\sfe})
 }{
 \prod_{\sfe}\sfP^{\otimes h}(\mathcal E_{I_{\sfe}}^{\sfe})
 }\bigg]
 \prod_{\sfe\in \Pi_h} \Big( \f{1}{1-\hb_\sfe} \Big).
\end{align}

The remaining task is therefore to control
the above expectation under $\boldsymbol \nu_{\bhb,N}$, uniformly in
the growing number of pairs.  We will establish a universal lower
bound and then identify the first strictly positive correction generated
by two pairs with one common index in section \ref{lb:sec:renewal}. For now, we derive some properties of the renewal measures. 

\subsection{Preliminary estimates of the renewal measure}
Let 
\begin{align*}
  R_M:= \sum_{n=1}^M p_{2n}(0). 
\end{align*}
By $p_{2n}(0) = 4^{-2n} \binom{2n}{n}^2  = \f{1}{\pi n}+ O(n^{-2})$ and Stirling's approximation, we have that $R_M = \f{\log M}{\pi}+ O(1)$ and $\gs^\sfe_N R_N =\hb_\sfe+O(1/L)<1$ uniformly over $N$ and $\sfe$. 

The first estimate shows that under the renewal measure, time sets $\bI$ that are too long occur with low probability. 
\begin{lemma}\label{lb:lem:count}
  There exists $C_{\ref*{lb:lem:count}}$, such that if we define
  \begin{align}\label{lb:def-m0}
    m_0 
    :=
    \lceil C_{\ref*{lb:lem:count}} h^2 \log L \rceil,
  \end{align}
  then
  \begin{align}
    \boldsymbol \nu_{\bhb,N} (|\bI|>m_0)
    \le 
    L^{-2h^2}.
  \end{align}
\end{lemma}
\begin{proof}

As $\mathsf U_{N,N}^{\sfe}\ge1$, 
\begin{align*}
 \nu_{\sfe,N}\bigl(|I_{\sfe}|=k\bigr)
 &=\frac{(\gs_N^{\sfe})^k}{\mathsf U_{N,N}^{\sfe}}
 \sum_{\substack{d_1,\ldots,d_k\ge1\\
                   d_1+\cdots+d_k\le N}}
 \prod_{r=1}^k p_{2d_r}(0)
 \le(\gs_N^{\sfe}R_N)^k
 \le \hb_\sfe^k.
\end{align*}
Choose ${c}>0$ so that
$\hb_\sfe \, e^c<1$.  Then
$$
 \sup_{N,\sfe}\E_{\nu_{\sfe,N}}
 \left[e^{c \, |I_{\sfe}|}\right]
 \le\sum_{k\ge0}
   \hb_\sfe^k \, e^{ck}
 =\f{1}{1-\hb_\sfe e^c}
 <\infty.
$$
As $\boldsymbol \nu_{\bhb,N}$ is defined as a product measure, the family $\{I_{\sfe}\}_{\sfe\in \Pi_h}$ is independent under $\boldsymbol \nu_{\bhb,N}$. Using Chernoff's bound, we get
\begin{align*}
 \boldsymbol\nu_{\bhb,N}
 \left(m>m_0\right)
 &\le
 e^{- c m_0}
 \prod_{\sfe\in\Pi_h}
 \E_{\nu_{\sfe,N}}
 \big[e^{{c}\, |I_{\sfe}|} \big]\\
 &\le
 \exp\left\{
  -c\,C_{\ref*{lb:lem:count}}h^2 \log L
  +\binom{h}{2}\log {\f{1}{1-\hb_\sfe e^c}}
 \right\}
 \le e^{-2h^2 \log L}=L^{-2h^2},
\end{align*}
if we choose $C_{\ref*{lb:lem:count}}\ge \f{4+2\log \f{1}{1-\hb_\sfe e^c}}{c}$.
\end{proof}

In the next lemma, we estimate the probability of a time set $I_\sfe$ containing a certain time $n$, and the probability that the first time of a time set is $n$. These estimates will be used repeatedly in section \ref{lb:sec:renewal}.
\begin{lemma}\label{lb:lem:renewal}
  Let $I\subseteq \{1,\cdots ,N\}$. Define the minimum element of $I$:
  \begin{align*}
    T(I) = \min I.
  \end{align*}
  There exists $0<c_{\ref*{lb:lem:renewal}}<C_{\ref*{lb:lem:renewal}}$ such that, uniformly in $\sfe\in \Pi_h$,
  \begin{align}
  \nu_{\sfe,N} (n\in I_\sfe) \le \f{C_{\ref*{lb:lem:renewal}} }{Ln},
  \quad \text{ for all $1\le n\le N$};\label{lb:n-in-I}\\
  \f{c_{\ref*{lb:lem:renewal}} }{Ln}
  \le 
  \nu_{\sfe,N}\big(T(I_\sfe) = n\big)
  \le 
  \f{C_{\ref*{lb:lem:renewal}} }{Ln},
  \quad \text{ for all $1\le n\le \f{N}{2}$}.\label{lb:T=n}
  \end{align} 
\end{lemma}
\begin{proof}
We first prove the following estimate. For $k\ge 1$ and $n\le N$,
\begin{equation}\label{lb:eq:one-big-gap}
 \sum_{0=n_0<n_1<\cdots<n_k=n}
 \prod_{r=1}^k p_{2(n_r-n_{r-1})}(0)
 \le C k p_{2n}(0) \, R_n^{k-1}.
\end{equation}
To see this, write $d_r:=n_r-n_{r-1}$. Then we have $\frac{d_jp_{2d_j}(0)}{np_{2n}(0)}\le C$ for $1\le j\le k$. Therefore,
\begin{align*}
 \prod_{r=1}^k p_{2d_r}(0)
 &=p_{2n}(0)\sum_{j=1}^k
   \frac{d_jp_{2d_j}(0)}{np_{2n}(0)}
   \prod_{r\ne j}p_{2d_r}(0)
   \le Cp_{2n}(0)\sum_{j=1}^k
   \prod_{r\ne j}p_{2d_r}(0).
\end{align*}
Summing over $d_1+\cdots+d_k=n$ and, for each fixed $j$, dropping the constraint on the remaining $k-1$ gaps yields
\begin{align*}
 &\sum_{\substack{d_1,\ldots,d_k\ge1\\d_1+\cdots+d_k=n}}
 \prod_{r=1}^k p_{2d_r}(0)
 \le Cp_{2n}(0)\sum_{j=1}^k
 \sum_{\substack{d_r\ge1,\ r\ne j\\
                   \sum_{r\ne j}d_r\le n-1}}
 \prod_{r\ne j}p_{2d_r}(0)
 \le Ck p_{2n}(0) R_n^{k-1}.
\end{align*}

Now we show \eqref{lb:n-in-I}. Assuming $n\in I_\sfe$, by decomposing $U_N^\sfe$, we have that
\begin{align*}
  \nu_{\sfe,N}(n\in I_\sfe)
  =
  \sum_{\substack{I_\sfe\subseteq \{1,\cdots, N\} \\ n\in I_\sfe }}
  \f{\big(\gs_N^\sfe\big)^{|I_\sfe|} \sfP^{\otimes h}(\cE^\sfe_{I_\sfe}) }{\sfU^\sfe_{N,N}}
  = 
  \f{ U_N^\sfe(n )\, \sfU^\sfe_{N,N-n} }{\sfU^\sfe_{N,N}}
  \le 
  U^\sfe_N(n),
\end{align*}
where in the last inequality we have used that $\sfU^\sfe_{N,M}$ is increasing in $M$. By definition of $U_N^\sfe(n)$ and \eqref{lb:eq:one-big-gap}, we continue to bound the above:
\begin{align*}
U^\sfe_N(n)
\le 
C \gs^\sfe_N \, p_{2n}(0)
\sum_{k\ge 1}
k\big( \gs_N^\sfe R_n \big)^{k-1}
\le 
C\gs_N^\sfe p_{2n}(0)
\le 
\f{C\hb_\sfe}{Ln},
\end{align*}
where we have used $\gs^\sfe_N = e^{\f{\pi \hb_\sfe}{L}} -1$ and \eqref{lb:n-in-I} is proven. Now suppose that the first time in $I$ is $n$. Then a similar decomposition of $U_N^\sfe$ as above gives
\begin{equation*}
 \nu_{\sfe,N}\bigl(T(I_{\sfe})=n\bigr)
 =\frac{\gs_N^{\sfe}p_{2n}(0)
        \mathsf U_{N,N-n}^{\sfe}}
       {\mathsf U_{N,N}^{\sfe}}.
\end{equation*}
For $n\le \f{N}{2}$, \eqref{ub:eq-renewal-cumulative} and monotonicity imply that
$
 1\le\mathsf U_{N,N-n}^{\sfe}
 \le\mathsf U_{N,N}^{\sfe}\le C
$. 
Together with $\gs^\sfe_N = \f{\pi \hb_\sfe}{L}+O(L^{-2})$ and $p_{2n}(0) = \f{1}{\pi n}+ O(n^{-2})$, we obtain \eqref{lb:T=n}. 
\end{proof}

\begin{remark}
Estimate \eqref{lb:eq:one-big-gap} is the subcritical version of the sharp
local estimate in \cite[Proposition~1.5]{CSZDickman}.  The latter contains an
additional super-exponential factor in the renewal order.  Here as we have $\gs^\sfe_N R_N<1$ uniformly over $N$, the presented elementary estimate is sufficient already.
\end{remark}

The next lemma proves that a time set $I_\sfe$ cannot contain two times that are too close.
\begin{lemma}\label{lb:lem:short}
There exists ${C_{\ref*{lb:lem:short}}}>0$ such that, for every
$\sfe\in\Pi_h$ and every integer $M\ge2$,
\begin{equation}\label{lb:eq:short-gap}
 \nu_{\sfe,N}\left(
   T(I_{\sfe})\le M\ \text{ or }
   \min_{2\le r\le |I_{\sfe}|}
   (n_r^{\sfe}-n_{r-1}^{\sfe})\le M
 \right)
 \le {C_{\ref*{lb:lem:short}}}\frac{\log M}{L},
\end{equation}
where the minimum is understood to be $+\infty$ when
$|I_{\sfe}|\le1$.
\end{lemma}

\begin{proof}

Fix $\sfe$, set $n_0^{\sfe}:=0$, and let
$d_r:=n_r^{\sfe}-n_{r-1}^{\sfe}$.  The event in
\eqref{lb:eq:short-gap} equals
$
 \left\{|I_\sfe|\ge 1,\ \min_{1\le r\le |I_\sfe|}d_r\le M\right\}
$.
Using
$\ind_{\{\min_r d_r\le M\}}\le\sum_{j=1}^k\ind_{\{d_j\le M\}}$
and then dropping the constraint on the remaining gaps, we obtain
\begin{align*}
 \nu_{\sfe,N}\Big(|I_\sfe|\ge1,\ \min_{1\le r\le |I_\sfe|}d_r\le M\Big)
 & \le
 \frac{1}{\mathsf U_{N,N}^{\sfe}}
 \sum_{k\ge1}(\gs_N^{\sfe})^k
 \sum_{j=1}^k
 \sum_{\substack{d_1,\ldots,d_k\ge1\\
                   d_1+\cdots+d_k\le N\\d_j\le M}}
 \prod_{r=1}^k p_{2d_r}(0)\\
 &\quad\le
 \frac{1}{\mathsf U_{N,N}^{\sfe}}
 \sum_{k\ge1}(\gs_N^{\sfe})^k kR_M R_N^{k-1}
 =
 \frac{\gs_N^{\sfe}R_M}
 {\mathsf U_{N,N}^{\sfe}(1-\gs_N^{\sfe}R_N)^2}.
\end{align*}
Now use $\mathsf U_{N,N}^{\sfe}\ge1$,
$R_M\le C\log M$, $\gs_N^{\sfe}\le C/L$, and $\gs_N^{\sfe} R_N<1$ uniformly over $N$ and $\sfe$ to finish.
\end{proof}

\section{Independence and dependence lower bound}
\label{lb:sec:renewal}
In this section we show Theorem \ref{lb:thm-main}. 
\subsection{Good set}\label{lb:sec:good-set}
In this subsection we introduce a good set, such that this set possesses high probability under the renewal measure, and upon the set we can control the errors from applying a local limit theorem.

We now define the good set. The good set $\mathcal G_N$ is the set of all full configurations
$\boldsymbol I=\union_{\sfe\in \Pi_h}I_{\sfe}$ satisfying the following three
conditions:
\begin{description}
 \item[\normalfont\textup{(G1)}]
 $m\le m_0 = \lceil C_{\ref*{lb:lem:count}}h^2\log L \rceil$;
 \item[\normalfont\textup{(G2)}]
 $n\ge  L^4$ for every labelled mark $(\sfe,n)$ with
 $n\in I_{\sfe}$;
 \item[\normalfont\textup{(G3)}]
 $|n-t|\ge L^4$ whenever
 $(\sfe,n)\ne(\sff,t)$, with
 $n\in I_{\sfe}$ and $t\in I_{\sff}$.
\end{description}
$\cG_N$ is typical under the product renewal measure.
\begin{lemma}\label{lb:prop:G}
There exists $C_{\ref*{lb:prop:G}}>0$, independent of $c_0$, such that for $h\le \sqrt{\f{c_0 L}{\log L}}$,
\begin{equation}\label{lb:eq:G-loss}
 \boldsymbol\nu_{\bhb,N}(\mathcal G_N^c)
 \le\varepsilon_{\ref*{lb:prop:G}}(N,h)
 :={C_{\ref*{lb:prop:G}}}\left(
   \frac{h^2\log L}{L}
   +\frac{h^4(\log L)^2}{L^2}
 \right)+L^{-2h^2}.
\end{equation}
\end{lemma}

\begin{proof}

The probability that \textup{(G1)} fails is at most $L^{-2h^2}$ by
Lemma~\ref{lb:lem:count}.  Fix the time set of one pair
$I_{\sfe}=\{n_1^{\sfe}<\cdots<n_{m_{\sfe}}^{\sfe}\}$. We next consider the failure probability of \textup{(G2)}, or the failure of \textup{(G3)} when the two close time marks belong to the time set $I_\sfe$ of the same pair. That is, $|n-t|<L^4$ for some $n\neq t$ and $n,t\in I_{\sfe}$.
By lemma~\ref{lb:lem:short} with $M=L^4$, we have
\begin{equation}\label{lb:eq:within-edge-loss}
 \boldsymbol\nu_{\bhb,N}
 \bigl(\text{failure of \textup{(G2)}, or failure of
 \textup{(G3)} within the time set of one pair}\bigr)
 \le
 C h^2\frac{\log L^4}{L}
 \le \frac{Ch^2\log L}{L}.
\end{equation}

Now we consider failure probability of \textup{(G3)} when the time marks come from time sets of different pairs. Fix two distinct edges $\sfe\ne\sff$.  By 
independence of $I_{\sfe}$ and $I_{\sff}$ under
$\boldsymbol\nu_{\bhb,N}$, and \eqref{lb:n-in-I},
\begin{align}
 &\boldsymbol\nu_{\bhb,N}\left(
   \exists\,n\in I_{\sfe},\ t\in I_{\sff}:
   |n-t|<L^4
 \right)
 \le
 \sum_{\substack{1\le n,t\le N\\|n-t|<L^4}}
 \nu_{\sfe,N}(n\in I_\sfe)
 \nu_{\sff,N}(t\in I_\sff)
 \le\frac{C}{L^2}
 \sum_{n=1}^N\frac1n
 \sum_{\substack{1\le t\le N\\|t-n|<L^4}}\frac1t.\label{lb:dif-pair}
\end{align}
For $n\le2L^4$, the above admits a bound
\begin{align}
  \f{C}{L^2} \sum_{1\le n\le 2L^4} \f{1}{n} 
  \sum_{1\le t\le 3L^4}\f{1}{t}
  \le 
  C\big(\f{\log L}{L}\big)^2.\label{lb:dif-pair1}
\end{align}
For $n>2L^4$ and
$|t-n|<L^4$, one has $t>n/2$, and there are at most
$2L^4$ such integers $t$.  Hence
\begin{align}
 \f{C}{L^2}\sum_{2L^4\le n\le N} \f{1}{n} 
 \sum_{\substack{1\le t\le N \\ |t-n|\le L^4}} \f{1}{t}
 \le 
 CL^2 \sum_{2L^4\le n\le N} \f{1}{n^2} 
 \le 
 \f{C}{L^2}.\label{lb:dif-pair2}
\end{align}
Combining \eqref{lb:dif-pair1} and \eqref{lb:dif-pair2} we obtain an upper bound of the form $C\big(\f{\log L}{L}\big)^2$ for \eqref{lb:dif-pair}. 
Summing over $\binom{h}{2}^2-1$ such pairs of pairs and combining with
\eqref{lb:eq:within-edge-loss} and the loss from \textup{(G1)} proves
\eqref{lb:eq:G-loss} after choosing sufficiently large ${C_{\ref*{lb:prop:G}}}$.
\end{proof}

\subsection{Uniform local limit comparison}\label{lb:sec:comparison}

Fix $\boldsymbol I = \{n^{\sfe_1}_1,...,n^{\sfe_m}_m\} = \union_{\sfe\in \Pi_h}I_{\sfe}\in\mathcal G_N$.  By \textup{(G3)}, all time marks
are distinct and therefore admit a unique ordering 
$
 n_1^{\sfe_1}<\cdots <n_m^{\sfe_m}
$ 
where for $1\le k\le m$, $\sfe_k$ is the pair such that $n_k^{\sfe_k}\in I_{\sfe_k}$. Below in abbreviation, sometimes we write $n_k = n_k^{\sfe_k}$.

Recall that $\bS = (S^1,...,S^h)$ is a length $h$ vector whose entries are planar simple symmetric random walks. As $\bI$ is the collision time set of $\bS$, each $n^{\sfe_k}_k\in \bI$ for $\sfe_k = \{\sfi_k,\sfj_k\}$ corresponds to $S^{\sfi_k}_{n_k} = S^{\sfj_k}_{n_k}$. By introducing the rotation $R(x_1,x_2) = \f{1}{\sqrt{2}}(x_1+x_2,x_1-x_2)$, we define the rotated random walk
\begin{align*}
  \tS^\sfi_n:= R(S^\sfi_n) = (\tS^{\sfi,1}_n,\tS^{\sfi,2}_n).
\end{align*}
Fix one coordinate $a\in \{1,2\}$. For every pair-time combination $(\sfe_r,n_r)$, with $\sfe_r = \{\sfi_r,\sfj_r\}$, $\sfi_r<\sfj_r$, we define the difference
\begin{align*}
Y^{(a)}_r:= \tS^{\sfi_r,a}_{n_r} - \tS^{\sfj_r,a}_{n_r}
\end{align*}
Collect $1\le r\le m$ to assemble a vector $Y^{(a)}$. We take $\Gamma = \cov(Y^{(a)})$ to be the covariance matrix. Its entries are given by:
\begin{equation}\label{lb:eq:Gamma}
 \Gamma_{rs}
 :=\frac12\langle\alpha_{\sfe_r},\alpha_{\sfe_s}\rangle
 (n_r\wedge n_s),
 \qquad 1\le r,s\le m.
\end{equation}
The vectors $\ga_{\sfe}$ are defined as follows. Let $\{\mathbf o_i\}_{1\le i\le h}$ be the canonical orthonormal basis of
$\bbR^h$.  For $\Pi_h\ni \sfe=\{\sfi,\sfj\}$ with $\sfi<\sfj$, set
$
 \alpha_{\sfe}:=\mathbf o_\sfi-\mathbf o_\sfj
$. 
For $m\ge1$, define the minimal gap and smallest eigenvalue of $\Gamma$ to be:
$$
 \Delta:=\min_{1\le r\le m}(n_r-n_{r-1}),
 \qquad
 \lambda:=\lambda_{\min}(\Gamma).
$$

\begin{proposition}[Dimension-dependent local limit theorem]
\label{lb:prop:LLT}
Recall that $\sfP^{\otimes h}$ is the law of $\bS$. Take a total time set $\bI\subseteq \{1,\cdots ,N\}$ with $\Delta>0$. Then there exists ${C_{\ref*{lb:prop:LLT}}},{c_{\ref*{lb:prop:LLT}}^{(1)}},{c_{\ref*{lb:prop:LLT}}^{(2)}}>0$ such that, if
$
 \lambda\ge {c_{\ref*{lb:prop:LLT}}^{(2)}}m^3
$, 
then
\begin{equation}\label{lb:eq:joint-LLT}
 \sfP^{\otimes h}\Big(
   \bigcap_{\sfe\in\Pi_h}\mathcal E_{I_{\sfe}}^{\sfe}
 \Big)
 =\frac{\pi^{-m}}{\det\Gamma}(1+\eps_{\ref*{lb:prop:LLT}}),
\end{equation}
where, for a standard Gaussian vector $G_{2m}$ in $\bbR^{2m}$,
\begin{align}
 |\eps_{\ref*{lb:prop:LLT}}|
 \le {C_{\ref*{lb:prop:LLT}}}\bigg[&
 \sqrt{\frac{m^3}{\lambda}}
 +\bbP\left(
   |G_{2m}|>{c_{\ref*{lb:prop:LLT}}^{(1)}}\left(\frac{\lambda}{m}\right)^{1/4}
 \right)
 +\det\Gamma\left(\frac{{C_{\ref*{lb:prop:LLT}}}}{\Delta}\right)^m
 \bbP\left(
   |G_{2m}|>{c_{\ref*{lb:prop:LLT}}^{(1)}}\sqrt{\frac{\Delta}{m}}
 \right)
 \bigg].
 \label{lb:eq:joint-LLT-error}
\end{align}
Furthermore, we have that
\begin{equation}\label{lb:eq:eig-det}
 \lambda\ge\frac{\Delta}{4},
 \qquad
 \det\Gamma\le N^m.
\end{equation}
\end{proposition}
The proof of this local limit theorem is given in Appendix~\ref{lb:app:LLT}.

For a pair $\sfe$, let $\Gamma_{\sfe}$ be the submatrix of
$\Gamma$ concerning only collision times in $I_{\sfe}$. Precisely, if
$I_{\sfe}=\{n_1^{\sfe}<\cdots<n_{m_{\sfe}}^{\sfe}\}$, then
\begin{equation}\label{lb:eq:Gamma-edge-explicit}
 (\Gamma_{\sfe})_{r,s}
 =\frac12\langle\alpha_{\sfe},\alpha_{\sfe}\rangle
   (n_r^{\sfe}\wedge n_s^{\sfe})
 =n_r^{\sfe}\wedge n_s^{\sfe},
 \qquad 1\le r,s\le m_{\sfe}.
\end{equation}
We use the convention $\det\Gamma_{\sfe}=1$ when
$I_{\sfe}=\varnothing$.

\begin{proposition}[Correlation inequality]\label{lb:prop:comparison}
Assume $h\le \sqrt{\f{c_0 L}{\log L}}$.  Uniformly for
$\boldsymbol I\in\mathcal G_N$, there exists $c_{\ref*{lb:prop:comparison}}, C_{\ref*{lb:prop:comparison}}>0$ independent of $c_0$ such that
\begin{equation}\label{lb:eq:comparison}
 \frac{
 \sfP^{\otimes h}(\bigcap_{\sfe}\mathcal E_{I_{\sfe}}^{\sfe})
 }{
 \prod_{\sfe}\sfP^{\otimes h}(\mathcal E_{I_{\sfe}}^{\sfe})
 }
 \ge\bigl(1-\eps_{\ref*{lb:prop:comparison}}(N,h)\bigr)
 \frac{\prod_{\sfe\in\Pi_h}\det\Gamma_{\sfe}}
      {\det\Gamma}
\end{equation}
for all sufficiently large $N$, where
\begin{align}
 \varepsilon_{\ref*{lb:prop:comparison}}(N,h)
 :={C_{\ref*{lb:prop:comparison}}}\Big(
   \frac{(h^2\log L)^{3/2}}{L^2}
   +\frac{h^2\log L}{L^4}
 \Big)
 &+\exp\Big\{-\frac{{c_{\ref*{lb:prop:comparison}}}L^2}
   {\sqrt{h^2\log L}}\Big\}
  +\exp\Big\{-\frac{{c_{\ref*{lb:prop:comparison}}}L^4}
   {h^2\log L}\Big\}.
 \label{lb:eq:error-loc}
\end{align}
\end{proposition}

\begin{proof}

The assertion is immediate when $m = |\bI| =0$, with all determinants being
one.  Suppose that $m\ge1$.  On $\mathcal G_N$, \textup{(G1)}, \textup{(G2)}, \textup{(G3)} and \eqref{lb:eq:eig-det} and the assumption on $h$ give
\begin{equation}\label{lb:eq:lambda-good}
 \Delta \ge L^4,
 \qquad
 \lambda\ge\frac{\Delta}{4}\ge\frac{L^4}{4},
 \qquad
 m\le C_{\ref*{lb:lem:count}} h^2\log L\le C L.
\end{equation}
In particular, the assumption $\gl \ge c^{(2)}_{\ref*{lb:prop:LLT}} m^3$ of Proposition \ref{lb:prop:LLT} holds for all sufficiently large $N$, since for some $c>0$,
\begin{align}\label{lb:gl-m}
 \frac{\lambda}{m^3}
 \ge cL\longrightarrow\infty.
\end{align}
So we apply Proposition \ref{lb:prop:LLT} on $\sfP^{\otimes h}(\bigcap_{\sfe}\mathcal E_{I_{\sfe}}^{\sfe})$. The first term in \eqref{lb:eq:joint-LLT-error} satisfies
\begin{equation}\label{lb:cor-err1}
 {C_{\ref*{lb:prop:LLT}}}\sqrt{\frac{m^3}{\lambda}}
 \le C\frac{(h^2\log L)^{3/2}}{L^2}.
\end{equation}
For a standard $d$-dimensional Gaussian vector $G_d$,
\begin{equation}\label{lb:eq:gauss-tail}
 \sfP(|G_d|>x)\le e^{-x^2/8},
 \qquad \text{ for } x\ge2\sqrt d.
\end{equation}
Notice that \eqref{lb:gl-m} implies
$
 {c_{\ref*{lb:prop:LLT}}^{(1)}}\left(\frac{\lambda}{m}\right)^{1/4}
 \ge2\sqrt{2m}
$ for sufficiently large $N$. 
So \eqref{lb:eq:gauss-tail} gives
\begin{align}\label{lb:cor-err2}
 &\sfP\Big(
   |G_{2m}|>{c_{\ref*{lb:prop:LLT}}^{(1)}}\big(\frac{\lambda}{m}\big)^{1/4}
 \Big)
 \le
 \exp\Big\{-c\sqrt{\frac{\lambda}{m}}\Big\}
 \le
 \exp\Big\{-\frac{cL^2}{\sqrt{h^2\log L}}\Big\}.
\end{align}
For the third term in \eqref{lb:eq:joint-LLT-error}, by \eqref{lb:eq:lambda-good},
$
 \frac{{c_{\ref*{lb:prop:LLT}}^{(1)}}\sqrt{\Delta/m}}{2\sqrt{2m}}
 \ge\frac{cL^2}{h^2\log L}\longrightarrow\infty
$. 
Thus ${c_{\ref*{lb:prop:LLT}}^{(1)}}\sqrt{\Delta/m}\ge2\sqrt{2m}$ for all sufficiently large $N$, and we are allowed to apply 
\eqref{lb:eq:gauss-tail}.  Together with \eqref{lb:eq:eig-det},
\begin{align}
 \det\Gamma\Big(\frac{{C_{\ref*{lb:prop:LLT}}}}{\Delta}\Big)^m
 \sfP\Big(|G_{2m}|>{c_{\ref*{lb:prop:LLT}}^{(1)}}\sqrt{\frac{\Delta}{m}}\Big)
 & \le
 N^m\big(\frac{{C_{\ref*{lb:prop:LLT}}}}{\Delta}\big)^m
 \exp\big\{-c\frac{\Delta}{m}\big\}
 =
 \exp\Big\{
   mL-m\log\frac{\Delta}{{C_{\ref*{lb:prop:LLT}}}}
   -c\frac{\Delta}{m}
 \Big\}\notag \\
 & \le
 \exp\Big\{mL-c\frac{L^4}{m}\Big\}
 \le 
 \exp\Big\{ -\f{cL^4}{h^2\log L} \Big\},
 \label{lb:eq:LLT-tail-two-a}
\end{align}
where we have used \eqref{lb:eq:lambda-good}.

Next we apply Proposition \ref{lb:prop:LLT} on $\prod_{\sfe}\sfP^{\otimes h}(\mathcal E_{I_{\sfe}}^{\sfe})$. First, fix a pair $\sfe\in \Pi_h$ and consider $I_\sfe = \{n^\sfe_1<\cdots <n^\sfe_{m_\sfe}\}$ with $n^\sfe_0:=0$ and $d^\sfe_r:= n^\sfe_r - n^\sfe_{r-1}$. 
 A diagonalisation argument on
\eqref{lb:eq:Gamma-edge-explicit} gives
\begin{equation}\label{lb:eq:edge-det-gaps}
 \det\Gamma_{\sfe}
 =\prod_{r=1}^{m_{\sfe}}d_r^{\sfe}.
\end{equation}
We observe that, by definition of $\mathcal E_{I_{\sfe}}^{\sfe}$ and \eqref{lb:eq:edge-det-gaps}, 
\begin{align}
 \prod_{\sfe\in\Pi_h}
 \sfP^{\otimes h}(\mathcal E_{I_{\sfe}}^{\sfe})
 &=\prod_{\sfe\in\Pi_h}\prod_{r=1}^{m_{\sfe}}
   p_{2d_r^{\sfe}}(0)
  =\frac{\pi^{-m}}
 {\prod_{\sfe\in\Pi_h}\det\Gamma_{\sfe}}
 \Big( \prod_{\sfe\in\Pi_h}\prod_{r=1}^{m_{\sfe}}
 \big(\pi d_r^{\sfe}p_{2d_r^{\sfe}}(0)\big)\Big).
 \label{lb:eq:pair-LLT-product}
\end{align}
By Stirling's approximation, $p_{2d}(0) = 4^{-2d}\binom{2d}{d}^2 = \f{1}{\pi d} -\f{1}{4\pi d^2} + O(\f{1}{d^3})$. So $\pi d p_{2d}(0)<1$, and we have
\begin{align}
  \bigg|\prod_{\sfe\in\Pi_h}\prod_{r=1}^{m_{\sfe}}
 \big(\pi d_r^{\sfe}p_{2d_r^{\sfe}}(0)\big)
 -1\bigg| 
 & =
 1-\prod_{\sfe\in\Pi_h}\prod_{r=1}^{m_{\sfe}}
 \big(\pi d_r^{\sfe}p_{2d_r^{\sfe}}(0)\big)
 \le 
 \sum_{\substack{\sfe\in \Pi_h\\ 1\le r\le m_{\sfe}}}
 (1-\pi d_r^{\sfe}p_{2d_r^{\sfe}}(0) )\notag \\ 
 & \le C \sum_{\substack{\sfe\in \Pi_h\\ 1\le r\le m_{\sfe}}}
 \f{1}{d^\sfe_r}
 \le \f{C}{\Delta} \sum_{\sfe\in \Pi_h} m_\sfe
 = \f{Cm}{\Delta} 
 \le \f{Ch^2\log L}{L^4},\label{lb:cor-err4}
\end{align}
where we have used the Stirling's expansion of $p_{2d}(0)$, definition of $\gD$ and  \eqref{lb:eq:lambda-good}. Assembling \eqref{lb:cor-err1}, \eqref{lb:cor-err2}, \eqref{lb:eq:LLT-tail-two-a} and \eqref{lb:cor-err4} we finish the proof.
\end{proof}

\subsection{Determinant inequalities and the wedge gain}
\label{lb:sec:wedge-gain}

In this subsection, we show that as long as $h\le \sqrt{\f{c_0 L}{\log L}}$, we have $\det \Gamma \le \prod_{\sfe\in \Pi_h} \det \Gamma_\sfe$. We also show that, if in addition we assume $h\ge CL^{1/3}$ for some $C>0$, then the above inequality is strict. We start from introducing some terminologies. 
Two distinct pairs $\sfe,\sff\in\Pi_h$ are called adjacent, and we write
$\sfe\sim\sff$, if they have exactly one common index. We call a pair $\{\sfe,\sff\}$ of two adjacent pairs an edge.

\begin{lemma}[Hadamard--Fischer determinant inequality]\label{lb:lem:HF}
For every configuration whose labelled marks occur at distinct times,

\begin{equation}\label{lb:eq:Rge1}
 \frac{\prod_{\sfe\in\Pi_h}\det\Gamma_{\sfe}}
      {\det\Gamma}
 \ge1.
\end{equation}

\end{lemma}

\begin{proof}

As a covariance matrix and due to \eqref{lb:eq:eig-det}, $\Gamma$ is positive definite.  Iterating the
Hadamard--Fischer inequality over principal submatrices concerning each $\sfe\in \Pi_h$ gives
\begin{equation*}
 \det\Gamma\le\prod_{\sfe\in\Pi_h}\det\Gamma_{\sfe}.
\end{equation*}
\end{proof}

We use the next few lemmas to show that the inequality in Lemma \ref{lb:lem:HF} is strict, uniformly over $N$, if we assume $h\ge CL^{1/3}$. First, we show that, if the first collision times of two time sets of adjacent pairs are comparable, then the two time sets have nontrivial correlation.

\begin{lemma}
\label{lb:lem:wedge-gain}
Let $\boldsymbol I = \union_{\sfe\in \Pi_h} I_{\sfe}$ be a total time set such that all the time marks are distinct. Let $\sfe,\sff\in\Pi_h$ be adjacent.  Suppose
\begin{equation*}
 s:=T(I_{\sfe})<t:=T(I_{\sff})<\infty,
 \qquad
 \frac{t}{2}\le s\le\frac{3t}{4}.
\end{equation*}
Then
\begin{equation*}
 \frac{\prod_{\sfg\in\Pi_h}\det\Gamma_{\sfg}}
      {\det\Gamma}
 \ge e^{1/8}.
\end{equation*}
\end{lemma}

\begin{proof}

Write
\[
 I_{\sfe}=\{s_1=s<s_2<\cdots<s_{m_{\sfe}}\},
 \qquad
 I_{\sff}=\{t_1=t<t_2<\cdots<t_{m_{\sff}}\},
\]
and set $s_0=t_0=0$. Moreover, define intervals
\[
 J_r:=(s_{r-1},s_r],
 \qquad
 K_u:=(t_{u-1},t_u].
\]
Let $\Gamma_{\sfe,\sff}$ be the submatrix of $\Gamma$ indexed by
all times in $I_{\sfe}\cup I_{\sff}$.  For any $q\in \bbN$, let $L_q$ be the $q\times q$
lower triangular matrix:
\begin{align*}
L_q:=
  \begin{pmatrix}
    1 & 0 & \cdots & \cdots & 0 \\ 
    -1 & 1 & 0 & \cdots  & 0 \\ 
    0 & -1 & 1 & \cdots & 0 \\ 
    \vdots & \vdots & \vdots & \ddots & \vdots \\ 
    0  & 0 & \hdots & -1 &  1
  \end{pmatrix}.
\end{align*}
We have $\det L_q=1$. By considering $\cL = \begin{pmatrix} L_{m_\sfe} & 0 \\ 0 & L_{m_\sff} \end{pmatrix}$, we have 
$
 \det(\mathcal L\Gamma_{\sfe,\sff}\mathcal L^{\mathsf T})
 =\det\Gamma_{\sfe,\sff}.
$ 
and
\begin{equation*}
 \mathcal L\Gamma_{\sfe,\sff}\mathcal L^{\mathsf T}
 =
 \begin{pmatrix}
  D_{\sfe}&B\\
  B^{\mathsf T}&D_{\sff}
 \end{pmatrix},
 \quad \text{ with }
 D_{\sfe}=\operatorname{diag}(|J_1|,\ldots,|J_{m_{\sfe}}|),
 \text{ and }
 D_{\sff}=\operatorname{diag}(|K_1|,\ldots,|K_{m_{\sff}}|).
\end{equation*}
The off-diagonal blocks are given by
\begin{align}
 B_{r,u}
 &=\frac12\langle\alpha_{\sfe},\alpha_{\sff}\rangle
 \bigl[
   s_r\wedge t_u-s_{r-1}\wedge t_u
   -s_r\wedge t_{u-1}+s_{r-1}\wedge t_{u-1}
 \bigr]
 =\frac12\langle\alpha_{\sfe},\alpha_{\sff}\rangle
 |J_r\cap K_u|.
 \label{lb:eq:B-overlap}
\end{align}
By the operation above and considering $I_\sfe$ or $I_\sff$ to be empty, one also observes that
\begin{equation}\label{lb:eq:single-edge-increment-dets}
 \det\Gamma_{\sfe}=\det D_{\sfe},
 \qquad
 \det\Gamma_{\sff}=\det D_{\sff}.
\end{equation}
By the Schur complement determinant formula, we have that:
\begin{align}\label{lb:shur}
  \det \Gamma_{\sfe,\sff}=
  \det \big( \cL \Gamma_{\sfe,\sff} \cL^{\sfT} \big)
  =
  \det D_\sfe 
  & \times 
  \det
  \big(
    D_\sff - B^\sfT D_\sfe^{-1} B
    \big) \notag \\
  & = \det D_\sfe \det D_\sff \det \big( 1-D_\sff^{-\f{1}{2}} B^\sfT D_\sfe^{-1} B D_\sff^{-\f{1}{2}} \big)\notag \\ 
  & = \det \Gamma_{\sfe} \, \det \Gamma_\sff  \,
  \det \big( I - H^\sfT H \big),
\end{align}
where we have used $\det \cL = 1$ and \eqref{lb:eq:single-edge-increment-dets} and define $H:= D_\sfe^{-\f{1}{2}} B D_\sff^{-\f{1}{2}}$. Notice that $\cL \Gamma_{\sfe,\sff} \cL^\sfT$ and $D_\sfe$ are positive definite. Therefore the Schur complement $D_\sff - B^\sfT D^{-1}_\sfe B$ is also positive definite. By
\begin{align*}
D_{\sff} - B^{\sfT}D^{-1}_{\sfe} B
=
D^{1/2}_{\sff} \big(I-H^{\mathsf T}H\big) D^{1/2}_{\sff},
\end{align*}
$I-H^\sfT H$ is also positive definite. Therefore, all eigenvalues $l_j$ of $(H^\sfT H)$ fall in $[0,1)$. As $-\log (1-x) \ge x$ for all $x\in [0,1)$, \eqref{lb:shur} and \eqref{lb:eq:B-overlap}, for $\sfe,\sff$ being adjacent, we have that
\begin{align*}
 \log\frac{\det\Gamma_{\sfe}\det\Gamma_{\sff}}
          {\det\Gamma_{\sfe,\sff}}
 &=-\log\det(I-H^{\mathsf T}H)
  =\sum_j-\log\bigl(1-l_j \bigr)\notag\\
 &\ge\operatorname{tr}(H^{\mathsf T}H)
  =\sum_{r=1}^{m_{\sfe}}\sum_{u=1}^{m_{\sff}}
   \frac{|J_r\cap K_u|^2}{4|J_r||K_u|}
  \ge\frac{|J_1\cap K_1|^2}{4|J_1||K_1|}
 =\frac{s^2}{4st}
 =\frac{s}{4t}
 \ge\frac18.
\end{align*}
So we achieve $ \frac{\det\Gamma_{\sfe}\det\Gamma_{\sff}}
      {\det\Gamma_{\sfe,\sff}}
 \ge e^{1/8}$. 
Finally, apply Hadamard-Fischer on $\Gamma$, with blocks $\Gamma_{\sfe,\sff}, \Gamma_{\sfg}$ for $\sfg\in \Pi_h\setminus \{\sfe,\sff\}$, we obtain
\begin{align*}
  \det \Gamma \le \det \Gamma_{\sfe,\sff} \prod_{\sfg\notin \{\sfe,\sff\}} \det \Gamma_{\sfg}
  \le e^{-1/8} \prod_{\sfg\in \Pi_h} \det \Gamma_\sfg.
\end{align*}
\end{proof}

Next, for two pairs $\sfe,\sff$, we define
\begin{equation*}
 \mathcal A_{\sfe<\sff}
 :=
 \Big\{ 
 4L^4\le T(I_{\sff})\le\frac N2,
\quad
 \frac12T(I_{\sff})
 \le T(I_{\sfe})
 \le\frac34T(I_{\sff})\Big\},
\end{equation*}
and let $\cA_{\sfe,\sff} := \cA_{\sfe<\sff} \cup \cA_{\sff<\sfe}$. Note that $\cA_{\sfe<\sff}$ and $\cA_{\sff<\sfe}$ are disjoint. $\cA_{\sfe,\sff}$ is the event that the first collision times of $I_\sfe$ and $I_\sff$ are comparable. 
\begin{lemma}\label{lb:lem:one-wedge}
There exist constants $0<{c_{\ref*{lb:lem:one-wedge}}}\le {C_{\ref*{lb:lem:one-wedge}}}<\infty$
such that, for every pair of pairs $\{\sfe,\sff\}$ and all sufficiently large $N$,
\begin{equation}\label{lb:eq:pw}
 \frac{{c_{\ref*{lb:lem:one-wedge}}}}{L}
 \le\boldsymbol\nu_{\bhb,N}(\mathcal A_{\sfe,\sff})
 \le\frac{{C_{\ref*{lb:lem:one-wedge}}} }{L}.
\end{equation}
\end{lemma}

\begin{proof}
Independence of $I_{\sfe}$ and $I_{\sff}$ under $\boldsymbol \nu_{\bhb,N}$ gives
\begin{align*}
 &\boldsymbol\nu_{\bhb,N}
 (\mathcal A_{\sfe<\sff} )
 =
 \sum_{t=\lceil4 L^4\rceil}^{\lfloor N/2\rfloor}
 \nu_{\sff,N}\big(T(I_{\sff})=t\big)
 \sum_{s=\lceil t/2\rceil}^{\lfloor3t/4\rfloor}
 \nu_{\sfe,N}\big(T(I_{\sfe})=s\big).
\end{align*}
By \eqref{lb:T=n}, we then have
\begin{align}
 &\frac{c_{\ref*{lb:lem:renewal}}^2\hb_{\sfe}\hb_{\sff}}{L^2}
 \sum_{t=\lceil4 L^4\rceil}^{\lfloor N/2\rfloor}\frac1t
 \sum_{s=\lceil t/2\rceil}^{\lfloor3t/4\rfloor}\frac1s
 \le
 \boldsymbol\nu_{\bhb,N}
 (\mathcal A_{\sfe<\sff})
 \le
 \frac{C_{\ref*{lb:lem:renewal}}^2\hb_{\sfe}\hb_{\sff}}{L^2}
 \sum_{t=\lceil4 L^4\rceil}^{\lfloor N/2\rfloor}\frac1t
 \sum_{s=\lceil t/2\rceil}^{\lfloor3t/4\rfloor}\frac1s.
 \label{lb:eq:one-wedge-two-sided}
\end{align}
Uniformly for $t\ge 4L^4$, we have $\sum_{s=\lceil t/2\rceil}^{\lfloor3t/4\rfloor}\frac1s
 =\log\frac32+O(t^{-1})$. Plugging this into the above proves the result for $\mathcal A_{\sfe<\sff}$.  The computation of $\mathcal A_{\sff<\sfe}$ is identical. Since they are disjoint,
summing up proves \eqref{lb:eq:pw}.
\end{proof}

\begin{lemma}\label{lb:lem:comparable-wedges}
  Define 
\begin{align*}
  \cN := 
  \sum_{\substack{\{\sfe,\sff\}\subset\Pi_h\\\sfe,\sff\text{ adjacent}}}
 \ind_{\mathcal A_{\sfe,\sff}}
\end{align*}
  to be the number of pairs of adjacent pairs with comparable first collision times. Take arbitrary $C_{\ref*{lb:lem:comparable-wedges}}>0$. For all $C_{\ref*{lb:lem:comparable-wedges}}\, L^{1/3}\le h\le \sqrt{\f{cL}{\log L}}$, there exists $C'_{\ref*{lb:lem:comparable-wedges}},c_{\ref*{lb:lem:comparable-wedges}}>0$, dependent only on $C_{\ref*{lb:lem:comparable-wedges}}$, $\hbn$ and $\hbm$, such that
  \begin{align}\label{lb:N>1}
    c_{\ref*{lb:lem:comparable-wedges}}
    \le 
    \boldsymbol \nu_{\bhb,N} (\cN \ge 1)
    \le \f{C'_{\ref*{lb:lem:comparable-wedges}}h^3}{L}.
  \end{align}
\end{lemma}
\begin{proof}
We show \eqref{lb:N>1} via Paley-Zygmund. First, notice that as there are $\f{h(h-1)(h-2)}{2}$ choices of pairs of adjacent pairs. Therefore by \eqref{lb:eq:pw}, for some $c,C>0$,
\begin{equation}\label{lb:eq:mu-weighted}
 {c}\frac{h^3}{L}
 \le \E_{\boldsymbol\nu_{\bhb,N}}[\cN]
 \le {C}\frac{h^3}{L}.
\end{equation}
Next we show a second moment bound of $\cN$. Consider two distinct pairs of pairs $\{\sfg,\sfe\}$ and $\{\sfg,\sff\}$ with one common pair $\sfg$. We have 
\begin{align*}
 \boldsymbol\nu_{\bhb,N}
 (\mathcal A_{\sfg,\sfe}\mid T(I_{\sfg}))
 & \le
 \nu_{\sfe,N}\left(
   T(I_{\sfe})\in \Big[\f{T(I_\sfg)}{2},\f{3T(I_{\sfg})}{4}\Big]\cap[1,N/2]
 \right)\notag\\
 &\qquad+
 \nu_{\sfe,N}\left(
   T(I_{\sfe})\in \Big[\f{4T(I_{\sfg})}{3},2T(I_\sfg)\Big]\cap[1,N/2]
 \right)\notag\\
 & \le\frac{C}{L},
\end{align*}
where in the last inequality we have used \eqref{lb:T=n}. Moreover, $\cA_{\sfg,\sfe}$ depends on the time set $I_\sfg$ only through $T(I_\sfg)$. Recall also that the time sets $I_{\sfe}$ and $I_{\sff}$ are independent. Therefore
\begin{align*}
 \boldsymbol\nu_{\bhb,N}(\mathcal A_{\sfg,\sfe}\cap\mathcal A_{\sfg,\sff})
 &=\E_{\nu_{\sfg,N}}\left[
   \boldsymbol\nu_{\bhb,N}(\mathcal A_{\sfg,\sfe}\mid I_{\sfg})
   \boldsymbol\nu_{\bhb,N}(\mathcal A_{\sfg,\sff}\mid I_{\sfg})
 \right]
 \le\frac{C}{L^2}.
\end{align*}
The number of pairs of adjacent pairs is at most $C h^4$. Therefore, by splitting the off-diagonal terms according to whether the two pairs of pairs have a common pair, we have
\begin{align*}
 \E_{\boldsymbol\nu_{\bhb,N}}[\cN^2]
 &=\sum_{\sfe,\sff}\boldsymbol\nu_{\bhb,N}(\mathcal A_{\sfe,\sff})
   +\!\!\! \sum_{\{\sfe,\sff\}\ne \{\sfe',\sff'\}}
    \boldsymbol\nu_{\bhb,N}(\mathcal A_{\sfe,\sff}\cap\mathcal A_{\sfe',\sff'})\\
 &\le \E_{\boldsymbol\nu_{\bhb,N}}[\cN]
   +\sum_{\{\sfe,\sff\}\ne \{\sfe',\sff'\}}
    \boldsymbol\nu_{\bhb,N}(\mathcal A_{\sfe,\sff})
    \boldsymbol\nu_{\bhb,N}(\mathcal A_{\sfe',\sff'})
   +C\frac{h^4}{L^2}\\
 &\le 
 \E_{\boldsymbol\nu_{\bhb,N}}[\cN] +
 \big(\E_{\boldsymbol\nu_{\bhb,N}}[\cN]\big)^2
   +C\frac{h^4}{L^2}.
\end{align*}
Now, by Paley-Zygmund, we have that for some $c>0$,
\begin{align}
  \boldsymbol \nu_{\bhb,N} (\cN\ge 1)
  \ge \f{\big( \bbE_{\boldsymbol \nu_{\bhb,N}}[\cN]\big)^2}
  {\bbE_{\boldsymbol \nu_{\bhb,N}}[\cN^2]}
  \ge 
  \f{\big(\bbE_{\boldsymbol \nu_{\bhb,N}}[\cN]\big)^2}
  {\bbE_{\boldsymbol \nu_{\bhb,N}}[\cN] 
  + 
  \big(\bbE_{\boldsymbol \nu_{\bhb,N}}[\cN]\big)^2 
  + Ch^4/L^2}
  \ge \f{1}{1+ \f{CL}{h^3}+\f{C}{h^2}},
\end{align}
where the last inequality follows from \eqref{lb:eq:mu-weighted}. Moreover, the last constant is bounded away from $0$ as long as $h\ge C_{\ref*{lb:lem:comparable-wedges}}L^{1/3}$ for some $C_{\ref*{lb:lem:comparable-wedges}}>0$. For the other side of \eqref{lb:N>1}, simply apply Markov's inequality and \eqref{lb:eq:mu-weighted}.
\end{proof}

\subsection{Proof of the lower-bound theorem}\label{lb:sec:proof-main}

\begin{proof}[Proof of Theorem~\ref{lb:thm-main}]
By \eqref{lb:start}, positivity outside $\mathcal G_N$, and
Proposition~\ref{lb:prop:comparison},

\begin{align}
 \cM^h_{N,\bhb}
 &\ge
 e^{-\f{Ch^2}{L}}
 \Big(\prod_{\sfe\in \Pi_h}\f{1}{1-\hb_{\sfe}}\Big)
 \bigl(1-\varepsilon_{\ref*{lb:prop:comparison}}(N,h)\bigr)
 \E_{\boldsymbol\nu_{\bhb,N}}\left[
  \frac{\prod_{\sfe\in\Pi_h}\det\Gamma_{\sfe}}
       {\det\Gamma}
  \ind_{\mathcal G_N}
 \right]\label{lb:raw exp}\\
 & \ge 
 e^{-\f{Ch^2}{L}}
 \bigl(1-\varepsilon_{\ref*{lb:prop:comparison}}(N,h)\bigr)
 \big( 1-\eps_{\ref*{lb:prop:G}}(N,h) \big)
 \prod_{\sfe\in \Pi_h}\f{1}{1-\hb_{\sfe}}\notag \\ 
 & \ge e^{-\f{Ch^2}{L}} \big(1-\f{Ch^2\log L}{L}\big)
 \prod_{\sfe\in \Pi_h}\f{1}{1-\hb_{\sfe}},
 \label{lb:eq:baseline-master}
\end{align}
where the second inequality is due to Lemma~\ref{lb:lem:HF} and Proposition~\ref{lb:prop:G}, and in the last inequality we have plugged in the exact forms of $\eps_{\ref*{lb:prop:G}}$ and $\eps_{\ref*{lb:prop:comparison}}$, with $\eps_{\ref*{lb:prop:G}}>>\eps_{\ref*{lb:prop:comparison}}$. If we choose $c_0>0$ sufficiently small, $\f{Ch^2\log L}{L}$ falls in $[0,\f{1}{2}]$. Then by $1-x\ge e^{-2x}$ on $[0,\f{1}{2}]$, we achieve \eqref{lb:eq:quant-lower}.

For \eqref{lb:eq:strict-log}, by Lemmas~\ref{lb:lem:HF} and \ref{lb:lem:wedge-gain}, on $\mathcal G_N$,
\begin{equation*}
 \frac{\prod_{\sfe\in\Pi_h}\det\Gamma_{\sfe}}
      {\det\Gamma}
 \ge1+(e^{1/8}-1)\ind_{\{\cN\ge1\}}.
\end{equation*}
Hence
\begin{align}
 \E_{\boldsymbol\nu_{\bhb,N}}\left[
  \frac{\prod_{\sfe\in\Pi_h}\det\Gamma_{\sfe}}
       {\det\Gamma}
  \ind_{\mathcal G_N}
 \right]
 &\ge
 \boldsymbol\nu_{\bhb,N}(\mathcal G_N)
 +(e^{1/8}-1)\boldsymbol\nu_{\bhb,N}
   (\mathcal G_N\cap\{\cN\ge1\})\notag\\
 &\quad\ge
 1+(e^{1/8}-1)\boldsymbol\nu_{\bhb,N}(\cN\ge1)
 -\big(1+(e^{1/8}-1)\big)\varepsilon_{\ref*{lb:prop:G}}(N,h).
 \label{lb:eq:strict-master}
\end{align}
Fix arbitrary $C_{\ref*{lb:lem:comparable-wedges}}>0$ and assume $h = o\big(\sqrt{L/\log L}\big)$ and $h \ge {C_{\ref*{lb:lem:comparable-wedges}}}L^{1/3}$.  Then \eqref{lb:N>1}  gives
$\boldsymbol\nu_{\bhb,N}(\cN\ge1)\ge c_{\ref*{lb:lem:comparable-wedges}}$, for some $c_{\ref*{lb:lem:comparable-wedges}}>0$ dependent on $C_{\ref*{lb:lem:comparable-wedges}}$. For $h = o\big(\sqrt{L/\log L}\big)$, $\eps_{\ref*{lb:prop:G}}(N,h)\rightarrow 0$ as $N\rightarrow \infty$. By recalling $\eps_{\ref*{lb:prop:comparison}}<<\eps_{\ref*{lb:prop:G}}$ and using \eqref{lb:eq:strict-master} in \eqref{lb:raw exp}, we have that 
\begin{align*}
  \cM^h_{N,\bhb} 
  \ge \Big(1+(e^{\f{1}{8}}-1) c_{\ref*{lb:lem:comparable-wedges}} - o(1)\Big)
  \prod_{\sfe\in \Pi_h} \f{1}{1-\hb_{\sfe}},
\end{align*}
from which we may conclude.
\end{proof}

\appendix
\section{Auxiliary results for the independence upper bound}
 \label{ub:app-auxiliary}

We first record the pair estimate from Cosco--Zeitouni. As previously defined, put
$
 \widebar \gs_N =e^{\pi\bar{\gb}/\log N}-1
$.

\begin{lemma}[Cosco--Zeitouni]
\label{ub:lem-coarse-pair-CZ}
For $0\le K\le N$ and $\bz\in\bbZ^{2h}$, define
\begin{equation*}
 \cM^h_{N,K,\bar \gb}(\bz)
 :=\sfE_{\bz}^{\otimes h}\left[
   \prod_{n=1}^K
   \left(1+\widebar \gs_N 
     \sum_{\sfe=\{\sfi,\sfj\}\in\Pi_h}
     \ind_{\{S_n^{\sfi}=S_n^{\sfj}\}}
   \right)
 \right],
\end{equation*}
If $h = o(\sqrt{L})$, then there exists $C>0$ such that
\begin{align}\label{ub:prep-est}
  \sup_{\bz\in \bbZ^{2h}} \cM^h_{N,K,\bar \gb}(\bz)
  \le 
  C e^{Ch^2 \f{\log(K+1)}{L}}.
\end{align}
If in addition $h\to\infty$, we have
\begin{equation*}
 \sup_{\bz\in\bbZ^{2h}}\cM^h_{N,\bar \gb}(\bz)
 \le
 \Big(\f{1}{1-\bar \gb}\Big)^{
 \binom h2\,\big(1+o(1)\big)}.
\end{equation*}
\end{lemma}
\begin{proof}
This is proven in \cite{CoscoZeitouni2023} with $\hb = \bar \gb$.
\end{proof}

The following proof is also originated from \cite{CoscoZeitouni2023}. We adapt it to the inhomogeneous case.
\begin{proof}[Proof of Lemma~\ref{ub:lem-multibody-domination}]
For \(\bx = (x^1,\cdots x^h)\in\bbZ^{2h}\), and $\sfe = \{\sfi_\sfe,\sfj_\sfe\}$ for all $\sfe\in \Pi_h$, write
\begin{align*}
 \cA_N(\bx)
 &:=\sum_{\sfe\in \Pi_h}
 \gs_N^{\sfe}\ind_{x^{\sfi_\sfe} = x^{\sfj_\sfe}},\\
 \cB_N(\bx)
 &:=\prod_{\sfe\in\Pi_h}
       \bigl(1+\gs_N^{\sfe}\ind_{x^{\sfi_\sfe} = x^{\sfj_\sfe}}\bigr)
       -1-\cA_N(\bx).
\end{align*}
Thus \(\cA_N\) is the contribution of selecting one pair at a given time,
while \(\cB_N\) is the contribution of selecting at least two distinct
pairs at that time.  For \(K\le N\), define
\begin{align*}
 \cM^{\mathrm{pair}, h}_{N,K,\bhb}(\bz)
 &:=\sfE_{\bz}^{\otimes h}
   \Big[\prod_{n=1}^K\bigl(1+\cA_N(\mathbf S_n)\bigr)\Big],
   \quad\quad
\cM^{h}_{N,K,\bhb}(\bz)
 :=\sfE_{\bz}^{\otimes h}
   \Big[\prod_{n=1}^K
      \bigl(1+\cA_N(\mathbf S_n)+\cB_N(\mathbf S_n)\bigr)
   \Big],
\end{align*}
and set
\[
 \widebar \cM^{\mathrm{pair}, h}_{N,\bhb}:=\sup_{K\le N}\sup_{\bz}
 \cM^{\mathrm{pair}, h}_{N,K,\bhb}(\bz),
 \qquad
 \widebar \cM^{h}_{N,\bhb}:=\sup_{K\le N}\sup_{\bz}
 \cM^{h}_{N,K,\bhb}(\bz).
\]
By expanding according to the first time when a factor $\cB_N$ is selected
and then applying the Markov property, we have
\begin{equation}\label{ub:eq-first-bad-decomposition}
 \widebar \cM^{h}_{N,\bhb}
 \le 
 \widebar \cM^{\mathrm{pair}, h}_{N,\bhb}
 +
 \delta_{N,h}\widebar \cM^{h}_{N,\bhb},
\end{equation}
where
\begin{equation}\label{ub:eq-delta-definition}
 \delta_{N,h}
 :=\sup_{K\le N}\sup_{\bz}
   \sum_{k=1}^K
   \sfE_{\bz}^{\otimes h}\Big[
     \prod_{n=1}^{k-1}\bigl(1+\cA_N(\mathbf S_n)\bigr)
     \cB_N(\mathbf S_k)
   \Big].
\end{equation}
 Next we expand $\cB_N(\bx)$. As it has at least one term that involves two pairs, for some $C>0$,
 \begin{align}
  \cB_N(\bx) 
  & \le 
  \sum_{\sfe,\sff\in \Pi_h: \, \sfe\neq \sff}
  \gs^{\sfe}_N \gs^{\sff}_N 
  \ind_{x^{\sfi_\sfe} = x^{\sfj_\sfe}}
  \ind_{x^{\sfi_\sff} = x^{\sfj_\sff}}
  \prod_{\sfg\neq \sfe,\sff}
  \big( 1+\gs_N^\sfg \ind_{x^{\sfi_\sfg} = x^{\sfj_\sfg}} \big)\notag \\ 
  & \le 
  \sum_{\sfe,\sff\in \Pi_h: \, \sfe\neq \sff}
  \gs^{\sfe}_N \gs^{\sff}_N 
  \ind_{x^{\sfi_\sfe} = x^{\sfj_\sfe}}
  \ind_{x^{\sfi_\sff} = x^{\sfj_\sff}}\, 
  \exp\Big\{ \sum_{\sfg\in \Pi_h}\gs_N^\sfg \Big\}\notag \\
  & \le 
  C \sum_{\sfe,\sff\in \Pi_h: \, \sfe\neq \sff}
  \gs^{\sfe}_N \gs^{\sff}_N 
  \ind_{x^{\sfi_\sfe} = x^{\sfj_\sfe}}
  \ind_{x^{\sfi_\sff} = x^{\sfj_\sff}},
 \label{ub:eq-B-two-edges}\end{align}
where in the last inequality we have used that $\gs^\sfg_N = O(1/L)$ and $h = o(\sqrt{L})$. Therefore,
\begin{align}
   & \sfE_{\bz}^{\otimes h}\Big[
     \prod_{n=1}^{k-1}\bigl(1+\cA_N(\mathbf S_n)\bigr)
     \cB_N(\mathbf S_k)
   \Big]\notag \\
   & \le 
   C \sum_{\sfe,\sff\in \Pi_h: \, \sfe\neq \sff}
   \sfE_{\bz}^{\otimes h}\Big[
     \prod_{n=1}^{k-1}\bigl(1+\cA_N(\mathbf S_n)\bigr)
     \gs^{\sfe}_N \gs^{\sff}_N 
  \ind_{S_k^{\sfi_\sfe} = S_k^{\sfj_\sfe}}
  \ind_{S_k^{\sfi_\sff} = S_k^{\sfj_\sff}}
   \Big] \notag \\ 
  & \le 
  C \sum_{\sfe,\sff\in \Pi_h: \, \sfe\neq \sff} \,
  \gs^{\sfe}_N \gs^{\sff}_N 
   \bigg( \sfE_{\bz}^{\otimes h}\Big[
     \prod_{n=1}^{k-1}\bigl(1+\cA_N(\mathbf S_n)\bigr)^{\f{5}{2}} \Big]\bigg)^{\f{2}{5}} \, 
     \bigg( \sfP_{\bz}^{\otimes h} \Big[ 
  S_k^{\sfi_\sfe} = S_k^{\sfj_\sfe},\,
  S_k^{\sfi_\sff} = S_k^{\sfj_\sff}
   \Big]\bigg)^{\f{3}{5}},\label{ub:holder}
\end{align}
where in the last inequality above we have applied H\"{o}lder's inequality. We observe that, uniformly over any $\bx\in \bbZ^{2h}$, $\cA_N(\bx)\ge 0$ and $\cA_N(\bx) = o(1)$, as $h = o(\sqrt{L})$ and $\gs^{\sfe}_N = O(1/L)$ uniformly over $\sfe\in \Pi_h$. Therefore for any $\bx\in \bbZ^{2h}$, we have
\begin{align*}
  (1+\cA_N(\bx))^{\f{5}{2}} \le 1+3\cA_N(\bx).
\end{align*}
Therefore, for some $C>0$, we have
\begin{align}
  \sfE_{\bz}^{\otimes h}\Big[
     \prod_{n=1}^{k-1}\bigl(1+\cA_N(\mathbf S_n)\bigr)^{\f{5}{2}} \Big]
  & \le 
  \sfE_{\bz}^{\otimes h}\Big[
     \prod_{n=1}^{k-1}\bigl(1+3\cA_N(\mathbf S_n)\bigr) \Big]\notag \\
  & \le 
  \sup_{\bz\in \bbZ^{2h}}
  \cM^{h}_{N,k-1,\bar \gb}(\bz)
  \le 
  C k^{\f{Ch^2}{L}},\label{ub:multi1}
\end{align}
where in the second last inequality we have used $3\hb_{\sfe}<\bar \gb$ for any $\sfe\in \Pi_h$, and in the last inequality we applied \eqref{ub:prep-est}. Moreover, using the standard bound for transition probability of planar random walk $\sup_x p_k(x) \le \f{C}{k}$, 
\begin{align}\label{ub:multi2}
  \sfP_{\bz}^{\otimes h} \Big[ 
  S_k^{\sfi_\sfe} = S_k^{\sfj_\sfe},\,
  S_k^{\sfi_\sff} = S_k^{\sfj_\sff}
   \Big]
   \le 
   \f{C}{k^2}.
\end{align}
Combining \eqref{ub:multi1} and \eqref{ub:multi2}, we obtain an upper bound for \eqref{ub:holder}:
\begin{align}
  \sfE_{\bz}^{\otimes h}\Big[
     \prod_{n=1}^{k-1}\bigl(1+\cA_N(\mathbf S_n)\bigr)
     \cB_N(\mathbf S_k)
   \Big]
   \le 
   C k^{\f{2Ch^2}{5L} - \f{6}{5}} \!\!\!
   \sum_{\sfe,\sff\in \Pi_h: \, \sfe\neq \sff} \!\!\!
  \gs^{\sfe}_N \gs^{\sff}_N 
  \le 
  \f{Ch^4}{L^2} k^{\f{2Ch^2}{5L} - \f{6}{5}},\label{ub:multi-last}
\end{align}
where in the last inequality we have again used $\gs^{\sfe}_N = O(1/L)$. Due to $h = o(\sqrt{L})$, the exponent $\f{2Ch^2}{5L} - \f{6}{5}$ can be made smaller than $-\f{11}{10}$. By summing over $k$, plugging \eqref{ub:multi-last} into \eqref{ub:eq-delta-definition} we obtain
\begin{align*}
  \gd_{N,h} \le \f{Ch^4}{L^2} \sum_{k\ge 1} k^{-\f{11}{10}}
  \le \f{Ch^4}{L^2}.
\end{align*}
Recall \eqref{ub:eq-first-bad-decomposition}. To finish, we only have to observe that $\cM^{\mathrm{pair},h}_{N,K,\bhb}(\bz)$ is increasing with respect to $K$.
\end{proof}

We next record the one-pair renewal estimate used in the preliminary
reduction. It is precisely the proof of
\cite[Propositions~3.4, 3.5 and~3.7]{CoscoZeitouni2023}, but we write the error explicitly.

\begin{proposition}
\label{ub:prop-renewal-convolution}
There is $C <\infty$, depending only on $\hbn$ and $\hbm$, such
that, uniformly in $\sfe\in\Pi_h$ and every real $u\in[1,N]$,
\begin{align}
 \sum_{v=0}^N U_N^{\sfe}(v)
 &\le\frac{1+C /L}{1-\hb_{\sfe}},
 \label{ub:eq-renewal-cumulative}\\
 \sum_{v=0}^N
 \frac{U_N^{\sfe}(v)}{\pi(u+v/2)}
 &\le
 \left(1+\frac{C }{L}\right)
 \frac{1}{\pi u}
 \frac{1}{1-\hb_{\sfe}\log u/L}.
 \label{ub:eq-renewal-weighted}
\end{align}
\end{proposition}

\begin{proof}
Fix $\sfe\in\Pi_h$.  To show \eqref{ub:eq-renewal-cumulative}, simply notice that from the renewal expansions
\eqref{ub:eq-U-zero} and \eqref{ub:eq-U-def},
\begin{align*}
 \sum_{v=0}^N U_N^{\sfe}(v)
 &\le\sum_{j\ge0}
 \left(\gs_N^{\sfe}\sum_{n=1}^Np_{2n}(0)\right)^j
 =\frac{1}{1-\gs_N^{\sfe}\sum_{n=1}^Np_{2n}(0)}
 \le\frac{1+C /L}{1-\hb_{\sfe}}.
\end{align*}
Similarly,
\begin{align}\label{ub:U1}
  \sum_{0\le v\le\lfloor u\rfloor}
 \frac{U_N^{\sfe}(v)}{\pi(u+v/2)}
 \le 
 \f{1}{\pi u}
 \sum_{0\le v\le\lfloor u\rfloor}
 U_N^{\sfe}(v)
 \le
 \left(1+\frac{C }{L}\right)
 \frac{1}{\pi u}
 \frac{1}{1-\hb_{\sfe}\log u/L}.
\end{align}
By \cite[Proposition~3.5]{CoscoZeitouni2023}, uniformly in
$\sfe\in\Pi_h$ and $v\in [1,N]$, we have 
$
 U_N^{\sfe}(v)
 \le\frac{C }
 {vL\bigl(1-\hb_{\sfe}\log v/L\bigr)^2}
$. 
Therefore,
\begin{align}\label{ub:U2}
  \sum_{\lfloor u\rfloor<v\le N}
 \frac{U_N^{\sfe}(v)}{\pi(u+v/2)}
 &\le\frac{C }{L}
 \sum_{u<v\le N}\frac{1}
 {v^2\bigl(1-\hb_{\sfe}\log v/L\bigr)^2}
 \le\frac{C }{Lu}.
\end{align}
Combining \eqref{ub:U1} and \eqref{ub:U2} we obtain \eqref{ub:eq-renewal-weighted}.
\end{proof}

\section{Auxiliary results for the lower bounds}
\label{lb:app:LLT}
\subsection{Renewal estimate}
\label{aux:renewal estimate}
\begin{lemma}
  There exists $C>0$, depending only on $\hbn$ and $\hbm$, such that
  \begin{align}\label{lb:asymp-U}
    \Big| \sum_{v=0}^N U^\sfe_N(v) - \f{1}{1-\hb_{\sfe}} \Big|
    \le 
    \f{C}{L}.
  \end{align}
\end{lemma}
\begin{proof}
By \eqref{ub:eq-renewal-cumulative} and $\hb_\sfe<1$ uniformly over $\sfe\in \Pi_h$, we have
  \begin{align*}
    \sum_{v=0}^N U^\sfe_N(v) 
    \le \f{1 }{1-\hb_{\sfe}}+\f{C}{L} .
  \end{align*}
For the reverse bound, by definition of $U^\sfe_N$, for $K_0 = \log L$, 
\begin{align}
  \sum^N_{v=0} U^\sfe_N(v)
  \ge 
  1
  + \sum_{1\le k\le K_0} 
  \Big( \gs^\sfe_N \sum_{n=1}^{N/k} p_{2n}(0)
    \Big)^k
  \ge 
  1
  +\sum_{1\le k\le K_0}
  \Big( \gs^\sfe_N \,  \big(\f{\log (N/k) }{\pi}-C\big) \Big)^k,  \label{lb:U}
\end{align}
where we have used Stirling's approximation $p_{2n}(0) = 4^{-2n} \binom{2n}{n}^2  = \f{1}{\pi n}+ O(n^{-2})$. Using $$
\gs^\sfe_N 
= e^{\pi \hb_\sfe/L} -1 
= \f{\pi \hb_\sfe}{L} + O(L^{-2}),
$$
Let $k\in [1,K_0]$. As $x^k-y^k = (x-y)\sum^{k-1}_{j=0}x^{k-1-j}y^j$, we have that
\begin{align*}
  \Big| \hb_{\sfe}^k
  -
  \Big( \gs^\sfe_N \,  \big(\f{\log (N/k) }{\pi}-C\big) \Big)^k \Big| 
  & \le  
  \Big| 
    \hb_\sfe - \gs^\sfe_N \,  \big(\f{\log (N/k) }{\pi}-C\big)\Big| \,\, 
  \sum_{j=0}^{k-1} \hb_\sfe^{k-1-j}
  \Big( \gs^\sfe_N \,  \big(\f{\log (N/k) }{\pi}-C\big) \Big)^{j} \notag \\ 
  & \le \f{C(1+\log k)}{L} 
  \sum^{k-1}_{j=0} \hb^{k-1-j}_{\sfe} \Big( \hb_\sfe \big( 1 - \f{C+ \log k}{L}\big)\Big)^j 
  \le \f{C k (1+\log k)}{L} \hb_{\sfe}^{k-1}
\end{align*}
Therefore, we continue to lower bound \eqref{lb:U} by
\begin{align*}
 1+ \sum_{1\le k\le K_0}
\hb_\sfe^k
-
\f{C}{L}\sum_{1\le k\le K_0} k(1+\log k) \hb_\sfe^{k-1}
 \ge
 \f{1-\hb_\sfe^{K_0}}{1-\hb_\sfe} - \f{C}{L}.
\end{align*}
By recalling $K_0 = \log L$, we merge $\hb_\sfe^{K_0}$ into $\f{C}{L}$ and finish.
\end{proof}

\subsection{Proof of the local limit theorem}

We prove Proposition~\ref{lb:prop:LLT}. Recall that
\begin{align*}
  \bI = \big\{ n^{\sfe_1}_1 <\cdots < n^{\sfe_m}_m \big\}.
\end{align*}
Sometimes we abbreviate $n_k = n^{\sfe_k}_k$ for all $1\le k\le m$, and take for convenience that $n_0:=0$. We define the gaps $d_r:= n_r-n_{r-1}$ for all $1\le r\le m$, and notice that $\Delta = \min_{1\le r\le m} d_r$. As we have introduced in Subsection~\ref{lb:sec:comparison}, we apply the orthogonal rotation
$$
 \frac1{\sqrt2}
 \begin{pmatrix}1&1\\1&-1\end{pmatrix}
$$
to every entry of $\bS$. Recall that we use $\tS_n^{\sfj,a}$ to denote the coordinate
$a\in\{1,2\}$ of rotated walk $\tS_n^\sfj$. Notice that the increments of the family of walks 
$(\tS^{\sfj,a})_{\sfj,a}$ are independent and take the values
$\pm \f{1}{\sqrt{2}}$ with equal probability.  For $1\le r\le m$ and
$a\in\{1,2\}$, define
\begin{equation*}
 V_r^{(a)}
 :=\sum_{j=1}^h\alpha_{\sfe_r}^j \, \tS_{n_r}^{j,a},
\end{equation*}
and let $V:= (V^{(a)}_r)_{1\le r\le m,\, a=1,2}$. Then
$
 \bigcap_{\sfe\in\Pi_h}\mathcal E_{I_{\sfe}}^{\sfe}
 =\{V=0\}
 $. 
Moreover, $V$ takes values in $(\sqrt2\bbZ)^{2m}$.
For
$\eta=(\eta^{(1)},\eta^{(2)})\in\bbR^{2m}$, with $\eta^{(a)} = (\eta^{(a)}_1,\cdots , \eta^{(a)}_m)$ for $a=1,2$, set
\begin{equation*}
 B_{j,r}^{(a)}(\eta)
 :=\sum_{s=r}^m\alpha_{\sfe_s}^j\eta_s^{(a)}.
\end{equation*}
Grouping the increments over the intervals $(n_{r-1},n_r]$ gives the
characteristic function
\begin{equation*}
 \phi(\eta)
 :=\sfE^{\otimes h}[e^{i\langle\eta,V\rangle}]
 =\prod_{r=1}^m\prod_{j=1}^h\prod_{a=1}^2
 \cos\Big(\frac{B_{j,r}^{(a)}(\eta)}{\sqrt2}\Big)^{d_r}.
\end{equation*}
Since the lattice spacing in each coordinate is $\sqrt2$, a
cell is
$
 \mathcal D_m:=[-\pi/\sqrt2,\pi/\sqrt2)^{2m}
$, 
and Fourier inversion yields
\begin{equation}\label{lb:whole-int}
 \sfP^{\otimes h}(V=0)
 =\frac{2^m}{(2\pi)^{2m}}
 \int_{\mathcal D_m}\phi(\eta)\,\mathrm d\eta.
\end{equation}
Define the quadratic operator
$
 Q(\eta)
 :=\eta^{\mathsf T}(\Gamma\oplus\Gamma)\eta
$. 
By \eqref{lb:eq:Gamma} and summation over the time intervals,
\begin{equation}\label{lb:eq:quadratic-identity}
 Q(\eta)
 =\frac12\sum_{r=1}^md_r
   \sum_{j=1}^h\sum_{a=1}^2
   \bigl(B_{j,r}^{(a)}(\eta)\bigr)^2.
\end{equation}
Recall $\lambda=\lambda_{\min}(\Gamma)$ is the smallest eigenvalue and define
\begin{align*}
 \mathcal V
 &:=\left\{\eta\in\bbR^{2m}:
   |\eta|\le\frac{\pi}{\sqrt{2m}}\right\},
 \\
 \mathcal U
 &:=\left\{\eta\in\bbR^{2m}:
   |(\Gamma\oplus\Gamma)^{1/2}\eta|
   \le {c_{\ref*{lb:prop:LLT}}^{(1)}}\left(\frac{\lambda}{m}\right)^{1/4}
 \right\}.
\end{align*}
If $\eta\in\mathcal U$, then we have
$
 |\eta|
 \le {c_{\ref*{lb:prop:LLT}}^{(1)}}/(\lambda \, m )^{\f{1}{4}},
$ and 
\begin{align}\label{lb:B^2}
 \max_{j,r,a}|B_{j,r}^{(a)}(\eta)|^2
 \le m|\eta|^2
 \le {c_{\ref*{lb:prop:LLT}}^{(1)}}^2\sqrt{\frac m\lambda}.
\end{align}

First we estimate \eqref{lb:whole-int} in $\cU$. As we assume $\lambda\ge c^{(2)}_{\ref*{lb:prop:LLT}}m^3$ in Proposition \ref{lb:prop:LLT}, choosing ${c_{\ref*{lb:prop:LLT}}^{(2)}}$ sufficiently large and
${c_{\ref*{lb:prop:LLT}}^{(1)}}$ sufficiently small gives
\begin{equation*}
 \max_{j,r,a}|B_{j,r}^{(a)}(\eta)|\le\frac12
 \quad\quad \text{ for all }\eta\in\mathcal U.
\end{equation*}
and $\cU\subseteq\cV$, as for $\eta\in \cU$, we have $|\eta|
 \le 
 \f{c_{\ref*{lb:prop:LLT}}^{(1)}}{(\lambda \, m )^{1/4}}
 \le 
 \f{c_{\ref*{lb:prop:LLT}}^{(1)}}{m\, {c_{\ref*{lb:prop:LLT}}^{(2)}}^{1/4}}
 \le
 \frac{\pi}{\sqrt{2m}}
$. 
On $\mathcal U$, all cosine factors are positive.  By
$\log\cos x=-x^2/2+O(x^4)$ and \eqref{lb:eq:quadratic-identity}, we obtain
\begin{equation}\label{lb:eq:log-phi-expansion}
 \log\phi(\eta)
 =-\frac12Q(\eta)+\sfR(\eta),
\end{equation}
where by \eqref{lb:eq:quadratic-identity} and \eqref{lb:B^2},
\begin{align}
 |\sfR(\eta)|
 &\le C\sum_{r,j,a}d_r|B_{j,r}^{(a)}(\eta)|^4
 \le C\max_{j,r,a}|B_{j,r}^{(a)}(\eta)|^2
   \sum_{r,j,a}d_r|B_{j,r}^{(a)}(\eta)|^2
  \le C\sqrt{\frac m\lambda}\,Q(\eta).
 \label{lb:eq:Taylor-remainder}
\end{align}
For $\eta\in\mathcal U$, by definition of $\cU$, we have that
$
 \sqrt{\frac m\lambda}\,Q(\eta)
 \le {c_{\ref*{lb:prop:LLT}}^{(1)}}^2
$. 
So ${c_{\ref*{lb:prop:LLT}}^{(1)}}$ may be chosen small such that $|\sfR(\eta)|\le1/4$. Moreover, notice that for the Gaussian part we have
\begin{equation*}
\int_{\bbR^{2m}}e^{-\f{1}{2}Q(\eta)}\,\mathrm d\eta
 =\frac{(2\pi)^m}{\det\Gamma}.
\end{equation*}
By the change of variables
$y=(\Gamma\oplus\Gamma)^{1/2}\eta$, 
and using \eqref{lb:eq:log-phi-expansion}, \eqref{lb:eq:Taylor-remainder}, we have
\begin{align}
 \frac{\det \Gamma}{(2\pi)^m}
 \left|
   \int_{\mathcal U}\phi(\eta)\,\mathrm d\eta
   -\int_{\mathcal U}e^{-\f{1}{2}Q(\eta)}\,\mathrm d\eta
 \right|
 & \le
 \frac{\det \Gamma}{(2\pi)^m}
 \int_{\mathcal U}e^{-\f{1}{2}Q(\eta)}\cdot 
 \Big| e^{\sfR(\eta)}-1 \Big|  
 \,\dd \eta \notag \\
 & \le 
 \frac{C\det \Gamma}{(2\pi)^m}
 \int_{\mathcal U} \sqrt{\f{m}{\gl}} Q(\eta) 
 e^{-\f{1}{2}Q(\eta)} 
 \,\dd \eta\notag \\
 & \le C\sqrt{\f{m}{\gl}}\frac{1}{(2\pi)^m}
 \int_{\bbR^{2m}} 
  |y|^2 e^{-\f{|y|^2}{2}}\, \dd y
 \le 
C\sqrt{\frac m\lambda}\,\bbE|G_{2m}|^2
 \le 
 C\sqrt{\frac{m^3}{\lambda}},
 \label{lb:eq:central-relative-error}
\end{align}
where $G_{2m}$ is a $2m$-dimensional standard Gaussian vector. 

Next we estimate \eqref{lb:whole-int} over $\cV\setminus\cU$. By definition of $B^{(a)}_{j,r}(\eta)$ and Cauchy-Schwarz, 
we have $
 |B_{j,r}^{(a)}(\eta)|
 \le\sqrt m\,|\eta|
$; hence for $\eta\in \cV$, we have $
 |B_{j,r}^{(a)}(\eta)|
 \le \f{\pi}{\sqrt{2}}.
$
Since
$|\cos(x/\sqrt2)|\le e^{-x^2/4}$ for $|x|\le\pi/\sqrt2$,
\eqref{lb:eq:quadratic-identity} gives
\begin{equation}\label{lb:eq:phi-Gaussian-on-V}
 |\phi(\eta)|\le e^{-Q(\eta)/2},
 \quad \quad \text{ for all } \eta\in\mathcal V.
\end{equation}
Consequently, we have the upper bound:
\begin{align}\label{lb:eq:central-Gaussian-tail}
  \f{\det \Gamma}{(2\pi)^m}
  \int_{\cV \setminus \cU }
  |\phi(\eta)| \, \dd \eta 
  \le 
  \f{\det \Gamma}{(2\pi)^m}
  \int_{\cU^c }
  e^{-\f{Q(\eta)}{2}} \, \dd \eta 
  =
  \bbP \Big( 
    |G_{2m}|> c^{(1)}_{\ref*{lb:prop:LLT}} \big( \f{\gl}{m} \big)^{\f{1}{4}}
     \Big),
\end{align}
where in the last equality we have again applied the change of variable $y = (\Gamma\oplus \Gamma)^{1/2}\eta$.

It remains to control \eqref{lb:whole-int} on $\mathcal D_m\setminus\mathcal V$. Consider
\begin{equation*}
 \tB_{j,r}^{(a)}(\eta)
 :=\dist(B_{j,r}^{(a)}(\eta),\sqrt2\pi\bbZ),
\end{equation*}
and set
$B_{j,m+1}^{(a)}:=\tB_{j,m+1}^{(a)}:=0$. Here $\dist(x,\sqrt2\pi\bbZ)$ for $x\in \bbR$ is the distance between $x$ and its closest point on $\sqrt2\pi\bbZ$. For any $x$, we have
\begin{equation*}
 \big|\cos(x/\sqrt2)\big|
 \le\exp\{-c\dist(x,\sqrt2\pi\bbZ)^2\}.
\end{equation*}
Therefore,
\begin{equation}\label{lb:eq:phi-distance}
 |\phi(\eta)|
 \le\exp\Big\{-c\sum_{r=1}^md_r
   \sum_{j=1}^h\sum_{a=1}^2 \, \tB_{j,r}^{(a)}(\eta)^2
 \Big\}.
\end{equation}
For every $r$, choose an index $\sfj_r$ of $\sfe_r$.  Since
$|\alpha_{\sfe_r}^{\sfj_r}|=1$, we have 
$
 \alpha_{\sfe_r}^{\sfj_r}\eta_r^{(a)}
 =B_{\sfj_r,r}^{(a)}(\eta)-B_{\sfj_r,r+1}^{(a)}(\eta)
$. 
Then we apply triangle inequality and square both sides, which yields
\begin{equation*}
 \dist(\eta_r^{(a)},\sqrt2\pi\bbZ)^2
 \le2\left[
   \tB_{\sfj_r,r}^{(a)}(\eta)^2+\tB_{\sfj_r,r+1}^{(a)}(\eta)^2
 \right].
\end{equation*}
Summing over $r,a$, we obtain
\begin{equation}\label{lb:eq:distance-charge-sum}
 4\sum_{r,j,a}(\tB_{j,r}^{(a)}(\eta))^2
 \ge \sum_{r,a}
   \dist(\eta_r^{(a)},\sqrt2\pi\bbZ)^2
 = |\eta|^2,
 \qquad \text{ for all }\eta\in\mathcal D_m.
\end{equation}
The last equality holds by definition of $\mathcal D_m$.  Combining
\eqref{lb:eq:phi-distance} and \eqref{lb:eq:distance-charge-sum}, by definition of $\cV$, we have
\begin{align}
 \f{\det \Gamma}{(2\pi)^m}
 \int_{\mathcal D_m\setminus\mathcal V}
 |\phi(\eta)|
 \,\dd\eta
 &\le \f{\det \Gamma}{(2\pi)^m}
 \int_{|\eta|>\f{\pi}{\sqrt{2m}}}
   e^{-c\Delta|\eta|^2}\,\dd\eta
   \le
   \f{\det \Gamma}{(2\pi)^m}
   \left(\frac{C}{\Delta}\right)^m
 \bbP\left(|G_{2m}|>{c_{\ref*{lb:prop:LLT}}^{(1)}}\sqrt{\frac{\Delta}{m}}\right).
 \label{lb:eq:outer-integral}
\end{align}
Combining 
\eqref{lb:eq:central-relative-error}, \eqref{lb:eq:central-Gaussian-tail}, \eqref{lb:eq:outer-integral} to plug in \eqref{lb:whole-int}, we obtain \eqref{lb:eq:joint-LLT}
with \eqref{lb:eq:joint-LLT-error}.  

It remains to show \eqref{lb:eq:eig-det}.  For $v\in\bbR^m$, set
$
 w_r:=\sum_{s=r}^mv_s\alpha_{\sfe_s}
$ for all $1\le r\le m$ and $w_{m+1}:=0$. 
Then
\begin{equation*}
 v^{\mathsf T}\Gamma v
 =\frac12\sum_{r=1}^md_r|w_r|^2.
\end{equation*}
Moreover,
$v_r\alpha_{\sfe_r}=w_r-w_{r+1}$ and
$|\alpha_{\sfe_r}|^2=2$, so
$
 2|v|^2
 =\sum_{r=1}^m|w_r-w_{r+1}|^2
 \le4\sum_{r=1}^m|w_r|^2
$. 
Therefore,
\begin{equation*}
 v^{\mathsf T}\Gamma v
 \ge\frac{\Delta}{2}\sum_{r=1}^m|w_r|^2
 \ge\frac{\Delta}{4}|v|^2,
\end{equation*}
which proves $\lambda_{\min}(\Gamma)\ge\Delta/4$.  Finally,
$\Gamma_{rr}=n_r$ and Hadamard's inequality gives

\begin{equation*}
 \det\Gamma\le\prod_{r=1}^m\Gamma_{rr}
 =\prod_{r=1}^mn_r\le N^m.
\end{equation*}


\begin{thebibliography}{99}
\bibitem{CSZDickman}
F.~Caravenna, R.~Sun and N.~Zygouras,
\emph{The Dickman subordinator, renewal theorems, and disordered systems},
Electron. J. Probab. \textbf{24} (2019), paper no.~101.

\bibitem{Comets2017}
F.~Comets,
\emph{Directed Polymers in Random Environments},
Lecture Notes in Mathematics, vol.~2175,
Springer, 2017.

\bibitem{CoscoZeitouni2023}
C.~Cosco and O.~Zeitouni,
\emph{Moments of partition functions of 2D Gaussian polymers in the weak
disorder regime--I},
Comm. Math. Phys. \textbf{403} (2023), no.~1, 417--450.

\bibitem{CoscoZeitouni2024}
C.~Cosco and O.~Zeitouni,
\emph{Moments of partition functions of 2D Gaussian polymers in the weak
disorder regime--II},
Electron. J. Probab. \textbf{29} (2024), paper no.~96.

\bibitem{ErdosTaylor1960}
P.~Erd\H{o}s and S.~J. Taylor,
\emph{Some problems concerning the structure of random walk paths},
Acta Math. Acad. Sci. Hungar. \textbf{11} (1960), 137--162.


\bibitem{Knight1993}
F.~B. Knight,
\emph{Some remarks on mutual windings},
S\'eminaire de Probabilit\'es XXVII,
Lecture Notes in Mathematics, vol.~1557,
Springer, 1993, pp.~36--43.

\bibitem{Knight1994}
F.~B. Knight,
\emph{Erratum to: ``Some remarks on mutual windings''},
S\'eminaire de Probabilit\'es XXVIII,
Lecture Notes in Mathematics, vol.~1583,
Springer, 1994, p.~334.

\bibitem{LiuZygouras2024}
Z.~Liu and N.~Zygouras.
\newblock On the moments of the mass of shrinking balls under the critical 2d stochastic heat flow.
\newblock arXiv:2410.14601, 2024.




\bibitem{LygkonisZygouras2024Multivariate}
D.~Lygkonis and N.~Zygouras,
\emph{A multivariate extension of the Erd\H{o}s--Taylor theorem},
Probab. Theory Related Fields \textbf{189} (2024), 179--227.


\bibitem{DLMF}
F.~W.~J. Olver, A.~B. Olde Daalhuis, D.~W. Lozier,
B.~I. Schneider, R.~F. Boisvert, C.~W. Clark,
B.~R. Miller, B.~V. Saunders, H.~S. Cohl and M.~A. McClain, eds.,
\emph{NIST Digital Library of Mathematical Functions},
National Institute of Standards and Technology,
\url{https://dlmf.nist.gov/}.

\bibitem{PitmanYor1986}
J.~Pitman and M.~Yor,
\emph{Asymptotic laws of planar Brownian motion},
Ann. Probab. \textbf{14} (1986), no.~3, 733--779.

\bibitem{Yor1991}
M.~Yor,
\emph{\'Etude asymptotique des nombres de tours de plusieurs mouvements
browniens complexes corr\'el\'es},
in \emph{Random Walks, Brownian Motion, and Interacting Particle Systems},
Progress in Probability, vol.~28,
Birkh\"auser, 1991, pp.~441--455.

\bibitem{Zygouras2024Review}
N.~Zygouras,
\emph{Directed polymers in a random environment: a review of the phase transitions},
Stochastic Process. Appl. \textbf{177} (2024), paper no.~104431.
\end{thebibliography}
\end{document}